\documentclass{aims}
\usepackage{amsmath}
\usepackage{paralist}
\usepackage[misc]{ifsym}
\usepackage[colorlinks=true]{hyperref}
\hypersetup{urlcolor=blue, citecolor=red}
\allowdisplaybreaks

\def\currentvolume{X}
 \def\currentissue{X}
  \def\currentyear{200X}
   \def\currentmonth{XX}
    \def\ppages{X-XX}
     \def\DOI{10.3934/xx.xxxxxxx}

 \usepackage{amssymb}
 \usepackage{mathrsfs}

 \newcommand{\esssup}[1]{\underset{#1}{\operatorname{ess\,sup} \,}}

 \def\AC{\operatorname{AC}}
 \def\Linf{\operatorname{L}_\infty}
 \def\C{\operatorname{C}}
 \def\rd{\mathrm{d}}
 \def\dist{\operatorname{dist}}
 \def\diam{\operatorname{diam}}
 \def\A2{\boldsymbol{A}}
 \def\L2{\boldsymbol{\operatorname{L}}_2}
 \def\f2{\boldsymbol{f}}
 \def\g2{\boldsymbol{g}}

\newtheorem{theorem}{Theorem}[section]
\newtheorem{corollary}[theorem]{Corollary}

\newtheorem{lemma}[theorem]{Lemma}
\newtheorem{proposition}[theorem]{Proposition}

\theoremstyle{definition}

\newtheorem{remark}[theorem]{Remark}

\title[Value functional and optimal feedback control \dots]
{Value functional and optimal feedback control in linear-quadratic optimal control problem for fractional-order system}

\author[Mikhail I. Gomoyunov]{}



\thanks{This work is supported by RSF grant 19-11-00105, https://rscf.ru/en/project/19-11-00105/}

\begin{document}
\maketitle

\centerline{\scshape Mikhail I. Gomoyunov$^{{\href{mailto:m.i.gomoyunov@gmail.com}{\textrm{\Letter}}}1,2}$}

\medskip

{\footnotesize
 \centerline{$^1$Krasovskii Institute of Mathematics and Mechanics,}

 \centerline{Ural Branch of Russian Academy of Sciences, Russia}
}

\medskip

{\footnotesize
 \centerline{$^2$Ural Federal University, Russia}
}

\bigskip

 \centerline{(Communicated by Handling Editor)}

\begin{abstract}
    In this paper, a finite-horizon optimal control problem involving a dynamical system described by a linear Caputo fractional differential equation and a quadratic cost functional is considered.
    An explicit formula for the value functional is given, which includes a solution of a certain Fredholm integral equation.
    A step-by-step feedback control procedure for constructing $\varepsilon$-optimal controls with any accuracy $\varepsilon > 0$ is proposed.
    The basis for obtaining these results is the study of a solution of the associated Hamilton--Jacobi--Bellman equation with so-called fractional coinvariant derivatives.
\end{abstract}

\section{Introduction}

    Linear-quadratic optimal control problems constitute an important class of problems investigated in optimal control theory.
    They have applications in various fields of knowledge, and, at the same time, they are relatively easy to solve.
    In addition, these problems are often considered as test examples illustrating effectiveness of basic methods of optimal control theory, such as methods originating in the classical calculus of variations, the Pontryagin maximum principle, as well as the Bellman dynamic programming principle, including its infinitesimal version in the form of the Hamilton--Jacobi--Bellman (HJB for short) equation.
    Some details about linear-quadratic optimal control problems can be found in, e.g., \cite{Lee_Markus_1967,Bryson_Ho_1975,Bressan_Piccoli_2007,Yong_2015}.

    This paper is devoted to the development of the theory of linear-quadratic optimal control problems for dynamical systems described by Caputo fractional differential equations.
    Note that a distinctive feature of fractional-order systems is that they possess an inherent memory effect, which makes such systems infinite-dimensional and, in some sense, close to time-delay systems.
    Moreover, the Volterra integral equation corresponding to a Caputo fractional differential equation turns out to be weakly-singular, which introduces certain difficulties in analysis and requires careful attention.
    For the basics of fractional calculus and the theory of fractional differential equations, the reader is referred to, e.g., \cite{Samko_Kilbas_Marichev_1993,Kilbas_Srivastava_Trujillo_2006,Diethelm_2010}.

    More precisely, this paper deals with a finite-horizon optimal control problem involving a dynamical system described by a linear Caputo fractional differential equation of order $\alpha \in (1 / 2, 1)$ and a quadratic cost functional to be minimized.
    In what follows, this problem is referred to as the fractional linear quadratic optimal control problem (FLQOCP for short).
    Such problems were previously considered in, e.g., \cite{Li_Chen_2008,Liang_Wang_Wang_2014,Idczak_Walczak_2016,Zhou_Speyer_2019}.
    Furthermore, some numerical solution methods were proposed in, e.g., \cite{Agrawal_2008,Bhrawy_Doha_Machado_Ezz-Eldien_2015,Baghani_2019,Dabiri_Chahrogh_Machado_2021} (see also the references therein).
    Of particular note is the recent paper \cite{Han_Lin_Yong_2021}, which serves as a motivation for the present study and constitutes its basis.
    In that paper, the FLQOCP (in fact, a more general optimal control problem) was examined both by the variational methods and by the Pontryagin maximum principle.
    In addition, the important question of constructing an optimal feedback control was raised and analyzed.
    As a result, a so-called causal state feedback representation for the open-loop optimal control was obtained, which assumes that a current value of the optimal control is determined from information about the current time, about the history of the system motion up to this time, and also about some auxiliary trajectory.
    However, to the best of our knowledge, there have been no studies devoted to the development of the Bellman dynamic programming principle approach and the application of technique of the corresponding HJB equations for solving the FLQOCP, which is directly related to the construction of optimal feedback controls.
    The present paper aims to contribute in this area.

    Following \cite{Gomoyunov_2020_SIAM}, we introduce a notion of the value functional of the FLQOCP.
    Note that this functional is defined on a space of so-called system positions, which are pairs consisting of a current time and a function treated as a history of the system motion up to this time.
    Then, we associate the FLQOCP to a Cauchy problem for the corresponding HJB equation with so-called fractional coinvariant ($ci$- for short) derivatives of order $\alpha$ and the natural right-end boundary condition.
    Based on the results of \cite{Han_Lin_Yong_2021}, we present a quadratic functional serving as a candidate solution of this Cauchy problem.
    We prove that this functional is $ci$-smooth of order $\alpha$ (i.e., it is continuous and has continuous $ci$-derivatives of order $\alpha$) and check that it does satisfy both the associated HJB equation and the boundary condition.
    This allows us to conclude that this functional is indeed the value functional and propose a step-by-step feedback control procedure for constructing $\varepsilon$-optimal controls with any predetermined accuracy $\varepsilon > 0$.
    Thus, the main contribution of this paper is that an explicit formula for the value functional of the FLQOCP is given, continuity and smoothness properties of this functional are studied, and practically realizable feedback control procedures generating $\varepsilon$-optimal controls are obtained.

    Before proceeding with the main body of the paper, we make some observations.

    (i)
        Since we want to apply the technique of fractional $ci$-differentiability of order $\alpha$ from \cite{Gomoyunov_2020_SIAM}, we must confine ourselves to (Lebesgue) measurable and essentially bounded controls.
        That is why we formulate the FLQOCP in this class of admissible open-loop controls from the very beginning, which differs from the standard formulation of linear-quadratic optimal control problems in the class of square integrable open-loop controls.
        In the classical case of dynamical systems described by ordinary differential equations, these two formulations are known to be essentially the same, but they differ in the case of fractional-order systems.
        In particular, although it can be verified that the optimal results in these two formulations coincide, it can happen that an optimal square integrable open-loop control becomes unbounded in a neighbourhood of the terminal time, and, consequently, an optimal essentially bounded open-loop control does not exist.
        This fact is also confirmed by the presented in this paper formula for the $ci$-gradient of order $\alpha$ of the value functional, which shows that, in general, this gradient has a singularity in the time variable at the terminal time.
        Note also that this fact does not allow us to directly use the main results of \cite{Gomoyunov_2020_SIAM} in this paper, because they were obtained under compact geometric constraints on control.

    (ii)
        In accordance with item (i), this paper deals with $\varepsilon$-optimal controls only.
        To generate the desired controls, we follow \cite{Gomoyunov_2020_SIAM} and suggest to use time-discrete recursive feedback control procedures, which go back to the theory of positional differential games \cite{Krasovskii_Subbotin_1988,Krasovskii_Krasovskii_1995}.
        This approach is reasonable from a practical viewpoint and ensures the stability of the control procedures with respect to computational and measurement errors.
        Moreover, this allows us to avoid the questions of existence and uniqueness of a solution of the Caputo fractional differential equation corresponding to a closed-loop system.

    (iii)
        The most difficult part of the paper is to prove that the proposed candidate solution of the associated Cauchy problem is $ci$-smooth of order $\alpha$.
        To this end, we suggest to elaborate a special technique for $ci$-differentiating of order $\alpha$ of quadratic functionals of a sufficiently general form.
        It seems that this technique is of independent interest and may find further application, for example, in the study of linear-quadratic differential games for fractional-order systems.

    (iv)
        The obtained formula for the value functional of the FLQOCP as well as the formulas used for the construction of $\varepsilon$-optimal controls include a solution of a certain Fredholm integral equation \cite[Section 5]{Han_Lin_Yong_2021}, which can be considered as an analog of the Ricatti equation (in this connection see, e.g., \cite[Section 6.7]{Yong_2015}).
        Note that, for the results of the present paper, a more detailed analysis of the properties of the solution of this Fredholm integral equation in comparison with \cite{Han_Lin_Yong_2021} is required.
        In particular, we establish some continuity and differentiability properties of this solution.

    The paper is organized as follows.
    After introducing some notation in Section \ref{section_notation}, we formulate the FLQOCP in Section \ref{section_FLQOCP}.
    In Section \ref{section_HJB}, we consider the associated Cauchy problem for the HJB equation and the right-end boundary condition and give an explicit formula for its solution.
    In Section \ref{section_OFC}, we prove that this solution is indeed the value functional and describe the feedback control procedure for constructing $\varepsilon$-optimal controls.
    Section \ref{section_conclusion} contains some conclusions.
    For the reader's convenience, some of the proofs are relegated to Appendices \ref{appendix}, \ref{appendix_2}, and \ref{appendix_3}.

\section{Notation}
\label{section_notation}
    The basic notation used in the paper is standard.
    For any $n$, $m \in \mathbb{N}$, let $\mathbb{R}^n$ and $\mathbb{R}^{n \times m}$ be the linear spaces of $n$-dimensional vectors and $(n \times m)$-matrices, respectively.
    We denote by $\langle \cdot, \cdot \rangle_{\mathbb{R}^n}$, $\|\cdot\|_{\mathbb{R}^n}$, and $\|\cdot\|_{\mathbb{R}^{n \times m}}$ the inner product in $\mathbb{R}^n$, the Euclidean norm in $\mathbb{R}^n$, and the corresponding (operator) norm in $\mathbb{R}^{n \times m}$, respectively.
    For a matrix $A \in \mathbb{R}^{n \times m}$, its transpose is denoted by $A^\top \in \mathbb{R}^{m \times n}$.
    If a matrix $A \in \mathbb{R}^{n \times n}$ is invertible, its inverse is denoted by $A^{- 1} \in \mathbb{R}^{n \times n}$.
    We denote by $\mathbb{O}_n \in \mathbb{R}^n$ and $\mathbb{O}_{n \times n}$, $\mathbb{I}_{n \times n} \in \mathbb{R}^{n \times n}$ the null vector, the null matrix, and the identity matrix, respectively.
    For any $a$, $b \in \mathbb{R}$, we put $a \vee b \doteq \max\{a, b\}$ and $a \wedge b \doteq \min\{a, b\}$.
    By a modulus of continuity, we mean a continuous and nondecreasing function $\omega \colon [0, \infty) \to [0, \infty)$ such that $\omega(0) = 0$.
    A modulus of continuity $\omega^\ast(\cdot)$ is called infinitesimal if $\omega^\ast(\delta) / \delta \to 0$ as $\delta \to 0^+$.

    For a number $T \geq 0$, we denote by $\C([0, T], \mathbb{R}^{n \times m})$ the Banach space of all continuous functions $X \colon [0, T] \to \mathbb{R}^{n \times m}$ endowed with the norm
    \begin{equation*}
        \|X(\cdot)\|_{\C([0, T], \mathbb{R}^{n \times m})}
        \doteq \max_{\tau \in [0, T]} \|X(\tau)\|_{\mathbb{R}^{n \times m}}
        \quad \forall X(\cdot) \in \C([0, T], \mathbb{R}^{n \times m}).
    \end{equation*}
    Similarly, we consider the Banach space $\C([0, T], \mathbb{R}^n)$ with the norm $\|\cdot\|_{\C([0, T], \mathbb{R}^n)}$.
    Further, let $\Linf([0, T], \mathbb{R}^n)$ be the linear space of all (Lebesgue) measurable functions $f \colon [0, T] \to \mathbb{R}^n$ such that
    \begin{equation*}
        \esssup{\tau \in [0, T]} \|f(\tau)\|_{\mathbb{R}^n}
        < \infty.
    \end{equation*}

    Let a number $\alpha \in (0, 1)$ be given.
    Following, e.g., \cite[Definition 2.3]{Samko_Kilbas_Marichev_1993}, we introduce the linear space $\AC^\alpha([0, T], \mathbb{R}^n)$ of all functions $x \colon [0, T] \to \mathbb{R}^n$ each of which can be represented in the form
    \begin{equation} \label{x_f}
        x(\tau)
        = x(0) + \frac{1}{\Gamma(\alpha)} \int_{0}^{\tau} \frac{f(\xi)}{(\tau - \xi)^{1 - \alpha}} \, \rd \xi
        \quad \forall \tau \in [0, T]
    \end{equation}
    for some function $f(\cdot) \in \Linf([0, T], \mathbb{R}^n)$.
    In the right-hand side of equality \eqref{x_f}, the second term is the {\it Riemann--Liouville fractional integral of order $\alpha$} of the function $f(\cdot)$ (see, e.g., \cite[Definition 2.1]{Samko_Kilbas_Marichev_1993}) and $\Gamma(\cdot)$ is the gamma-function.
    Note that $\AC^\alpha([0, T], \mathbb{R}^n) \subset \C([0, T], \mathbb{R}^n)$ (see, e.g., \cite[Remark 3.3]{Samko_Kilbas_Marichev_1993}), and, respectively, the space $\AC^\alpha([0, T], \mathbb{R}^n)$ is also endowed with the norm $\|\cdot\|_{\C([0, T], \mathbb{R}^n)}$.

    According to, e.g., \cite[Theorem 2.4]{Samko_Kilbas_Marichev_1993}, every function $x(\cdot) \in \AC^\alpha([0, T], \mathbb{R}^n)$ has at almost every (a.e.) $\tau \in [0, T]$ a {\it Caputo fractional derivative of order $\alpha$}, which is defined by (see, e.g., \cite[Section 2.4]{Kilbas_Srivastava_Trujillo_2006} and \cite[Chapter 3]{Diethelm_2010})
    \begin{equation} \label{Caputo}
        (^C D^\alpha x)(\tau)
        \doteq \frac{1}{\Gamma(1 - \alpha)} \frac{\rd}{\rd \tau} \int_{0}^{\tau} \frac{x(\xi) - x(0)}{(\tau - \xi)^\alpha} \, \rd \xi.
    \end{equation}
    Moreover, if representation \eqref{x_f} is valid for some function $f(\cdot) \in \Linf([0, T], \mathbb{R}^n)$, then $(^C D^\alpha x)(\tau) = f(\tau)$ for a.e. $\tau \in [0, T]$.
    In particular, we have
    \begin{equation} \label{I_D}
        x(\tau)
        = x(0) + \frac{1}{\Gamma(\alpha)} \int_{0}^{\tau} \frac{(^C D^\alpha x)(\xi)}{(\tau - \xi)^{1 - \alpha}} \, \rd \xi
        \quad \forall \tau \in [0, T].
    \end{equation}

    Further, let us consider the sets
    \begin{equation} \label{G}
        \begin{split}
            \mathcal{G}
            & \doteq \bigcup_{t \in [0, T]} \bigl( \{ t \} \times \AC^\alpha([0, t], \mathbb{R}^n) \bigr) \\
            & = \bigl\{ (t, w(\cdot)) \colon
            \, t \in [0, T], \, w(\cdot) \in \AC^\alpha([0, t], \mathbb{R}^n) \bigr\}
        \end{split}
    \end{equation}
    and
    \begin{equation} \label{G_0}
        \mathcal{G}_0
        \doteq \bigl\{ (t, w(\cdot)) \in \mathcal{G}
        \colon t < T \bigr\}.
    \end{equation}
    Note that, in the case $t = 0$, the space $\AC^\alpha([0, t], \mathbb{R}^n)$ can be identified with the space $\mathbb{R}^n$.
    The set $\mathcal{G}$ (and, respectively, its subset $\mathcal{G}_0$) is endowed with the metric
    \begin{equation} \label{dist}
        \dist \bigl( (t^\prime, w^\prime(\cdot)), (t, w(\cdot)) \bigr)
        \doteq |t^\prime - t| + \max_{\tau \in [0, T]} \|w^\prime(\tau \wedge t^\prime) - w(\tau \wedge t)\|_{\mathbb{R}^n}
    \end{equation}
    for all $(t, w(\cdot))$, $(t^\prime, w^\prime(\cdot)) \in \mathcal{G}$.
    Note that, for any $x(\cdot) \in \AC^\alpha([0, T], \mathbb{R}^n)$ and any $t \in [0, T]$, we have $(t, x_t(\cdot)) \in \mathcal{G}$, where we denote by $x_t(\cdot)$ the restriction of the function $x(\cdot)$ to the interval $[0, t]$:
    \begin{equation} \label{x_t}
        x_t(\tau)
        \doteq x(\tau)
        \quad \forall \tau \in [0, t].
    \end{equation}
    We note also that the mapping $[0, T] \times \AC^\alpha([0, T], \mathbb{R}^n) \ni (t, x(\cdot)) \mapsto (t, x_t(\cdot)) \in \mathcal{G}$ is continuous.
    Finally, for every pair $(t, w(\cdot)) \in \mathcal{G}$, we introduce the set of admissible extensions of the function $w(\cdot)$ to the right up to $T$:
    \begin{equation} \label{X_t_w}
        \mathcal{X}(t, w(\cdot))
        \doteq \bigl\{ x(\cdot) \in \AC^\alpha([0, T], \mathbb{R}^n) \colon
        \, x_t(\cdot) = w(\cdot) \bigr\}.
    \end{equation}

\section{Fractional linear-quadratic optimal control problem (FLQOCP)}
\label{section_FLQOCP}

    Let numbers $n$, $m \in \mathbb{N}$, $T > 0$, and $\alpha \in (1 / 2, 1)$ and functions $A(\cdot) \in \C([0, T], \mathbb{R}^{n \times n})$ and $B(\cdot) \in \C([0, T], \mathbb{R}^{n \times m})$ be given.
    We consider a {\it dynamical system} described by the linear fractional differential equation
    \begin{equation} \label{differential_equation}
        (^C D^\alpha x)(\tau)
        = A(\tau) x(\tau) + B(\tau) u(\tau).
    \end{equation}
    In the above, $\tau \in [0, T]$ is time, $x(\tau) \in \mathbb{R}^n$ is a current state, $(^C D^\alpha x)(\tau)$ is the Caputo fractional derivative of order $\alpha$ (see \eqref{Caputo}), $u(\tau) \in \mathbb{R}^m$ is a current control.

    Let a pair $(t, w(\cdot)) \in \mathcal{G}$ be given (see \eqref{G}), where $t$ is treated as an initial time and $w(\cdot)$ is treated as a history of a system motion on the interval $[0, t]$.
    According to, e.g., \cite{Gomoyunov_2020_SIAM}, the pair $(t, w(\cdot))$ plays a role of a {\it position of system} \eqref{differential_equation}.
    By an {\it admissible control} on the interval $[t, T]$, we mean any function $u(\cdot) \in \Linf([t, T], \mathbb{R}^m)$.
    A {\it motion} of system \eqref{differential_equation} generated from $(t, w(\cdot))$ by a control $u(\cdot) \in \Linf([t, T], \mathbb{R}^m)$ is defined as a function $x(\cdot) \in \mathcal{X}(t, w(\cdot))$ (see \eqref{X_t_w}) that together with $u(\cdot)$ satisfies the fractional differential equation \eqref{differential_equation} for a.e. $\tau \in [t, T]$.
    By \cite[Proposition 2]{Gomoyunov_2020_DGA}, such a function $x(\cdot)$ exists and is unique, and we denote it by $x(\cdot) \doteq x(\cdot \mid t, w(\cdot), u(\cdot))$.
    Note that $x(\cdot)$ is a unique function from $\C([0, T], \mathbb{R}^n)$ satisfying the initial condition (see \eqref{x_t})
    \begin{equation} \label{initial_condition}
        x_t(\cdot)
        = w(\cdot)
    \end{equation}
    and the following {\it weakly-singular Volterra integral equation}:
    \begin{equation*}
        x(\tau)
        = a(\tau \mid t, w(\cdot))
        + \frac{1}{\Gamma(\alpha)} \int_{t}^{\tau} \frac{A(\xi) x(\xi) + B(\xi) u(\xi)}{(\tau - \xi)^{1 - \alpha}} \, \rd \xi
        \quad \forall \tau \in [t, T],
    \end{equation*}
    where we use the notation
    \begin{equation} \label{a_definition}
        a(\tau \mid t, w(\cdot))
        \doteq \begin{cases}
            w(\tau) & \mbox{if } \tau \in [0, t], \\
            \displaystyle
            w(0) + \frac{1}{\Gamma(\alpha)} \int_{0}^{t} \frac{(^C D^\alpha w)(\xi)}{(\tau - \xi)^{1 - \alpha}} \, \rd \xi
            & \mbox{if } \tau \in (t, T].
          \end{cases}
    \end{equation}
    Note also that, in the particular case when $t = 0$, the initial condition \eqref{initial_condition} specified by the inclusion $x(\cdot) \in \mathcal{X}(t, w(\cdot))$ takes the usual form $x(0) = x_0$ with $x_0 \doteq w(0)$.

    In addition, let a matrix $P \in \mathbb{R}^{n \times n}$ and functions $Q(\cdot) \in \C([0, T], \mathbb{R}^{n \times n})$ and $R(\cdot) \in \C([0, T], \mathbb{R}^{m \times m})$ be given.
    It is assumed that the matrices $P$ and $Q(\tau)$ for all $\tau \in [0, T]$ are symmetric positive semi-definite, the matrices $R(\tau)$ for all $\tau \in [0, T]$ are symmetric positive definite, and there exists a number $\theta > 0$ such that
    \begin{equation} \label{theta}
        \langle u, R(\tau) u \rangle_{\mathbb{R}^m}
        \geq \theta \|u\|_{\mathbb{R}^m}^2
        \quad \forall u \in \mathbb{R}^m \quad \forall \tau \in [0, T].
    \end{equation}
    We consider the quadratic {\it cost functional}
    \begin{equation} \label{cost_functional}
        \begin{split}
            & J(t, w(\cdot), u(\cdot)) \\
            & \doteq \langle x(T), P x(T) \rangle_{\mathbb{R}^n}
            + \int_{t}^{T} \bigl( \langle x(\tau), Q(\tau) x(\tau) \rangle_{\mathbb{R}^n}
            + \langle u(\tau), R(\tau) u(\tau) \rangle_{\mathbb{R}^m} \bigr) \, \rd \tau
        \end{split}
    \end{equation}
    for all $(t, w(\cdot)) \in \mathcal{G}$ and all $u(\cdot) \in \Linf([t, T], \mathbb{R}^m)$, where $x(\cdot) \doteq x(\cdot \mid t, w(\cdot), u(\cdot))$.

    For a pair $(t, w(\cdot)) \in \mathcal{G}$, the {\it FLQOCP} is to find a control $u(\cdot) \in \Linf([t, T], \mathbb{R}^m)$ that minimizes the cost functional \eqref{cost_functional}.
    Hence, the value of the {\it optimal result} is
    \begin{equation} \label{value}
        \rho(t, w(\cdot))
        \doteq \inf_{u(\cdot) \in \Linf([t, T], \mathbb{R}^m)} J(t, w(\cdot), u(\cdot)),
    \end{equation}
    and, for a number $\varepsilon > 0$, a control $u(\cdot) \in \Linf([t, T], \mathbb{R}^m)$ is called {\it $\varepsilon$-optimal} if
    \begin{equation*}
        J(t, w(\cdot), u(\cdot))
        \leq \rho(t, w(\cdot)) + \varepsilon.
    \end{equation*}

    The main goal of the present paper is to give an explicit formula for the value functional $\mathcal{G} \ni (t, w(\cdot)) \mapsto \rho(t, w(\cdot)) \in \mathbb{R}$ of the FLQOCP and to propose a step-by-step feedback control procedure for constructing $\varepsilon$-optimal controls with any accuracy $\varepsilon > 0$.
    In order to achieve this goal, we will use the technique of the corresponding HJB equations from \cite{Gomoyunov_2020_SIAM}.

    \begin{remark}
        In the case when $Q(\tau) = \mathbb{O}_{n \times n}$ for all $\tau \in [0, T]$, by the scheme from \cite{Gomoyunov_2020_ACS}, the FLQOCP can be reduced to an auxiliary linear-quadratic optimal control problem for a dynamical system described by ordinary differential equations, and, therefore, no special analysis is required.
    \end{remark}

\section{Hamilton--Jacobi--Bellman (HJB) equation and its solution}
\label{section_HJB}

    In this section, following \cite{Gomoyunov_2020_SIAM}, we associate the FLQOCP to a Cauchy problem for the corresponding HJB equation with fractional coinvariant derivatives and the natural right-end boundary condition.
    We give an explicit formula for a solution of this Cauchy problem, which, in particular, agrees with the results of \cite{Han_Lin_Yong_2021}.

    \subsection{Fractional coinvariant ($ci$-) derivatives}

        In accordance with, e.g., \cite{Kim_1999,Gomoyunov_2020_SIAM} (see also \cite{Gomoyunov_Lukoyanov_Plaksin_2021}), we say that a functional $\varphi \colon \mathcal{G} \to \mathbb{R}$ is {\it $ci$-differentiable of order $\alpha$} at a point $(t, w(\cdot)) \in \mathcal{G}_0$ (see \eqref{G_0}) if there exist $\partial_t^\alpha \varphi (t, w(\cdot)) \in \mathbb{R}$ and $\nabla^\alpha \varphi (t, w(\cdot)) \in \mathbb{R}^n$ such that, for any $x(\cdot) \in \mathcal{X}(t, w(\cdot))$ (see \eqref{X_t_w}) and any $t^\prime \in (t, T)$,
        \begin{equation} \label{ci-differentiability}
            \begin{split}
                & \varphi(t^\prime, x_{t^\prime}(\cdot)) - \varphi(t, w(\cdot)) \\
                & = \partial_t^\alpha \varphi(t, w(\cdot)) (t^\prime - t)
                + \biggl\langle \nabla^\alpha \varphi(t, w(\cdot)),
                \int_{t}^{t^\prime} (^C D^\alpha x)(\xi) \, \rd \xi \biggr\rangle_{\mathbb{R}^n}
                + o(t^\prime - t),
            \end{split}
        \end{equation}
        where $x_{t^\prime}(\cdot)$ denotes the restriction of the function $x(\cdot)$ to the interval $[0, t^\prime]$ (see \eqref{x_t}) and the function $o \colon (0, \infty) \to \mathbb{R}$, which may depend on $t$ and $x(\cdot)$, satisfies the condition $o(\delta) / \delta \to 0$ as $\delta \to 0^+$.
        In this case, the quantities $\partial_t^\alpha \varphi (t, w(\cdot))$ and $\nabla^\alpha \varphi(t, w(\cdot))$ are called a {\it $ci$-derivative of order $\alpha$ in $t$} and a {\it $ci$-gradient of order $\alpha$} of the functional $\varphi$ at the point $(t, w(\cdot))$, respectively.

        In addition, we say that a functional $\varphi \colon \mathcal{G} \to \mathbb{R}$ is {\it $ci$-smooth of order $\alpha$} if it is continuous, $ci$-differentiable of order $\alpha$ at every point $(t, w(\cdot)) \in \mathcal{G}_0$, and the mappings $\partial_t^\alpha \varphi \colon \mathcal{G}_0 \to \mathbb{R}$ and $\nabla^\alpha \varphi \colon \mathcal{G}_0 \to \mathbb{R}^n$ are continuous.

    \subsection{Cauchy problem for HJB equation}

        Taking \eqref{differential_equation} and \eqref{cost_functional} into account, let us denote
        \begin{equation} \label{pre_Hamiltonian}
            h(t, x, u, s)
            \doteq \langle s, A(t) x + B(t) u \rangle_{\mathbb{R}^n}
            + \langle x, Q(t) x \rangle_{\mathbb{R}^n}
            + \langle u, R(t) u \rangle_{\mathbb{R}^m}
        \end{equation}
        for all $t \in [0, T]$, all $x$, $s \in \mathbb{R}^n$, and all $u \in \mathbb{R}^m$ and consider the {\it Hamiltonian}
        \begin{equation} \label{Hamiltonian}
            \begin{split}
                H(t, x, s)
                & \doteq \inf_{u \in \mathbb{R}^m} h(t, x, u, s) \\
                & = \langle s, A(t) x \rangle_{\mathbb{R}^n} + \langle x, Q(t) x \rangle_{\mathbb{R}^n}
                - \frac{1}{4} \langle B(t)^\top s, R(t)^{- 1} B(t)^\top s \rangle_{\mathbb{R}^m}
            \end{split}
        \end{equation}
        for all $t \in [0, T]$ and all $x$, $s \in \mathbb{R}^n$.
        Then, according to \cite{Gomoyunov_2020_SIAM}, the FLQOCP can be associated to the {\it Cauchy problem} for the {\it HJB equation}
        \begin{equation} \label{HJB}
            \partial_t^\alpha \varphi(t, w(\cdot)) + H \bigl( t, w(t), \nabla^\alpha \varphi(t, w(\cdot)) \bigr)
            = 0
            \quad \forall (t, w(\cdot)) \in \mathcal{G}_0
        \end{equation}
        under the right-end {\it boundary condition}
        \begin{equation} \label{boundary_condition}
            \varphi(T, w(\cdot))
            = \langle w(T), P w(T) \rangle_{\mathbb{R}^n}
            \quad \forall w(\cdot) \in \AC^\alpha([0, T], \mathbb{R}^n).
        \end{equation}
        By a {\it solution} of this problem, we mean a $ci$-smooth of order $\alpha$ functional $\varphi \colon \mathcal{G} \to \mathbb{R}$ that satisfies both the HJB equation \eqref{HJB} and the boundary condition \eqref{boundary_condition}.

        An explicit formula for a solution of the Cauchy problem \eqref{HJB} and \eqref{boundary_condition}, given in Section \ref{subsection_formulation_of_result} below, includes a solution of a certain Fredholm integral equation (see, e.g., \cite[Section 5]{Han_Lin_Yong_2021}), which is the subject of the next Section \ref{subsection_preliminaries}.

    \subsection{Fredholm integral equation}
    \label{subsection_preliminaries}

        Let us denote
        \begin{equation} \label{Omega}
            \Omega
            \doteq \bigl\{ (\tau, \xi) \in [0, T] \times [0, T] \colon \tau \geq \xi \bigr\}.
        \end{equation}
        Following, e.g., \cite[Section 4 and Proposition 4.3]{Gomoyunov_2019_FCAA_2} (see also, e.g., \cite{Bourdin_2018}), let us consider the {\it fundamental solution matrix} of the fractional differential equation \eqref{differential_equation}, which is a continuous function $\Phi \colon \Omega \to \mathbb{R}^{n \times n}$ such that, for every $\tau \in [0, T]$, the function $[0, \tau] \ni \xi \mapsto \Phi(\tau, \xi) \in \mathbb{R}^{n \times n}$ satisfies the integral equation
        \begin{equation} \label{fundamental_solution_matrix}
            \Phi(\tau, \xi)
            = \frac{\mathbb{I}_{n \times n}}{\Gamma(\alpha)}
            + \frac{(\tau - \xi)^{1 - \alpha}}{\Gamma(\alpha)} \int_{\xi}^{\tau} \frac{\Phi(\tau, \eta) A(\eta)}
            {(\tau - \eta)^{1 - \alpha} (\eta - \xi)^{1 - \alpha}} \, \rd \eta
            \quad \forall \xi \in [0, \tau].
        \end{equation}

        Let us introduce a function $K \colon [0, T] \times [0, T] \to \mathbb{R}^{m \times m}$ by
        \begin{equation} \label{K_definition}
            \begin{split}
                K(\tau, \xi)
                & \doteq B(\tau)^\top \Phi(T, \tau)^\top P \Phi(T, \xi) B(\xi) \\
                & \quad + (T - \tau)^{1 - \alpha} (T - \xi)^{1 - \alpha} B(\tau)^\top \int_{\tau \vee \xi}^{T} \frac{\Phi(\eta, \tau)^\top Q(\eta) \Phi(\eta, \xi)}
                {(\eta - \tau)^{1 - \alpha} (\eta - \xi)^{1 - \alpha}} \, \rd \eta \, B(\xi)
            \end{split}
        \end{equation}
        for all $\tau$, $\xi \in [0, T]$.

        \begin{proposition} \label{proposition_K}
            The function $K(\cdot, \cdot)$ is continuous and the equality below holds:
            \begin{equation} \label{K_symmetric}
                K(\tau, \xi)^\top
                = K(\xi, \tau)
                \quad \forall \tau, \xi \in [0, T].
            \end{equation}
        \end{proposition}

        The proof is relegated to Appendix \ref{appendix}.

        Further, let us denote
        \begin{equation} \label{Theta}
            \Theta
            \doteq \big\{ (\tau, \xi, t) \in [0, T] \times [0, T] \times [0, T] \colon
            \tau \wedge \xi \geq t \bigr\}.
        \end{equation}
        We have
        \begin{proposition} \label{proposition_M}
            There exists a unique continuous function
            \begin{equation*}
                \Theta \ni (\tau, \xi, t) \mapsto M(\tau, \xi \mid t) \in \mathbb{R}^{m \times m}
            \end{equation*}
            such that, for any $t \in [0, T]$ and any $\xi \in [t, T]$, the function $[t, T] \ni \tau \mapsto M(\tau, \xi \mid t)$ is a unique continuous solution of the {\rm Fredholm integral equation}
            \begin{equation} \label{M_definition}
                R(\tau) M(\tau, \xi \mid t) + \int_{t}^{T} \frac{K(\tau, \eta) M(\eta, \xi \mid t)}{(T - \eta)^{2 - 2 \alpha}} \, \rd \eta
                = - K(\tau, \xi) R(\xi)^{- 1}
                \quad \forall \tau \in [t, T].
            \end{equation}
            In addition, the function $M(\cdot, \cdot \mid \cdot)$ possesses the following two properties:

            {\rm (i)}
                For any $t \in [0, T)$ and any $\vartheta \in (t, T)$, there exists an infinitesimal modulus of continuity $\omega_{M}^\ast(\cdot)$ such that, for any $t^\prime \in (t, \vartheta]$ and any $\tau$, $\xi \in [t^\prime, T]$,
                \begin{equation*}
                    \biggl\| M(\tau, \xi \mid t^\prime) - M(\tau, \xi \mid t)
                    + \frac{(t^\prime - t) M(\tau, t \mid t) R(t) M(t, \xi \mid t)}{(T - t)^{2 - 2 \alpha}} \biggr\|_{\mathbb{R}^{m \times m}}
                    \leq \omega_{M}^\ast(t^\prime - t).
                \end{equation*}

            {\rm (ii)}
                The equality below is valid:
                \begin{equation} \label{M_symmetric}
                    M(\tau, \xi \mid t)^\top
                    = M(\xi, \tau \mid t)
                    \quad \forall (\tau, \xi, t) \in \Theta.
                \end{equation}
        \end{proposition}

        The proof is relegated to Appendix \ref{appendix}.

    \subsection{Solution of Cauchy problem}
    \label{subsection_formulation_of_result}

        For every pair $(t, w(\cdot)) \in \mathcal{G}$, taking the function $a(\cdot \mid t, w(\cdot))$ from \eqref{a_definition}, let us define the auxiliary functions
        \begin{equation} \label{b_definition}
            b(\tau \mid t, w(\cdot))
            \doteq a(\tau \mid t, w(\cdot))
            + \int_{t}^{\tau} \frac{\Phi(\tau, \xi) A(\xi) a(\xi \mid t, w(\cdot))}{(\tau - \xi)^{1 - \alpha}} \, \rd \xi
        \end{equation}
        and
        \begin{equation} \label{c_definition}
            \begin{split}
                c(\tau \mid t, w(\cdot))
                & \doteq B(\tau)^\top \Phi(T, \tau)^\top P b(T \mid t, w(\cdot)) \\
                & \quad + (T - \tau)^{1 - \alpha} B(\tau)^\top \int_{\tau}^{T} \frac{\Phi(\xi, \tau)^\top Q(\xi) b(\xi \mid t, w(\cdot))}
                {(\xi - \tau)^{1 - \alpha}} \, \rd \xi
            \end{split}
        \end{equation}
        for all $\tau \in [t, T]$.
        Let us note that, for any $t \in [0, T]$ and any $\tau \in [t, T]$, the mapping $\AC^\alpha ([0, t], \mathbb{R}^n) \ni w(\cdot) \mapsto (a(\tau \mid t, w(\cdot)), b(\tau \mid t, w(\cdot)), c(\tau \mid t, w(\cdot)))$ is linear.
        Then, let us consider a quadratic functional $\varphi \colon \mathcal{G} \to \mathbb{R}$ given by
        \begin{equation} \label{varphi_main}
            \begin{split}
                \hspace*{-0.1em} \varphi(t, w(\cdot))
                & \doteq \langle b(T), P b(T) \rangle_{\mathbb{R}^n}
                + \int_{t}^{T} \langle b(\tau), Q(\tau) b(\tau) \rangle_{\mathbb{R}^n} \, \rd \tau \\
                & \quad - \int_{t}^{T} \biggl\langle \frac{c(\tau)}{(T - \tau)^{2 - 2 \alpha}},
                R(\tau)^{- 1} c(\tau)
                + \int_{t}^{T} \frac{M(\tau, \xi \mid t) c(\xi)}{(T - \xi)^{2 - 2 \alpha}} \, \rd \xi \biggr \rangle_{\mathbb{R}^m} \, \rd \tau
            \end{split}
        \end{equation}
        for all $(t, w(\cdot)) \in \mathcal{G}$, where the function $M(\cdot, \cdot \mid \cdot)$ is a solution of the Fredholm integral equation \eqref{M_definition} from Proposition \ref{proposition_M} and, for brevity, the shorthand notation $b(\cdot) \doteq b(\cdot \mid t, w(\cdot))$ and $c(\cdot) \doteq c(\cdot \mid t, w(\cdot))$ is used.

        The main result of the first part of the paper is
        \begin{theorem} \label{theorem_solution}
            The functional $\varphi$ given by \eqref{varphi_main} is a solution of the Cauchy problem for the HJB equation \eqref{HJB} and the boundary condition \eqref{boundary_condition}.
            In addition, the equality
            \begin{equation} \label{nabla_varphi_main}
                B(t)^\top \nabla^\alpha \varphi (t, w(\cdot))
                = \frac{2 R(t)}{(T - t)^{1 - \alpha}}
                \biggl( R(t)^{- 1} c(t)
                + \int_{t}^{T} \frac{M(t, \xi \mid t) c(\xi)}{(T - \xi)^{2 - 2 \alpha}} \, \rd \xi \biggr)
            \end{equation}
            holds for all $(t, w(\cdot)) \in \mathcal{G}_0$, where $\nabla^\alpha \varphi (t, w(\cdot))$ is the $ci$-gradient of order $\alpha$ of the functional $\varphi$ at the point $(t, w(\cdot))$ and $c(\cdot) \doteq c(\cdot \mid t, w(\cdot))$.
        \end{theorem}

        In the next Section \ref{subsection_quadratic_functionals}, we develop a special technique for $ci$-differentiating of order $\alpha$ of quadratic functionals of a rather general form.
        This allows us to simplify the proof of Theorem \ref{theorem_solution}, which is given in Section \ref{subsection_proof_theorem_solution} below.

        Note that, From the proof of Theorem \ref{theorem_solution}, explicit formulas for $ci$-derivatives of order $\alpha$ of the functional $\varphi$ given by \eqref{varphi_main} can be derived.
        Nevertheless, we do not include these formulas in the formulation of the theorem since they are rather cumbersome and are not important for further results of the paper.

    \subsection{Quadratic functionals of general form}
    \label{subsection_quadratic_functionals}

        Within this section, we fix a number $k \in \mathbb{N}$, keeping in mind that $k = n$ or $k = m$.

        Let us introduce an auxiliary notion.
        Suppose that, for every point $(t, w(\cdot)) \in \mathcal{G}$, a function
        \begin{equation} \label{v}
            [t, T] \ni \tau \mapsto v(\tau \mid t, w(\cdot)) \in \mathbb{R}^k
        \end{equation}
        is given (for example, the restriction of the function $a(\cdot \mid t, w(\cdot))$ from \eqref{a_definition} to $[t, T]$, the function $b(\cdot \mid t, w(\cdot))$ from \eqref{b_definition}, or the function $c(\cdot \mid t, w(\cdot))$ from \eqref{c_definition}).
        In what follows, it is convenient to refer to $v(\cdot \mid \cdot, \cdot)$ as a {\it family of functions}.
        Then, we say that the {\it family of functions $v(\cdot \mid \cdot, \cdot)$ satisfies Condition $(\star)$ with parameters $(\beta, \gamma) \in [0, \infty) \times [0, 1)$} if, for every point $(t, w(\cdot)) \in \mathcal{G}$, there exist functions
        \begin{equation} \label{v_derivatives}
            [t, T] \ni \tau \mapsto \partial_t^\alpha v(\tau \mid t, w(\cdot)) \in \mathbb{R}^k,
            \quad [t, T] \ni \tau \mapsto \nabla^\alpha v(\tau \mid t, w(\cdot)) \in \mathbb{R}^{k \times n}
        \end{equation}
        such that the following three conditions are satisfied:

        (i$_\star$)
            For every compact set $\mathcal{G}_\ast \subset \mathcal{G}$, functions \eqref{v} and \eqref{v_derivatives} for $(t, w(\cdot)) \in \mathcal{G}_\ast$ are uniformly bounded and equicontinuous, i.e., there are a number $\mu_v \geq 0$ and a modulus of continuity $\omega_v^{(1)}(\cdot)$ such that, for any $(t, w(\cdot)) \in \mathcal{G}_\ast$ and any $\tau$, $\tau^\prime \in [t, T]$,
            \begin{equation*}
                \| v(\tau \mid t, w(\cdot))\|_{\mathbb{R}^k}
                \vee \|\partial_t^\alpha v(\tau \mid t, w(\cdot))\|_{\mathbb{R}^k}
                \vee \|\nabla^\alpha v(\tau \mid t, w(\cdot))\|_{\mathbb{R}^{k \times n}}
                \leq \mu_v
            \end{equation*}
            and
            \begin{equation*}
                \begin{split}
                    & \|v(\tau^\prime \mid t, w(\cdot)) - v(\tau \mid t, w(\cdot))\|_{\mathbb{R}^k}
                    \vee \|\partial_t^\alpha v(\tau^\prime \mid t, w(\cdot)) - \partial_t^\alpha v(\tau \mid t, w(\cdot))\|_{\mathbb{R}^k} \\
                    & \vee \|\nabla^\alpha v(\tau^\prime \mid t, w(\cdot)) - \nabla^\alpha v(\tau \mid t, w(\cdot))\|_{\mathbb{R}^{k \times n}}
                    \leq \omega_v^{(1)}(|\tau^\prime - \tau|).
                \end{split}
            \end{equation*}

        (ii$_\star$)
            For every compact set $\mathcal{G}_\ast \subset \mathcal{G}$, there exists a modulus of continuity $\omega_v^{(2)}(\cdot)$ such that, for any $(t, w(\cdot))$, $(t^\prime, w^\prime(\cdot)) \in \mathcal{G}_\ast$ and any $\tau \in [t \vee t^\prime, T]$,
            \begin{align*}
                & \|v(\tau \mid t^\prime, w^\prime(\cdot)) - v(\tau \mid t, w(\cdot))\|_{\mathbb{R}^k}
                \vee \|\partial_t^\alpha v(\tau \mid t^\prime, w^\prime(\cdot))
                - \partial_t^\alpha v(\tau \mid t, w(\cdot))\|_{\mathbb{R}^k} \\
                & \vee \|\nabla^\alpha v(\tau \mid t^\prime, w^\prime(\cdot))
                - \nabla^\alpha v(\tau \mid t, w(\cdot))\|_{\mathbb{R}^{k \times n}}
                \leq \omega_v^{(2)} \Bigl( \dist \bigl( (t^\prime, w^\prime(\cdot)), (t, w(\cdot)) \bigr) \Bigr),
            \end{align*}
            where the metric $\dist$ is given by \eqref{dist}.

        (iii$_\star$)
            Let $(t, w(\cdot)) \in \mathcal{G}_0$, $x(\cdot) \in \mathcal{X}(t, w(\cdot))$ (see \eqref{X_t_w}), and $\vartheta \in (t, T)$.
            Then, there are a number $\varkappa_v \in [0, 1)$ and an infinitesimal modulus of continuity $\omega_v^\ast(\cdot)$ such that
            \begin{align*}
                & \biggl\| v(\tau \mid t^\prime, x_{t^\prime}(\cdot)) - v(\tau \mid t, w(\cdot))
                - \frac{\partial_t^\alpha v(\tau \mid t, w(\cdot)) (t^\prime - t)
                + \nabla^\alpha v(\tau \mid t, w(\cdot)) z(t^\prime)}{(T - t)^\beta (\tau - t)^\gamma} \biggr\|_{\mathbb{R}^k} \\
                & \leq \frac{\omega_v^\ast(t^\prime - t)}{(\tau - t^\prime)^{\varkappa_v}}
                \quad \forall \tau \in (t^\prime, T] \quad \forall t^\prime \in (t, \vartheta],
            \end{align*}
            where $x_{t^\prime}(\cdot)$ is the restriction of the function $x(\cdot)$ to the interval $[0, t^\prime]$ (see \eqref{x_t}) and
            \begin{equation} \label{z}
                z(\xi)
                \doteq \int_{t}^{\xi} (^C D^\alpha x)(\eta) \, \rd \eta
                \quad \forall \xi \in [t, T].
            \end{equation}

        Note that, in accordance with relation \eqref{ci-differentiability}, condition (iii$_\star$) plays a role of a property of $ci$-differentiability of order $\alpha$ of the mapping $(t, w(\cdot)) \mapsto v(\tau \mid t, w(\cdot))$.
        That is why we propose to use the notation $\partial_t^\alpha v(\tau \mid t, w(\cdot))$ and $\nabla^\alpha v(\tau \mid t, w(\cdot))$.
        On the other hand, it would be more accurate to use this notation for the considered values divided by the quantity $(T - t)^\beta (\tau - t)^\gamma$.
        However, for the goals of the present section, it is convenient to deal with these possible singularities in a direct way.

        Five propositions below describe some of transformations that can be done with families of functions satisfying Condition $(\star)$ so that the resulting family of functions also satisfies this condition.
        Propositions \ref{proposition_v_linear_combination} and \ref{proposition_v_T} follow directly from the definition, and, therefore, their proofs are omitted.
        The proofs of Propositions \ref{proposition_v_int_t_T}, \ref{proposition_v_int_t_ast_t}, and \ref{proposition_v_int_M} are relegated to Appendix \ref{appendix_2}.

        \begin{proposition} \label{proposition_v_linear_combination}
            Let families of functions $v_1(\cdot \mid \cdot, \cdot)$, $v_2(\cdot \mid \cdot, \cdot)$ satisfy Condition $(\star)$ with parameters $(\beta_1, \gamma_1)$, $(\beta_2, \gamma_2) \in [0, \infty) \times [0, 1)$, respectively.
            Then, for any $\hat{k} \in \mathbb{N}$ and any $V_1(\cdot)$, $V_2(\cdot) \in \C([0, T], \mathbb{R}^{\hat{k} \times k})$, the family of functions
            \begin{equation*}
                \hat{v}(\tau \mid t, w(\cdot))
                \doteq V_1(\tau) v_1(\tau \mid t, w(\cdot)) + V_2(\tau) v_2(\tau \mid t, w(\cdot))
                \quad \forall \tau \in [t, T] \quad \forall (t, w(\cdot)) \in \mathcal{G}
            \end{equation*}
            satisfies Condition $(\star)$ {\rm(}in which we should accordingly replace $k$ with $\hat{k}${\rm)} with the parameters $(\hat{\beta}, \hat{\gamma}) \doteq (\beta_1 \vee \beta_2, \gamma_1 \vee \gamma_2)$ and with
            \begin{equation*}
                \begin{split}
                    \partial_t^\alpha \hat{v}(\tau \mid t, w(\cdot))
                    & \doteq (T - t)^{\hat{\beta} - \beta_1} (\tau - t)^{\hat{\gamma} - \gamma_1} V_1(\tau) \partial_t^\alpha v_1(\tau \mid t, w(\cdot)) \\
                    & \quad + (T - t)^{\hat{\beta} - \beta_2} (\tau - t)^{\hat{\gamma} - \gamma_2} V_2(\tau) \partial_t^\alpha v_2(\tau \mid t, w(\cdot))
                \end{split}
            \end{equation*}
            and
            \begin{equation*}
                \begin{split}
                    \nabla^\alpha \hat{v}(\tau \mid t, w(\cdot))
                    & \doteq (T - t)^{\hat{\beta} - \beta_1} (\tau - t)^{\hat{\gamma} - \gamma_1} V_1(\tau) \nabla^\alpha v_1(\tau \mid t, w(\cdot)) \\
                    & \quad + (T - t)^{\hat{\beta} - \beta_2} (\tau - t)^{\hat{\gamma} - \gamma_2} V_2(\tau) \nabla^\alpha v_2(\tau \mid t, w(\cdot))
                \end{split}
            \end{equation*}
            for all $\tau \in [t, T]$ and all $(t, w(\cdot)) \in \mathcal{G}$.
        \end{proposition}

        \begin{proposition} \label{proposition_v_T}
            Let a family of functions $v(\cdot \mid \cdot, \cdot)$ satisfy Condition $(\star)$ with parameters $(\beta, \gamma) \in [0, \infty) \times [0, 1)$.
            Then, the family of functions
            \begin{equation*}
                \hat{v}(\tau \mid t, w(\cdot))
                \doteq v(T \mid t, w(\cdot))
                \quad \forall \tau \in [t, T] \quad \forall (t, w(\cdot)) \in \mathcal{G}
            \end{equation*}
            satisfies Condition $(\star)$ with the parameters $(\beta + \gamma, 0)$ and with
            \begin{equation*}
                \partial_t^\alpha \hat{v}(\tau \mid t, w(\cdot))
                \doteq \partial_t^\alpha v(T \mid t, w(\cdot)),
                \quad \nabla^\alpha \hat{v}(\tau \mid t, w(\cdot))
                \doteq \nabla^\alpha v(T \mid t, w(\cdot))
            \end{equation*}
            for all $\tau \in [t, T]$ and all $(t, w(\cdot)) \in \mathcal{G}$.
        \end{proposition}

        \begin{proposition} \label{proposition_v_int_t_T}
            Let a family of functions $v(\cdot \mid \cdot, \cdot)$ satisfy Condition $(\star)$ with parameters $(\beta, \gamma) \in [0, \infty) \times [0, 1 / 2)$ and let a function $W \colon \Omega \to \mathbb{R}^{k \times k}$ be continuous {\rm(}see \eqref{Omega}{\rm)}.
            Then, the family of functions
            \begin{equation*}
                \hat{v}(\tau \mid t, w(\cdot))
                \doteq \int_{\tau}^{T} \frac{W(\xi, \tau) v(\xi \mid t, w(\cdot))}{(\xi - \tau)^\gamma} \, \rd \xi
                \quad \forall \tau \in [t, T] \quad \forall (t, w(\cdot)) \in \mathcal{G}
            \end{equation*}
            satisfies Condition $(\star)$ with the parameters $(\beta, 0)$ and with
            \begin{align*}
                \partial_t^\alpha \hat{v}(\tau \mid t, w(\cdot))
                & \doteq \int_{\tau}^{T} \frac{W(\xi, \tau) \partial_t^\alpha v(\xi \mid t, w(\cdot))}
                {(\xi - \tau)^\gamma (\xi - t)^\gamma} \, \rd \xi, \\
                \nabla^\alpha \hat{v}(\tau \mid t, w(\cdot))
                & \doteq \int_{\tau}^{T} \frac{W(\xi, \tau) \nabla^\alpha v(\xi \mid t, w(\cdot))}
                {(\xi - \tau)^\gamma (\xi - t)^\gamma} \, \rd \xi
            \end{align*}
            for all $\tau \in [t, T]$ and all $(t, w(\cdot)) \in \mathcal{G}$.
        \end{proposition}

        \begin{proposition} \label{proposition_v_int_t_ast_t}
            Under the assumptions of Proposition \ref{proposition_v_int_t_T}, the family of functions
            \begin{equation*}
                \hat{v}(\tau \mid t, w(\cdot))
                \doteq \int_{t}^{\tau} \frac{W(\tau, \xi) v(\xi \mid t, w(\cdot))}{(\tau - \xi)^\gamma} \, \rd \xi
                \quad \forall \tau \in [t, T] \quad \forall (t, w(\cdot)) \in \mathcal{G}
            \end{equation*}
            satisfies Condition $(\star)$ with the parameters $(\beta, \gamma)$ and with
            \begin{equation*}
                \begin{split}
                    \partial_t^\alpha \hat{v}(\tau \mid t, w(\cdot))
                    & \doteq (\tau - t)^\gamma \int_{t}^{\tau} \frac{W(\tau, \xi) \partial_t^\alpha v(\xi \mid t, w(\cdot))}
                    {(\tau - \xi)^\gamma (\xi - t)^\gamma} \, \rd \xi \\
                    & \quad - (T - t)^\beta W(\tau, t) v(t \mid t, w(\cdot))
                \end{split}
            \end{equation*}
            and
            \begin{equation*}
                \begin{split}
                    \nabla^\alpha \hat{v}(\tau \mid t, w(\cdot))
                    & \doteq (\tau - t)^\gamma \int_{t}^{\tau} \frac{W(\tau, \xi) \nabla^\alpha v(\xi \mid t, w(\cdot))}
                    {(\tau - \xi)^\gamma (\xi - t)^\gamma} \, \rd \xi
                \end{split}
            \end{equation*}
            for all $\tau \in [t, T]$ and all $(t, w(\cdot)) \in \mathcal{G}$.
        \end{proposition}

        \begin{proposition} \label{proposition_v_int_M}
            Let a family of functions $v(\cdot \mid \cdot, \cdot)$ satisfy Condition $(\star)$ with parameters $(\beta, 0)$, where $\beta \geq 0$.
            Suppose that a number $\sigma \in [0, 1)$ and continuous functions $N \colon \Theta \to \mathbb{R}^{k \times k}$ and $\partial_t N \colon \Theta \to \mathbb{R}^{k \times k}$ are given {\rm(}see \eqref{Theta}{\rm)} satisfying the following condition:
            for any $t \in [0, T)$ and any $\vartheta \in (t, T)$, there exists an infinitesimal modulus of continuity $\omega_{N}^\ast(\cdot)$ such that, for any $t^\prime \in (t, \vartheta]$ and any $\tau$, $\xi \in [t^\prime, T]$,
            \begin{equation} \label{omega_N^ast}
                \biggl\| N(\tau, \xi \mid t^\prime) - N(\tau, \xi \mid t)
                - \frac{(t^\prime - t) \partial_t N (\tau, \xi \mid t)}{(T - t)^\sigma} \biggr\|_{\mathbb{R}^{k \times k}}
                \leq \omega_{N}^\ast(t^\prime - t).
            \end{equation}
            Then, the family of functions
            \begin{equation*}
                \hat{v}(\tau \mid t, w(\cdot))
                \doteq \int_{t}^{T} \frac{N(\tau, \xi \mid t) v(\xi \mid t, w(\cdot))}{(T - \xi)^\sigma} \, \rd \xi
                \quad \forall \tau \in [t, T] \quad \forall (t, w(\cdot)) \in \mathcal{G}
            \end{equation*}
            satisfies Condition $(\star)$ with the parameters $(\hat{\beta}, 0)$, where $\hat{\beta} \doteq \beta \vee \sigma$, and with
            \begin{equation*}
                \begin{split}
                    \partial_t^\alpha \hat{v}(\tau \mid t, w(\cdot))
                    & \doteq (T - t)^{\hat{\beta} - \sigma} \int_{t}^{T} \frac{\partial_t N(\tau, \xi \mid t) v(\xi \mid t, w(\cdot))}{(T - \xi)^\sigma} \, \rd \xi \\
                    & \quad + (T - t)^{\hat{\beta} - \beta} \int_{t}^{T} \frac{N(\tau, \xi \mid t) \partial_t^\alpha v(\xi \mid t, w(\cdot))}
                    {(T - \xi)^\sigma} \, \rd \xi \\
                    & \quad - (T - t)^{\hat{\beta} - \sigma} N(\tau, t \mid t) v(t \mid t, w(\cdot))
                \end{split}
            \end{equation*}
            and
            \begin{equation*}
                \begin{split}
                    \nabla^\alpha \hat{v}(\tau \mid t, w(\cdot))
                    & \doteq (T - t)^{\hat{\beta} - \beta} \int_{t}^{T} \frac{N(\tau, \xi \mid t) \nabla^\alpha v(\xi \mid t, w(\cdot))}
                    {(T - \xi)^\sigma} \, \rd \xi
                \end{split}
            \end{equation*}
            for all $\tau \in [t, T]$ and all $(t, w(\cdot)) \in \mathcal{G}$.
        \end{proposition}

        Let us proceed with consideration of two types of $ci$-smooth of order $\alpha$ quadratic functionals and providing rules for calculation their $ci$-derivatives of order $\alpha$.
        The proofs of Lemmas \ref{lemma_terminal_functional} and \ref{lemma_integral_functional} below are relegated to Appendix \ref{appendix_2}.

        \begin{lemma} \label{lemma_terminal_functional}
            Let families of functions $v_1(\cdot \mid \cdot, \cdot)$, $v_2(\cdot \mid \cdot, \cdot)$ satisfy Condition $(\star)$ with parameters $(\beta_1, \gamma_1)$, $(\beta_2, \gamma_2) \in [0, \infty) \times [0, 1)$, respectively.
            Then, the functional
            \begin{equation*}
                \varphi(t, w(\cdot))
                \doteq \langle v_1(T \mid t, w(\cdot)), v_2(T \mid t, w(\cdot)) \rangle_{\mathbb{R}^k}
                \quad \forall (t, w(\cdot)) \in \mathcal{G}
            \end{equation*}
            is $ci$-smooth of order $\alpha$ and its $ci$-derivatives of order $\alpha$ are given by
            \begin{align}
                \label{varphi_ci_derivative_t}
                \partial_t^\alpha \varphi(t, w(\cdot))
                & = \frac{\langle \partial_t^\alpha v_1(T), v_2(T) \rangle_{\mathbb{R}^k}}
                {(T - t)^{\beta_1 + \gamma_1}}
                + \frac{\langle \partial_t^\alpha v_2(T), v_1(T) \rangle_{\mathbb{R}^k}}
                {(T - t)^{\beta_2 + \gamma_2}}, \\
                \label{varphi_ci_gradient}
                \nabla^\alpha \varphi(t, w(\cdot))
                & = \frac{\nabla^\alpha v_1(T)^\top v_2(T)}
                {(T - t)^{\beta_1 + \gamma_1}}
                + \frac{\nabla^\alpha v_2(T)^\top v_1(T)}
                {(T - t)^{\beta_2 + \gamma_2}}
            \end{align}
            for all $(t, w(\cdot)) \in \mathcal{G}_0$, where the following notation is used with $i \in \{1, 2\}$:
            \begin{equation} \label{v_i_notation_lemma}
                v_i(\cdot)
                \doteq v_i(\cdot \mid t, w(\cdot)),
                \ \, \partial_t^\alpha v_i(\cdot)
                \doteq \partial_t^\alpha v_i(\cdot \mid t, w(\cdot)),
                \ \, \nabla^\alpha v_i(\cdot)
                \doteq \nabla^\alpha v_i(\cdot \mid t, w(\cdot)).
            \end{equation}
        \end{lemma}

        \begin{lemma} \label{lemma_integral_functional}
            Let families of functions $v_1(\cdot \mid \cdot, \cdot)$, $v_2(\cdot \mid \cdot, \cdot)$ satisfy Condition $(\star)$ with parameters $(\beta_1, \gamma_1)$, $(\beta_2, \gamma_2) \in [0, \infty) \times [0, 1)$, respectively, and let a number $\sigma \in [0, 1)$ be given.
            Then, the functional
            \begin{equation*}
                \varphi(t, w(\cdot))
                \doteq \int_{t}^{T} \frac{\langle v_1(\tau \mid t, w(\cdot)),
                v_2(\tau \mid t, w(\cdot)) \rangle_{\mathbb{R}^k}}{(T - \tau)^\sigma} \, \rd \tau
                \quad \forall (t, w(\cdot)) \in \mathcal{G}
            \end{equation*}
            is $ci$-smooth of order $\alpha$ and its $ci$-derivatives of order $\alpha$ are given by
            \begin{equation} \label{varphi_integral_ci_derivative_t}
                \begin{split}
                    \partial_t^\alpha \varphi(t, w(\cdot))
                    & = \int_{t}^{T} \frac{1}{(T - \tau)^\sigma}
                    \biggl( \frac{\langle \partial_t^\alpha v_1(\tau), v_2(\tau) \rangle_{\mathbb{R}^k}}
                    {(T - t)^{\beta_1} (\tau - t)^{\gamma_1}}
                    + \frac{\langle \partial_t^\alpha v_2(\tau), v_1(\tau) \rangle_{\mathbb{R}^k}}
                    {(T - t)^{\beta_2} (\tau - t)^{\gamma_2}} \biggr) \, \rd \tau \\
                    & \quad - \frac{\langle v_1(t), v_2(t) \rangle_{\mathbb{R}^k}}{(T - t)^\sigma}
                \end{split}
            \end{equation}
            and
            \begin{equation} \label{varphi_integral_ci_gradient}
                    \nabla^\alpha \varphi(t, w(\cdot))
                    =  \int_{t}^{T} \frac{1}{(T - \tau)^\sigma}
                    \biggl( \frac{\nabla^\alpha v_1(\tau)^\top v_2(\tau)}{(T - t)^{\beta_1} (\tau - t)^{\gamma_1}}
                    + \frac{\nabla^\alpha v_2(\tau)^\top v_1(\tau)}
                    {(T - t)^{\beta_2} (\tau - t)^{\gamma_2}} \biggr) \, \rd \tau
            \end{equation}
            for all $(t, w(\cdot)) \in \mathcal{G}_0$, where the shorthand notation \eqref{v_i_notation_lemma} is used.
        \end{lemma}

    \subsection{Proof of Theorem \ref{theorem_solution}}
    \label{subsection_proof_theorem_solution}
        Let us consider the family of functions $a(\cdot \mid \cdot, \cdot)$ from \eqref{a_definition}.
        By \cite[Lemma 3]{Gomoyunov_2020_DE}, the mapping $\mathcal{G} \ni (t, w(\cdot)) \mapsto a(\cdot \mid t, w(\cdot)) \in \AC^\alpha([0, T], \mathbb{R}^n)$ is continuous.
        Further, let us fix $(t, w(\cdot)) \in \mathcal{G}_0$ and $x(\cdot) \in \mathcal{X}(t, w(\cdot))$ and denote
        \begin{equation} \label{mu_x^ast}
            \mu_x^\ast
            \doteq \esssup{\tau \in [0, T]} \|(^C D^\alpha x)(\tau)\|_{\mathbb{R}^n}.
        \end{equation}
        In particular, for the function $z(\cdot)$ given by \eqref{z}, we obtain the estimate
        \begin{equation} \label{z_estimate}
            \|z(\xi)\|_{\mathbb{R}^n}
            \leq \mu_x^\ast (\xi - t)
            \quad \forall \xi \in [t, T].
        \end{equation}
        Hence, applying the integration by parts formula, we get (see also \cite[Section 12]{Gomoyunov_2020_SIAM})
        \begin{equation*}
            \begin{split}
                & \biggl\| a(\tau \mid t^\prime, x_{t^\prime}(\cdot)) - a(\tau \mid t, w(\cdot))
                - \frac{z(t^\prime)}{\Gamma(\alpha) (\tau - t)^{1 - \alpha}} \biggr\|_{\mathbb{R}^n} \\
                & \leq \frac{1}{\Gamma(\alpha)} \biggl\| \frac{z(t^\prime)}{(\tau - t^\prime)^{1 - \alpha}}
                - \frac{z(t^\prime)}{(\tau - t)^{1 - \alpha}} \biggr\|_{\mathbb{R}^n}
                + \frac{1 - \alpha}{\Gamma(\alpha)} \int_{t}^{t^\prime} \frac{\|z(\xi)\|_{\mathbb{R}^n}}{(\tau - \xi)^{2 - \alpha}} \, \rd \xi \\
                & \leq \frac{2 \mu_x^\ast (t^\prime - t)^{2 - \alpha}}{\Gamma(\alpha) (\tau - t^\prime)^{2 - 2 \alpha}}
                \quad \forall \tau \in (t^\prime, T] \quad \forall t^\prime \in [t, T).
            \end{split}
        \end{equation*}
        Thus, we conclude that the family of functions $a(\cdot \mid \cdot, \cdot)$ (more precisely, the family consisting of restrictions of the functions $a(\cdot \mid t, w(\cdot))$ to $[t, T]$ for all $(t, w(\cdot)) \in \mathcal{G}$) satisfies Condition $(\star)$ with the parameters $(0, 1 - \alpha)$ and with
        \begin{equation*}
            \partial_t^\alpha a(\tau \mid t, w(\cdot))
            \doteq \mathbb{O}_n,
            \quad \nabla^\alpha a(\tau \mid t, w(\cdot))
            \doteq \frac{\mathbb{I}_{n \times n}}{\Gamma(\alpha)}
            \quad \forall \tau \in [t, T] \quad \forall (t, w(\cdot)) \in \mathcal{G}.
        \end{equation*}
        Consequently, by Propositions \ref{proposition_v_linear_combination} and \ref{proposition_v_int_t_ast_t}, the family of functions $b(\cdot \mid \cdot, \cdot)$ from \eqref{b_definition} satisfies Condition $(\star)$ with the parameters $(0, 1 - \alpha)$ and with (see also definition \eqref{fundamental_solution_matrix} of the fundamental solution matrix $\Phi(\cdot, \cdot)$)
        \begin{align*}
            \partial_t^\alpha b(\tau \mid t, w(\cdot))
            & \doteq - \Phi(\tau, t) A(t) a(t \mid t, w(\cdot))
            = - \Phi(\tau, t) A(t) w(t), \\
            \nabla^\alpha b(\tau \mid t, w(\cdot))
            & \doteq \frac{\mathbb{I}_{n \times n}}{\Gamma(\alpha)}
            + \frac{(\tau - t)^{1 - \alpha}}{\Gamma(\alpha)} \int_{t}^{\tau} \frac{\Phi(\tau, \xi) A(\xi)}
            {(\tau - \xi)^{1 - \alpha} (\xi - t)^{1 - \alpha}} \, \rd \xi
            = \Phi(\tau, t)
        \end{align*}
        for all $\tau \in [t, T]$ and all $(t, w(\cdot)) \in \mathcal{G}$.
        Let us observe that
        \begin{equation*}
            \partial_t^\alpha b(\tau \mid t, w(\cdot))
            = - \nabla^\alpha b(\tau \mid t, w(\cdot)) A(t) w(t)
            \quad \forall \tau \in [t, T] \quad  \forall (t, w(\cdot)) \in \mathcal{G}.
        \end{equation*}

        Then, taking Proposition \ref{proposition_v_linear_combination} into account, applying Lemma \ref{lemma_terminal_functional}, and recalling that the matrix $P$ is symmetric, we derive that the functional
        \begin{equation*}
            \varphi_1(t, w(\cdot))
            \doteq \langle b(T \mid t, w(\cdot)), P b(T \mid t, w(\cdot)) \rangle_{\mathbb{R}^n}
            \quad \forall (t, w(\cdot)) \in \mathcal{G}
        \end{equation*}
        is $ci$-smooth of order $\alpha$ and its $ci$-derivatives of order $\alpha$ are as follows:
        \begin{align*}
            \partial_t^\alpha \varphi_1(t, w(\cdot))
            & = - \langle \nabla^\alpha \varphi_1(t, w(\cdot)), A(t) w(t) \rangle_{\mathbb{R}^n}, \\
            \nabla^\alpha \varphi_1(t, w(\cdot))
            & = \frac{2 \Phi(T, t)^\top P b(T \mid t, w(\cdot))}
            {(T - t)^{1 - \alpha}}
            \quad \forall (t, w(\cdot)) \in \mathcal{G}_0.
        \end{align*}
        In addition, by Proposition \ref{proposition_v_linear_combination} and Lemma \ref{lemma_integral_functional}, recalling that the matrices $Q(\tau)$ for all $\tau \in [0, T]$ are symmetric and noting that $b(t \mid t, w(\cdot)) = a(t \mid t, w(\cdot)) = w(t)$ by definition, we obtain that the functional
        \begin{equation*}
            \varphi_2(t, w(\cdot))
            \doteq \int_{t}^{T}
            \langle b(\tau \mid t, w(\cdot)), Q(\tau) b(\tau \mid t, w(\cdot)) \rangle_{\mathbb{R}^n} \, \rd \tau
            \quad \forall (t, w(\cdot)) \in \mathcal{G}
        \end{equation*}
        is $ci$-smooth of order $\alpha$ and
        \begin{align*}
            \partial_t^\alpha \varphi_2(t, w(\cdot))
            & = - \langle \nabla^\alpha \varphi_2(t, w(\cdot)), A(t) w(t) \rangle_{\mathbb{R}^n}
            - \langle w(t), Q(t) w(t) \rangle_{\mathbb{R}^n}, \\
            \nabla^\alpha \varphi_2(t, w(\cdot))
            & = 2 \int_{t}^{T}
            \frac{\Phi(\tau, t)^\top Q(\tau) b(\tau \mid t, w(\cdot))}
            {(\tau - t)^{1 - \alpha}} \, \rd \tau
            \quad \forall (t, w(\cdot)) \in \mathcal{G}_0.
        \end{align*}

        Further, in view of the properties of the family of functions $b(\cdot \mid \cdot, \cdot)$ described above and Propositions \ref{proposition_v_linear_combination}, \ref{proposition_v_T}, and \ref{proposition_v_int_t_T}, the family of functions $c(\cdot \mid \cdot, \cdot)$ from \eqref{c_definition} satisfies Condition $(\star)$ with the parameters $(1 - \alpha, 0)$ and with
        \begin{equation*}
            \partial_t^\alpha c(\tau \mid t, w(\cdot))
            \doteq - \nabla^\alpha c(\tau \mid t, w(\cdot)) A(t) w(t)
        \end{equation*}
        and
        \begin{equation*}
            \begin{split}
                \nabla^\alpha c(\tau \mid t, w(\cdot))
                & \doteq B(\tau)^\top \Phi(T, \tau)^\top P \Phi(T, t) \\
                & \quad + (T - t)^{1 - \alpha} (T - \tau)^{1 - \alpha} B(\tau)^\top \int_{\tau}^{T} \frac{\Phi(\xi, \tau)^\top Q(\xi) \Phi(\xi, t)}
                {(\xi - \tau)^{1 - \alpha} (\xi - t)^{1 - \alpha}} \, \rd \xi
            \end{split}
        \end{equation*}
        for all $\tau \in [t, T]$ and all $(t, w(\cdot)) \in \mathcal{G}$.
        In accordance with \eqref{varphi_main}, let us introduce one more family of functions:
        \begin{equation*}
            d(\tau \mid t, w(\cdot))
            \doteq R(\tau)^{- 1} c(\tau \mid t, w(\cdot))
            + \int_{t}^{T} \frac{M(\tau, \xi \mid t) c(\xi \mid t, w(\cdot))}{(T - \xi)^{2 - 2 \alpha}} \, \rd \xi
        \end{equation*}
        for all $\tau \in [t, T]$ and all $(t, w(\cdot)) \in \mathcal{G}$.
        By Propositions \ref{proposition_v_linear_combination} and \ref{proposition_v_int_M} and due to the properties of the function $M(\cdot, \cdot \mid \cdot)$ established in Proposition \ref{proposition_M}, this family of functions satisfies Condition $(\star)$ with the parameters $(2 - 2 \alpha, 0)$ and with
        \begin{align*}
            \partial_t^\alpha d(\tau \mid t, w(\cdot))
            & \doteq - \nabla^\alpha d(\tau \mid t, w(\cdot)) A(t) w(t)
            - M(\tau, t \mid t) c(t \mid t, w(\cdot)) \\
            & \quad - M(\tau, t \mid t) R(t) \int_{t}^{T} \frac{M(t, \xi \mid t) c(\xi \mid t, w(\cdot))}
            {(T - \xi)^{2 - 2 \alpha}} \, \rd \xi
        \end{align*}
        and
        \begin{align*}
            \nabla^\alpha d(\tau \mid t, w(\cdot))
            & \doteq (T - t)^{1 - \alpha} R(\tau)^{- 1} \nabla^\alpha c(\tau \mid t, w(\cdot)) \\
            & \quad + (T - t)^{1 - \alpha} \int_{t}^{T} \frac{M(\tau, \xi \mid t) \nabla^\alpha c(\xi \mid t, w(\cdot))}
            {(T - \xi)^{2 - 2 \alpha}} \, \rd \xi
        \end{align*}
        for all $\tau \in [t, T]$ and all $(t, w(\cdot)) \in \mathcal{G}$.
        Then, by Lemma \ref{lemma_integral_functional}, the functional
        \begin{equation*}
            \varphi_3(t, w(\cdot))
            \doteq \int_{t}^{T} \frac{\langle c(\tau \mid t, w(\cdot)), d(\tau \mid t, w(\cdot)) \rangle_{\mathbb{R}^m}}
             {(T - \tau)^{2 - 2 \alpha}} \, \rd \tau
             \quad \forall (t, w(\cdot)) \in \mathcal{G}
        \end{equation*}
        is $ci$-smooth of order $\alpha$ and its $ci$-derivatives of order $\alpha$ are as follows:
        \begin{equation*}
            \begin{split}
                & \partial_t^\alpha \varphi_3(t, w(\cdot)) \\
                & = \int_{t}^{T}
                \biggl( \frac{\langle \partial_t^\alpha c(\tau \mid t, w(\cdot)), d(\tau \mid t, w(\cdot)) \rangle_{\mathbb{R}^m}}
                {(T - \tau)^{2 - 2 \alpha} (T - t)^{1 - \alpha}}
                + \frac{\langle \partial_t^\alpha d(\tau \mid t, w(\cdot)), c(\tau \mid t, w(\cdot)) \rangle_{\mathbb{R}^m}}
                {(T - \tau)^{2 - 2 \alpha} (T - t)^{2 - 2 \alpha}} \biggr) \, \rd \tau \\
                & \quad - \frac{\langle c(t \mid t, w(\cdot)), d(t \mid t, w(\cdot)) \rangle_{\mathbb{R}^m}}{(T - t)^{2 - 2 \alpha}}
            \end{split}
        \end{equation*}
        and
        \begin{equation*}
            \begin{split}
                & \nabla^\alpha \varphi_3(t, w(\cdot)) \\
                & = \int_{t}^{T} \biggl( \frac{\nabla^\alpha c(\tau \mid t, w(\cdot))^\top d(\tau \mid t, w(\cdot))}
                {(T - \tau)^{2 - 2 \alpha} (T - t)^{1 - \alpha}}
                + \frac{\nabla^\alpha d(\tau \mid t, w(\cdot))^\top c(\tau \mid t, w(\cdot))}
                {(T - \tau)^{2 - 2 \alpha} (T - t)^{2 - 2 \alpha}} \biggr) \, \rd \tau
            \end{split}
        \end{equation*}
        for all $(t, w(\cdot)) \in \mathcal{G}_0$.
        Moreover, since that the matrices $R(\tau)$ are symmetric for all $\tau \in [0, T]$ and the function $M(\cdot, \cdot \mid \cdot)$ satisfies equality \eqref{M_symmetric}, we derive
        \begin{equation*}
            \partial_t^\alpha \varphi_3(t, w(\cdot))
            = - \langle \nabla^\alpha \varphi_3(t, w(\cdot)), A(t) w(t) \rangle_{\mathbb{R}^n}
            - \frac{\langle d(t \mid t, w(\cdot)), R(t) d(t \mid t, w(\cdot)) \rangle_{\mathbb{R}^m}}
            {(T - t)^{2 - 2 \alpha}}
        \end{equation*}
        and
        \begin{equation*}
            \nabla^\alpha \varphi_3(t, w(\cdot))
            = \frac{2}{(T - t)^{1 - \alpha}} \int_{t}^{T}
            \frac{\nabla^\alpha c(\tau \mid t, w(\cdot))^\top d(\tau \mid t, w(\cdot))}{(T - \tau)^{2 - 2 \alpha}} \, \rd \tau
        \end{equation*}
        for all $(t, w(\cdot)) \in \mathcal{G}_0$.

        As a result, observing that $\varphi = \varphi_1 + \varphi_2 - \varphi_3$ by construction, we obtain that the functional $\varphi$ is $ci$-smooth of order $\alpha$ with the $ci$-derivatives of order $\alpha$ given by
        \begin{equation} \label{partial_t_varphi_main}
            \begin{split}
                \partial_t^\alpha \varphi(t, w(\cdot))
                & = - \langle \nabla^\alpha \varphi(t, w(\cdot)), A(t) w(t) \rangle_{\mathbb{R}^n}
                - \langle w(t), Q(t) w(t) \rangle_{\mathbb{R}^n} \\
                & \quad + \frac{\langle d(t \mid t, w(\cdot)), R(t) d(t \mid t, w(\cdot)) \rangle_{\mathbb{R}^m}}
                {(T - t)^{2 - 2 \alpha}}
            \end{split}
        \end{equation}
        and
        \begin{equation*}
            \nabla^\alpha \varphi(t, w(\cdot))
            = \nabla^\alpha \varphi_1 (t, w(\cdot)) + \nabla^\alpha \varphi_2 (t, w(\cdot))
            - \nabla^\alpha \varphi_3 (t, w(\cdot))
        \end{equation*}
        for all $(t, w(\cdot)) \in \mathcal{G}_0$.

        Further, let $(t, w(\cdot)) \in \mathcal{G}_0$ be fixed.
        We have $B(t)^\top \nabla^\alpha c(\tau \mid t, w(\cdot))^\top = K(t, \tau)$ for all $\tau \in [t, T]$ in view of definition \eqref{K_definition} of the function $K(\cdot, \cdot)$ and equality \eqref{K_symmetric}.
        Hence, using definition \eqref{M_definition} of the function $M(\cdot, \cdot \mid \cdot)$, we get
        \begin{align*}
            B(t)^\top \nabla^\alpha \varphi_3(t, w(\cdot))
            & = \frac{2}{(T - t)^{1 - \alpha}} \int_{t}^{T}
            \frac{K(t, \tau) d(\tau \mid t, w(\cdot))}{(T - \tau)^{2 - 2 \alpha}} \, \rd \tau \\
            & = \frac{2}{(T - t)^{1 - \alpha}} \biggl( \int_{t}^{T}
            \frac{K(t, \tau) R(\tau)^{- 1} c(\tau \mid t, w(\cdot))}{(T - \tau)^{2 - 2 \alpha}} \, \rd \tau \\
            & \quad + \int_{t}^{T} \int_{t}^{T}
            \frac{K(t, \tau) M(\tau, \xi \mid t) c(\xi \mid t, w(\cdot))}
            {(T - \tau)^{2 - 2 \alpha} (T - \xi)^{2 - 2 \alpha}} \, \rd \tau \, \rd \xi \biggr) \\
            & = \frac{- 2 R(t)}{(T - t)^{1 - \alpha}} \int_{t}^{T} \frac{M(t, \xi \mid t) c(\xi \mid t, w(\cdot))}{(T - \xi)^{2 - 2 \alpha}} \, \rd \xi.
        \end{align*}
        Consequently, observing that
        \begin{equation*}
            B(t)^\top \bigl( \nabla^\alpha \varphi_1(t, w(\cdot)) + \nabla^\alpha \varphi_2(t, w(\cdot)) \bigr) \\
            = \frac{2 c(t \mid t, w(\cdot))}{(T - t)^{1 - \alpha}},
        \end{equation*}
        we derive
        \begin{equation} \label{nabla_varphi_main_proof}
            B(t)^\top \nabla^\alpha \varphi (t, w(\cdot))
            = \frac{2 R(t) d(t \mid t, w(\cdot))}{(T - t)^{1 - \alpha}},
        \end{equation}
        which, in particular, yields \eqref{nabla_varphi_main}.
        Moreover, it follows directly from \eqref{partial_t_varphi_main} and \eqref{nabla_varphi_main_proof} that the functional $\varphi$ satisfies the HJB equation \eqref{HJB} at the point $(t, w(\cdot))$.

        Thus, in order to complete the proof, it remains to note that the functional $\varphi$ satisfies the boundary condition \eqref{boundary_condition} since
        \begin{equation*}
            \varphi(T, w(\cdot))
            = \varphi_1(T, w(\cdot))
            = \langle b(T \mid T, w(\cdot)), P b(T \mid T, w(\cdot)) \rangle_{\mathbb{R}^n}
            = \langle w(T), P w(T) \rangle_{\mathbb{R}^n}
        \end{equation*}
        for all $w(\cdot) \in \AC^\alpha([0, T], \mathbb{R}^n)$ by construction.

\section{Optimal feedback control}
\label{section_OFC}

    By the functional $\varphi$ from \eqref{varphi_main}, in accordance with Theorem \ref{theorem_solution} and definition \eqref{Hamiltonian} of the Hamiltonian, let us define a mapping $U^\circ \colon \mathcal{G}_0 \to \mathbb{R}^m$ as follows:
    \begin{equation} \label{U^0}
        \begin{split}
            U^\circ(t, w(\cdot))
            & \doteq - \frac{1}{2} R(t)^{- 1} B(t)^\top \nabla^\alpha \varphi(t, w(\cdot)) \\
            & = - \frac{1}{(T - t)^{1 - \alpha}} \biggl( R(t)^{- 1} c(t)
            + \int_{t}^{T} \frac{M(t, \xi \mid t) c(\xi)}{(T - \xi)^{2 - 2 \alpha}} \, \rd \xi \biggr)
        \end{split}
    \end{equation}
    for all $(t, w(\cdot)) \in \mathcal{G}_0$, where the function $M(\cdot, \cdot \mid \cdot)$ is a solution of the Fredholm integral equation \eqref{M_definition} from Proposition \ref{proposition_M} and the function $c(\cdot) \doteq c(\cdot \mid t, w(\cdot))$ is taken from \eqref{c_definition}.
    In particular, we have (see \eqref{pre_Hamiltonian})
    \begin{equation} \label{U^0_corollary}
        h \bigl( t, w(t), U^\circ(t, w(\cdot)), \nabla^\alpha \varphi(t, w(\cdot)) \bigr)
        = H \bigl( t, w(t), \nabla^\alpha \varphi(t, w(\cdot)) \bigr)
    \end{equation}
    for all $(t, w(\cdot)) \in \mathcal{G}_0$.
    Note that, for any $t \in [0, T)$, the mapping below is linear:
    \begin{equation*}
        \AC^\alpha([0, t], \mathbb{R}^n) \ni w(\cdot) \mapsto U^\circ(t, w(\cdot)).
    \end{equation*}

    In the terminology of the theory of positional differential games (see, e.g., \cite{Krasovskii_Subbotin_1988,Krasovskii_Krasovskii_1995} and also \cite{Gomoyunov_2020_SIAM}), the mapping $U^\circ(\cdot, \cdot)$ is called a {\it positional control strategy}.
    Let us describe how we use the strategy $U^\circ(\cdot, \cdot)$ to generate an admissible control and the corresponding motion of system \eqref{differential_equation}.

    Let a pair $(t, w(\cdot)) \in \mathcal{G}_0$ be given.
    Let us fix a number $\vartheta \in (t, T)$ and consider a partition $\Delta$ of the interval $[t, \vartheta]$:
    \begin{equation*}
        \Delta \doteq \{\tau_j\}_{j \in \overline{1, \ell + 1}},
        \quad \tau_1 = t,
        \quad \tau_j < \tau_{j + 1}
        \quad \forall j \in \overline{1, \ell},
        \quad \tau_{\ell + 1} = \vartheta,
    \end{equation*}
    where $\ell \in \mathbb{N}$.
    Then, we form a piecewise constant control $u^\circ(\cdot) \in \Linf([t, T], \mathbb{R}^m)$ and the motion $x^\circ(\cdot) \doteq x(\cdot \mid t, w(\cdot), u^\circ(\cdot))$ of system \eqref{differential_equation} by the following {\it step-by-step feedback control procedure}:
    at every time $\tau_j$ with $j \in \overline{1, \ell}$, we measure the history $x^\circ_{\tau_j}(\cdot)$ (see \eqref{x_t}) of the motion $x^\circ(\cdot)$ on $[0, \tau_j]$, compute the value $u_j^\circ \doteq U^\circ(\tau_j, x^\circ_{\tau_j}(\cdot))$ by \eqref{U^0}, and then apply the constant control $u^\circ(\tau) \doteq u_j^\circ$ until time $\tau_{j + 1}$, when we take a new measurement of the history.
    In a short form, we have
    \begin{equation} \label{control_law}
        u^\circ(\tau)
        \doteq U^\circ(\tau_j, x^\circ_{\tau_j}(\cdot))
        \quad \forall \tau \in [\tau_j, \tau_{j + 1}) \quad \forall j \in \overline{1, \ell}.
    \end{equation}
    On the remaining (small) interval $[\vartheta, T]$, we put formally
    \begin{equation} \label{control_law_0}
        u^\circ(\tau)
        \doteq \mathbb{O}_{\mathbb{R}^m}
        \quad \forall \tau \in [\vartheta, T].
    \end{equation}
    Let us note that the described control procedure determines $u^\circ(\cdot)$ and $x^\circ(\cdot)$ uniquely.
    In what follows, we denote the obtained control by $u^\circ(\cdot \mid t, w(\cdot), \vartheta, \Delta)$.

    \begin{remark}
        According to \eqref{U^0}, the control strategy $U^\circ(\cdot, \cdot)$, in general, has the singularity $(T - t)^{1 - \alpha}$, which is a consequence of relation \eqref{nabla_varphi_main}.
        In the proposed control procedure, similarly to \cite{Gomoyunov_2020_ACS}, this difficulty is overcome by introducing the parameter $\vartheta$ and, roughly speaking, by shifting the terminal time from $T$ to $\vartheta$.
    \end{remark}

    The main result of the second part of the paper is
    \begin{theorem} \label{theorem_optimal_control}
        The functional $\varphi$ from \eqref{varphi_main} is the value functional of the FLQOCP.
        In addition, for any $(t, w(\cdot)) \in \mathcal{G}_0$ and any $\varepsilon > 0$, there exists a number $\vartheta^\circ \in (t, T)$ such that the following statement holds:
        for any $\vartheta \in [\vartheta^\circ, T)$, there exists a number $\delta > 0$ such that, for any partition $\Delta \doteq \{\tau_j\}_{j \in \overline{1, \ell + 1}}$ of the interval $[t, \vartheta]$ with the diameter $\diam(\Delta) \doteq \max_{j \in \overline{1, \ell}} (\tau_{j + 1} - \tau_j) \leq \delta$, the control $u^\circ(\cdot \mid t, w(\cdot), \vartheta, \Delta)$, which is generated by the procedure described above, is $\varepsilon$-optimal.
    \end{theorem}

    Before proceeding with the proof of Theorem \ref{theorem_optimal_control}, we establish some estimates.

    Let us note that, according to, e.g., \cite[Section 6, p. 75]{Gomoyunov_Lukoyanov_2021}, we have
    \begin{equation} \label{a_estimate}
        \|a(\tau \mid t, w(\cdot))\|_{\mathbb{R}^n}
        \leq \|w(\cdot)\|_{\C([0, t], \mathbb{R}^n)}
        \quad \forall \tau \in [t, T] \quad \forall (t, w(\cdot)) \in \mathcal{G},
    \end{equation}
    where the function $a(\cdot \mid t, w(\cdot))$ is given by \eqref{a_definition}.

    \begin{proposition} \label{proposition_mu_U}
        There exists a number $\mu_{U^\circ} \geq 0$ such that
        \begin{equation*}
            \|U^\circ(t, w(\cdot))\|_{\mathbb{R}^m}
            \leq \frac{\mu_{U^\circ} \|w(\cdot)\|_{\C([0, t], \mathbb{R}^n)}}{(T - t)^{1 - \alpha}}
            \quad \forall (t, w(\cdot)) \in \mathcal{G}_0.
        \end{equation*}
    \end{proposition}
    \begin{proof}
        Recalling that the functions $A(\cdot)$, $B(\cdot)$, $Q(\cdot)$, $\Phi(\cdot, \cdot)$ (see \eqref{fundamental_solution_matrix}), and $M(\cdot, \cdot \mid \cdot)$ (see Proposition \ref{proposition_M}) are continuous, let us denote
        \begin{equation} \label{mu_A_mu_B_mu_Q}
            \mu_A
            \doteq \|A(\cdot)\|_{\C([0, T], \mathbb{R}^{n \times n})},
            \ \mu_B
            \doteq \|B(\cdot)\|_{\C([0, T], \mathbb{R}^{n \times m})},
            \ \mu_Q
            \doteq \|Q(\cdot)\|_{\C([0, T], \mathbb{R}^{n \times n})}
        \end{equation}
        and
        \begin{equation} \label{mu_Phi_mu_M}
            \mu_\Phi
            \doteq \max_{(\tau, \xi) \in \Omega} \|\Phi(\tau, \xi)\|_{\mathbb{R}^{m \times m}},
            \quad \mu_M
            \doteq \max_{(\tau, \xi, t) \in \Theta} \|M(\tau, \xi \mid t)\|_{\mathbb{R}^{m \times m}}.
        \end{equation}
        Then, based on \eqref{a_estimate}, we obtain that the families of functions $b(\cdot \mid \cdot, \cdot)$ and $c(\cdot \mid \cdot, \cdot)$ from \eqref{b_definition} and \eqref{c_definition} satisfy the estimates
        \begin{align}
            \label{mu_b}
            \|b(\tau \mid t, w(\cdot))\|_{\mathbb{R}^n}
            & \leq \mu_b \|w(\cdot)\|_{\C([0, t], \mathbb{R}^n)}, \\
            \label{mu_c}
            \|c(\tau \mid t, w(\cdot))\|_{\mathbb{R}^m}
            & \leq \mu_c \|w(\cdot)\|_{\C([0, t], \mathbb{R}^n)}
        \end{align}
        for all $\tau \in [t, T]$ and all $(t, w(\cdot)) \in \mathcal{G}$ with the numbers
        \begin{equation*}
            \mu_b
            \doteq 1 + \frac{T^\alpha \mu_\Phi \mu_A}{\alpha},
            \quad \mu_c
            \doteq \mu_B \mu_\Phi \mu_b \biggl( \|P\|_{\mathbb{R}^{n \times n}} + \frac{T \mu_Q}{\alpha} \biggr).
        \end{equation*}
        Therefore, according to \eqref{U^0}, and since
        \begin{equation} \label{R^-1_estimate}
            \|R(\cdot)^{- 1}\|_{\C([0, T], \mathbb{R}^{m \times m})}
            \leq \frac{1}{\theta}
        \end{equation}
        due to \eqref{theta}, we derive
        \begin{equation*}
            \|U^\circ(t, w(\cdot))\|_{\mathbb{R}^m}
            \leq \frac{\mu_c \|w(\cdot)\|_{\C([0, t], \mathbb{R}^n)}}{(T - t)^{1 - \alpha}} \biggl( \frac{1}{\theta} + \frac{T^{2 \alpha - 1} \mu_M }{2 \alpha - 1} \biggr)
            \quad \forall (t, w(\cdot)) \in \mathcal{G}_0,
        \end{equation*}
        which completes the proof of the proposition.
    \end{proof}

    \begin{proposition} \label{proposition_mu_varphi}
        There exists a number $\mu_\varphi \geq 0$ such that, for any $\vartheta \in [0, T)$ and any $x(\cdot) \in \AC^\alpha([0, T], \mathbb{R}^n)$, the functional $\varphi$ from \eqref{varphi_main} satisfies the inequality
        \begin{equation*}
            \begin{split}
                & | \varphi(T, x(\cdot)) - \varphi(\vartheta, x_{\vartheta}(\cdot)) | \\
                & \leq (T - \vartheta)^{2 \alpha - 1} \mu_\varphi
                \Bigl( \|x(\cdot)\|_{\C([0, T], \mathbb{R}^n)} + \esssup{\tau \in [\vartheta, T]} \|(^C D^\alpha x)(\tau)\|_{\mathbb{R}^n} \Bigr)
                \|x(\cdot)\|_{\C([0, T], \mathbb{R}^n)}.
            \end{split}
        \end{equation*}
    \end{proposition}
    \begin{proof}
        Let $\vartheta \in [0, T)$ and $x(\cdot) \in \AC^\alpha([0, T], \mathbb{R}^n)$ be fixed.
        Let us denote
        \begin{equation*}
            \mu_x
            \doteq \|x(\cdot)\|_{\C([0, T], \mathbb{R}^n)},
            \quad \hat{\mu}_x^\ast
            \doteq \esssup{\tau \in [\vartheta, T]} \|(^C D^\alpha x)(\tau)\|_{\mathbb{R}^n}
        \end{equation*}
        for brevity.
        Then, recalling the notation introduced in \eqref{mu_A_mu_B_mu_Q} and \eqref{mu_Phi_mu_M} and using estimates \eqref{mu_b}, \eqref{mu_c}, and \eqref{R^-1_estimate}, we derive
        \begin{equation*}
            \begin{split}
                & |\varphi(T, x(\cdot)) - \varphi(\vartheta, x_\vartheta(\cdot))| \\
                & \leq \bigl| \langle b(T \mid T, x(\cdot)), P b(T \mid T, x(\cdot)) \rangle_{\mathbb{R}^n}
                - \langle b(T \mid \vartheta, x_\vartheta(\cdot)), P b(T \mid \vartheta, x_\vartheta(\cdot)) \rangle_{\mathbb{R}^n} \bigr| \\
                & \quad + \biggl| \int_{\vartheta}^{T} \langle b(\tau \mid \vartheta, x_\vartheta(\cdot)), Q(\tau) b(\tau \mid \vartheta, x_\vartheta(\cdot)) \rangle_{\mathbb{R}^n} \, \rd \tau \biggr| \\
                & \quad + \biggl| \int_{\vartheta}^{T} \frac{\langle c(\tau \mid \vartheta, x_\vartheta(\cdot)), R(\tau)^{- 1} c(\tau \mid \vartheta, x_\vartheta(\cdot)) \rangle_{\mathbb{R}^m}}{(T - \tau)^{2 - 2 \alpha}} \, \rd \tau \biggr| \\
                & \quad + \biggl| \int_{\vartheta}^{T} \biggl\langle \frac{c(\tau \mid \vartheta, x_\vartheta(\cdot))}{(T - \tau)^{2 - 2 \alpha}},
                \int_{\vartheta}^{T} \frac{M(\tau, \xi \mid \vartheta) c(\xi \mid \vartheta, x_\vartheta(\cdot))}{(T - \xi)^{2 - 2 \alpha}} \, \rd \xi \biggr\rangle_{\mathbb{R}^m} \, \rd \tau \biggr| \\
                & \leq 2 \| b(T \mid T, x(\cdot)) - b(T \mid \vartheta, x_\vartheta(\cdot))\|_{\mathbb{R}^n}
                \|P\|_{\mathbb{R}^{n \times n}} \mu_b \mu_x \\
                & \quad + (T - \vartheta) \mu_b^2 \mu_Q \mu_x^2
                + \frac{(T - \vartheta)^{2 \alpha - 1} \mu_c^2 \mu_x^2}{(2 \alpha - 1) \theta}
                + \frac{(T - \vartheta)^{4 \alpha - 2} \mu_c^2 \mu_M \mu_x^2}{(2 \alpha - 1)^2}.
            \end{split}
        \end{equation*}
        According to \eqref{a_definition}, we have
        \begin{equation*}
            \|a(T \mid T, x(\cdot)) - a(T \mid \vartheta, x_\vartheta(\cdot))\|_{\mathbb{R}^n}
            = \frac{1}{\Gamma(\alpha)} \biggl\| \int_{\vartheta}^{T} \frac{(^C D^\alpha x)(\xi)}{(T - \xi)^{1 - \alpha}} \, \rd \xi \biggr\|_{\mathbb{R}^n}
            \leq \frac{(T - \vartheta)^\alpha \hat{\mu}_x^\ast}{\Gamma(\alpha + 1)},
        \end{equation*}
        and, therefore, in view of \eqref{b_definition} and \eqref{a_estimate}, we get
        \begin{equation*}
            \|b(T \mid T, x(\cdot)) - b(T \mid \vartheta, x_\vartheta(\cdot))\|_{\mathbb{R}^n}
            \leq \frac{(T - \vartheta)^\alpha \hat{\mu}_x^\ast}{\Gamma(\alpha + 1)}
            + \frac{(T - \vartheta)^\alpha \mu_\Phi \mu_A \mu_x}{\alpha}.
        \end{equation*}
        The obtained estimates imply the result.
    \end{proof}

    \begin{proof}[Proof of Theorem \ref{theorem_optimal_control}]
        Let $(t, w(\cdot)) \in \mathcal{G}_0$ be fixed.
        Let us consider an arbitrary control $u(\cdot) \in \Linf([t, T], \mathbb{R}^m)$ and the corresponding motion $x(\cdot) \doteq x(\cdot \mid t, w(\cdot), u(\cdot))$ of system \eqref{differential_equation}.
        Let us introduce the function
        \begin{equation} \label{kappa}
            \kappa(\tau)
            \doteq \varphi(\tau, x_\tau(\cdot))
            - \int_{\tau}^{T} \bigl( \langle x(\xi), Q(\xi) x(\xi) \rangle_{\mathbb{R}^n}
            + \langle u(\xi), R(\xi) u(\xi) \rangle_{\mathbb{R}^m} \bigr) \, \rd \xi
        \end{equation}
        for all $\tau \in [t, T]$.
        Since the functional $\varphi$ is $ci$-smooth of order $\alpha$ by Theorem \ref{theorem_solution} and the inclusion $x(\cdot) \in \mathcal{X}(t, w(\cdot))$ holds by definition, it follows from \cite[Lemma 9.2]{Gomoyunov_2020_SIAM} that the function $\kappa(\cdot)$ is continuous on $[t, T]$ and Lipschitz continuous on $[t, \vartheta]$ for every $\vartheta \in (t, T)$ and
        \begin{equation} \label{kappa_dot}
            \begin{split}
                \dot{\kappa}(\tau)
                & = \partial_t^\alpha \varphi(\tau, x_\tau(\cdot)) + \langle \nabla^\alpha \varphi(\tau, x_\tau(\cdot)), (^C D^\alpha x)(\tau) \rangle_{\mathbb{R}^n} \\
                & \quad + \langle x(\tau), Q(\tau) x(\tau) \rangle_{\mathbb{R}^n}
                + \langle u(\tau), R(\tau) u(\tau) \rangle_{\mathbb{R}^m} \\
                & = \partial_t^\alpha \varphi(\tau, x_\tau(\cdot))
                + h \bigl( \tau, x(\tau), u(\tau), \nabla^\alpha \varphi(\tau, x_\tau(\cdot)) \bigr)
                \quad \text{for a.e. } \tau \in [t, T],
            \end{split}
        \end{equation}
        where $\dot{\kappa}(\tau) \doteq \rd \kappa(\tau) / \rd \tau$ and notation \eqref{pre_Hamiltonian} is used.
        Then, recalling definition \eqref{Hamiltonian} of the Hamiltonian and the fact that the functional $\varphi$ satisfies the HJB equation \eqref{HJB} by Theorem \ref{theorem_solution}, we get
        \begin{equation*}
            \dot{\kappa}(\tau)
            \geq \partial_t^\alpha \varphi(\tau, x_\tau(\cdot)) + H \bigl(\tau, x(\tau), \nabla^\alpha \varphi(\tau, x_\tau(\cdot)) \bigr)
            = 0
            \quad \text{for a.e. } \tau \in [t, T].
        \end{equation*}
        Therefore, we obtain $\kappa(\vartheta) \geq \kappa(t)$ for every $\vartheta \in (t, T)$, which yields $\kappa(T) \geq \kappa(t)$.
        Hence, and since the functional $\varphi$ satisfies the boundary condition \eqref{boundary_condition} by Theorem \ref{theorem_solution}, we derive $J(t, w(\cdot), u(\cdot)) \geq \varphi(t, w(\cdot))$.
        As a result, taking the infimum over all $u(\cdot) \in \Linf([t, T], \mathbb{R}^m)$, we conclude that (see \eqref{value})
        \begin{equation} \label{rho_geq_varphi}
            \rho(t, w(\cdot))
            \geq \varphi(t, w(\cdot)).
        \end{equation}

        Now, let us fix $\varepsilon > 0$ and show that
        \begin{equation} \label{rho_leq_varphi}
            \rho(t, w(\cdot))
            \leq \varphi(t, w(\cdot)) + \varepsilon.
        \end{equation}
        Below, we will use the numbers $\mu_A$, $\mu_B$, and $\mu_Q$ from \eqref{mu_A_mu_B_mu_Q} as well as the numbers $\mu_{U^\circ}$ and $\mu_\varphi$ from Propositions \ref{proposition_mu_U} and \ref{proposition_mu_varphi}, respectively.
        Let us denote
        \begin{equation*}
            \mu^\circ
            \doteq T^{1 - \alpha} \mu_A + \mu_B \mu_{U^\circ},
            \quad \mu_x^\circ
            \doteq \mathrm{E}_{2 \alpha - 1}(\mu^\circ T^{2 \alpha - 1}) \|w(\cdot)\|_{\C([0, t], \mathbb{R}^n)},
        \end{equation*}
        where $\mathrm{E}_{2 \alpha - 1}(\cdot)$ is the Mittag-Leffler function with the exponent $2 \alpha - 1$ (see, e.g., \cite[equation (1.90)]{Samko_Kilbas_Marichev_1993}).
        Then, let us choose a number $\vartheta^\circ \in (t, T)$ from the condition
        \begin{equation*}
            (T - \vartheta^\circ)^{2 \alpha - 1} \mu_\varphi (1 + \mu_A) (\mu_x^\circ)^2
            + (T - \vartheta^\circ) \mu_Q (\mu_x^\circ)^2
            \leq \frac{\varepsilon}{2}
        \end{equation*}
        and fix $\vartheta \in [\vartheta^\circ, T)$.
        Further, let us consider the set
        \begin{equation*}
            \mathcal{X}_\ast
            \doteq \biggl\{ x(\cdot) \in \mathcal{X}(t, w(\cdot)) \colon
            \, \|(^C D^\alpha x)(\tau)\|_{\mathbb{R}^n}
            \leq \frac{\mu^\circ \mu_x^\circ}{(T - \vartheta)^{1 - \alpha}}
            \text{ for a.e. } \tau \in [t, T] \biggr\}.
        \end{equation*}
        According to, e.g., \cite[Theorem 1]{Gomoyunov_2020_DE}, the set $\mathcal{X}_\ast$ is compact in $\AC^\alpha([0, T], \mathbb{R}^n)$.
        Since the mappings $\partial_t^\alpha \varphi \colon \mathcal{G}_0 \to \mathbb{R}$ and $\nabla^\alpha \varphi \colon \mathcal{G}_0 \to \mathbb{R}$ are continuous by Theorem \ref{theorem_solution} and the function $h \colon [0, T] \times \mathbb{R}^n \times \mathbb{R}^m \times \mathbb{R}^n \to \mathbb{R}$ defined by \eqref{pre_Hamiltonian} is continuous due to the imposed assumptions, there exists a modulus of continuity $\omega^\circ(\cdot)$ such that, for any $\tau$, $\tau^\prime \in [t, \vartheta]$, any $x(\cdot) \in \mathcal{X}_\ast$, and any $u \in \mathbb{R}^m$ with $\|u\|_{\mathbb{R}^m} \leq \mu_{U^\circ} \mu_x^\circ / (T - \vartheta)^{1 - \alpha}$,
        \begin{align*}
            & |\partial_t^\alpha \varphi(\tau^\prime, x_{\tau^\prime}(\cdot)) - \partial_t^\alpha \varphi(\tau, x_\tau(\cdot)) | \\
            & + \bigl| h \bigl( \tau^\prime, x(\tau^\prime), u, \nabla^\alpha \varphi(\tau^\prime, x_{\tau^\prime}(\cdot)) \bigr)
            - h \bigl( \tau, x(\tau), u, \nabla^\alpha \varphi(\tau, x_\tau(\cdot)) \bigr) \bigr|
            \leq \omega^\circ(|\tau^\prime - \tau|).
        \end{align*}
        Then, let us choose a number $\delta > 0$ satisfying the inequality
        \begin{equation*}
            \omega^\circ(\delta)
            \leq \frac{\varepsilon}{2 (\vartheta - t)}
        \end{equation*}
        and take a partition $\Delta \doteq \{\tau_j\}_{j \in \overline{1, \ell + 1}}$ of the interval $[t, \vartheta]$ with $\diam(\Delta) \leq \delta$.

        Now, let $u^\circ(\cdot) \doteq u^\circ(\cdot \mid t, w(\cdot), \vartheta, \Delta)$ and $x^\circ(\cdot) \doteq x(\cdot \mid t, w(\cdot), u^\circ(\cdot))$ be the control and the motion of system \eqref{differential_equation} generated by the rules \eqref{control_law} and \eqref{control_law_0}.
        Similarly to \eqref{kappa}, let us introduce the corresponding function
        \begin{equation*}
            \kappa^\circ(\tau)
            \doteq \varphi(\tau, x^\circ_\tau(\cdot))
            - \int_{\tau}^{T} \bigl( \langle x^\circ(\xi), Q(\xi) x^\circ(\xi) \rangle_{\mathbb{R}^n}
            + \langle u^\circ(\xi), R(\xi) u^\circ(\xi) \rangle_{\mathbb{R}^m} \bigr) \, \rd \xi
        \end{equation*}
        for all $\tau \in [t, T]$.

        In view of \eqref{control_law}, we obtain
        \begin{equation} \label{derivative_estimate}
            \begin{split}
                \|(^C D^\alpha x^\circ)(\tau)\|_{\mathbb{R}^n}
                & \leq \mu_A \|x^\circ(\tau)\|_{\mathbb{R}^n}
                + \frac{\mu_B \mu_{U^\circ} \|x^\circ_{\tau_j}(\cdot)\|_{\C([0, \tau_j], \mathbb{R}^n)}}{(T - \tau_j)^{1 - \alpha}} \\
                & \leq \frac{\mu^\circ \|x^\circ_\tau(\cdot)\|_{\C([0, \tau], \mathbb{R}^n)}}{(T - \tau)^{1 - \alpha}}
                \quad \text{for a.e. } \tau \in [\tau_j, \tau_{j + 1}) \quad \forall j \in \overline{1, \ell}.
            \end{split}
        \end{equation}
        In addition, due to \eqref{control_law_0}, we have
        \begin{equation} \label{derivative_estimate_0}
            \|(^C D^\alpha x^\circ)(\tau)\|_{\mathbb{R}^n}
            \leq \mu_A \|x^\circ(\tau)\|_{\mathbb{R}^n}
            \leq \frac{\mu^\circ \|x^\circ_\tau(\cdot)\|_{\C([0, \tau], \mathbb{R}^n)}}{(T - \tau)^{1 - \alpha}}
            \quad \text{for a.e. } \tau \in [\vartheta, T].
        \end{equation}
        Hence, according to \eqref{I_D}, \eqref{a_definition}, and \eqref{a_estimate}, we derive
        \begin{equation*}
            \begin{split}
                \|x^\circ(\tau)\|_{\mathbb{R}^n}
                & = \biggl\| a(\tau \mid t, w(\cdot))
                + \frac{1}{\Gamma(\alpha)} \int_{t}^{\tau} \frac{(^C D^\alpha x^\circ)(\xi)}{(\tau - \xi)^{1 - \alpha}} \, \rd \xi \biggr\|_{\mathbb{R}^n} \\
                & \leq \|w(\cdot)\|_{\C([0, t], \mathbb{R}^n)}
                + \frac{\mu^\circ}{\Gamma(\alpha)} \int_{t}^{\tau} \frac{\|x^\circ_\xi(\cdot)\|_{\C([0, \xi], \mathbb{R}^n)}}
                {(T - \xi)^{1 - \alpha} (\tau - \xi)^{1 - \alpha}} \, \rd \xi \\
                & \leq \|w(\cdot)\|_{\C([0, t], \mathbb{R}^n)}
                + \frac{\mu^\circ}{\Gamma(\alpha)} \int_{t}^{\tau} \frac{\|x^\circ_\xi(\cdot)\|_{\C([0, \xi], \mathbb{R}^n)}}{(\tau - \xi)^{2 - 2 \alpha}} \, \rd \xi
                \quad \forall \tau \in [t, T].
            \end{split}
        \end{equation*}
        Therefore, taking \cite[Proposition 2]{Gomoyunov_2019_PFDA} into account, we get
        \begin{equation*}
            \|x^\circ_\tau(\cdot)\|_{\C([0, \tau], \mathbb{R}^n)}
            \leq \|w(\cdot)\|_{\C([0, t], \mathbb{R}^n)}
            + \frac{\mu^\circ}{\Gamma(\alpha)} \int_{0}^{\tau} \frac{\|x^\circ_\xi(\cdot)\|_{\C([0, \xi], \mathbb{R}^n)}}
            {(\tau - \xi)^{2 - 2 \alpha}} \, \rd \xi
            \quad \forall \tau \in [0, T].
        \end{equation*}
        Thus, applying the fractional Gronwall-type inequality (see, e.g., \cite[Lemma 6.19]{Diethelm_2010}), we conclude that
        \begin{equation*}
            \|x^\circ_\tau(\cdot)\|_{\C([0, \tau], \mathbb{R}^n)}
            \leq \mathrm{E}_{2 \alpha - 1}(\mu^\circ \tau^{2 \alpha - 1}) \|w(\cdot)\|_{\C([0, t], \mathbb{R}^n)}
            \leq \mu_x^\circ
            \quad \forall \tau \in [0, T].
        \end{equation*}

        Consequently, we have (see \eqref{derivative_estimate_0})
        \begin{equation*}
            \|(^C D^\alpha x^\circ)(\tau)\|_{\mathbb{R}^n}
            \leq \mu_A \mu_x^\circ
            \quad \text{for a.e. } \tau \in [\vartheta, T],
        \end{equation*}
        wherefrom, due to the choice of the number $\vartheta^\circ$ and equality \eqref{control_law_0}, it follows that
        \begin{equation} \label{kappa^0_T_vartheta}
            |\kappa^\circ(T) - \kappa^\circ(\vartheta)|
            \leq \frac{\varepsilon}{2}.
        \end{equation}
        Moreover, we derive (see \eqref{derivative_estimate} and \eqref{derivative_estimate_0})
        \begin{equation*}
            \|(^C D^\alpha x^\circ)(\tau)\|_{\mathbb{R}^n}
            \leq \frac{\mu^\circ \mu_x^\circ}{(T - \vartheta)^{1 - \alpha}}
            \quad \text{for a.e. } \tau \in [t, T],
        \end{equation*}
        which implies the inclusion $x^\circ(\cdot) \in \mathcal{X}_\ast$.
        In addition, we get
        \begin{equation*}
            \|U^\circ(\tau_j, x^\circ_{\tau_j}(\cdot))\|_{\mathbb{R}^m}
            \leq \frac{\mu_{U^\circ} \mu_x^\circ}{(T - \vartheta)^{1 - \alpha}}
            \quad \forall j \in \overline{1, \ell}.
        \end{equation*}

        Then, similarly to \eqref{kappa_dot}, for every $j \in \overline{1, \ell}$ and a.e. $\tau \in [\tau_j, \tau_{j + 1})$, using equalities \eqref{U^0_corollary} and \eqref{control_law}, the fact that the functional $\varphi$ satisfies the HJB equation \eqref{HJB}, and the choice of the number $\delta$, we derive
        \begin{align*}
            \dot{\kappa}^\circ(\tau)
            & = \partial_t^\alpha \varphi(\tau, x^\circ_\tau(\cdot))
            + h \bigl( \tau, x^\circ(\tau), U^\circ(\tau_j, x^\circ_{\tau_j}(\cdot)), \nabla^\alpha \varphi(\tau, x^\circ_\tau(\cdot)) \bigr) \\
            & \leq \partial_t^\alpha \varphi(\tau_j, x^\circ_{\tau_j}(\cdot))
            + h \bigl( \tau_j, x^\circ(\tau_j), U^\circ(\tau_j, x^\circ_{\tau_j}(\cdot)), \nabla^\alpha \varphi(\tau_j, x^\circ_{\tau_j}(\cdot)) \bigr)
            + \omega^\circ(\tau - \tau_j) \\
            & \leq \partial_t^\alpha \varphi(\tau_j, x^\circ_{\tau_j}(\cdot))
            + H \bigl( \tau_j, x^\circ(\tau_j), \nabla^\alpha \varphi(\tau_j, x^\circ_{\tau_j}(\cdot)) \bigr)
            + \omega^\circ(\delta) \\
            & \leq \frac{\varepsilon}{2 (\vartheta - t)}.
        \end{align*}
        Hence, we have $\kappa^\circ(\vartheta) \leq \kappa^\circ(t) + \varepsilon / 2$, and, therefore, $\kappa^\circ(T) \leq \kappa^\circ(t) + \varepsilon$ by \eqref{kappa^0_T_vartheta}.
        Since the functional $\varphi$ satisfies the boundary condition \eqref{boundary_condition}, we obtain the estimate
        \begin{equation*}
            J(t, w(\cdot), u^\circ(\cdot))
            \leq \varphi(t, w(\cdot)) + \varepsilon,
        \end{equation*}
        which yields \eqref{rho_leq_varphi} due to \eqref{value}.

        According to \eqref{rho_geq_varphi} and \eqref{rho_leq_varphi}, we get $\rho(t, w(\cdot)) = \varphi (t, w(\cdot))$ for all $(t, w(\cdot)) \in \mathcal{G}_0$.
        Hence, observing that the value functional $\rho$ satisfies the boundary condition \eqref{boundary_condition} by definition \eqref{value}, we conclude that the functionals $\rho$ and $\varphi$ coincide.
        It remains to note that the second part of the theorem follows from the presented proof of inequality \eqref{rho_leq_varphi}.
        The theorem is proved.
    \end{proof}

\section{Conclusion}
\label{section_conclusion}

    In this paper, we have considered the linear-quadratic optimal control problem for the dynamical system described by the linear Caputo fractional differential equation \eqref{differential_equation} and the quadratic cost functional \eqref{cost_functional}.
    We have introduced the associated Cauchy problem for the Hamilton--Jacobi--Bellman equation with fractional $ci$-derivatives \eqref{HJB} and the right-end boundary condition \eqref{boundary_condition}.
    We have proved that the value functional $\rho$ of the optimal control problem is a solution of the Cauchy problem and is given by the explicit formula \eqref{varphi_main}.
    Moreover, we have proposed a practically realizable feedback control procedure \eqref{control_law} and \eqref{control_law_0} that generates $\varepsilon$-optimal controls with any predetermined accuracy $\varepsilon > 0$.

    In particular, we have shown that, along with the methods developed in \cite{Han_Lin_Yong_2021}, the optimal control problem under consideration can also be effectively solved using the formalism of Hamilton--Jacobi--Bellman equations with fractional $ci$-derivatives proposed in \cite{Gomoyunov_2020_SIAM}.
    The latter approach has its own advantages, primarily related to the fact that it allows us to obtain explicit formulas for the optimal positional control strategy \eqref{U^0} and justify the possibility of using it in a time-discrete scheme.

    Finally, let us note that it is of interest to further develop the results of the paper, for example, in the following directions:

    (i)
        to elaborate the corresponding numerical methods for solving the optimal control problem under consideration;

    (ii)
        to investigate the problem in the class of square integrable controls and justify the optimality of the constructed control strategy \eqref{U^0} within this formulation;

    (iii)
        to study ``time-continuous'' feedback control schemes, when a control strategy is substituted directly into the right-hand side of the dynamic equation \eqref{differential_equation};

    (iv)
        to apply the technique developed in the paper in order to solve linear-quadratic differential games in fractional-order systems.

\appendix

\section{Proofs of Section \ref{subsection_preliminaries}}
\label{appendix}

    \begin{proof}[Proof of Proposition \ref{proposition_K}]
        First of all, note that equality \eqref{K_symmetric} follows directly from definition \eqref{K_definition} of the function $K(\cdot, \cdot)$ and the assumption that the matrices $P$ and $Q(\tau)$ for all $\tau \in [0, T]$ are symmetric.
        Further, since the functions $B(\cdot)$ and $\Phi(\cdot, \cdot)$ are continuous, in order to show that the function $K(\cdot, \cdot)$ is continuous, it suffices to prove continuity of the function
        \begin{equation*}
            \hat{K}(\tau, \xi)
            \doteq \int_{\tau \vee \xi}^{T} \frac{\Phi(\eta, \tau)^\top Q(\eta) \Phi(\eta, \xi)}{(\eta - \tau)^{1 - \alpha} (\eta - \xi)^{1 - \alpha}} \, \rd \eta
            \quad \forall \tau, \xi \in [0, T].
        \end{equation*}

        Recalling that the function $Q(\cdot)$ is continuous, let us take a number $\hat{\mu} \geq 0$ and a modulus of continuity $\hat{\omega}(\cdot)$ such that
        \begin{equation*}
            \|\Phi(\eta, \tau)^\top Q(\eta) \Phi(\eta, \xi) \|_{\mathbb{R}^{n \times n}}
            \leq \hat{\mu}
            \quad \forall \tau, \xi \in [0, T] \quad \forall \eta \in [\tau \vee \xi, T]
        \end{equation*}
        and
        \begin{equation*}
            \|\Phi(\eta, \tau^\prime)^\top Q(\eta) \Phi(\eta, \xi^\prime)
            - \Phi(\eta, \tau)^\top Q(\eta) \Phi(\eta, \xi) \|_{\mathbb{R}^{n \times n}}
            \leq \hat{\omega}(|\tau^\prime - \tau| + |\xi^\prime - \xi|)
        \end{equation*}
        for all $\tau$, $\tau^\prime$, $\xi$, $\xi^\prime \in [0, T]$ and all $\eta \in [\tau \vee \tau^\prime \vee \xi \vee \xi^\prime, T]$.

        Then, for any $\tau$, $\tau^\prime$, $\xi$, $\xi^\prime \in [0, T]$ with $\tau^\prime \vee \xi^\prime \geq \tau \vee \xi$, taking Proposition \ref{proposition_technical_continuity_2} into account (see also Remark \ref{remark_appendix}), we derive
        \begin{align*}
            & \|\hat{K}(\tau^\prime, \xi^\prime) - \hat{K}(\tau, \xi)\|_{\mathbb{R}^{n \times n}} \\*
            & \leq \int_{\tau^\prime \vee \xi^\prime}^{T}
            \frac{\|\Phi(\eta, \tau^\prime)^\top Q(\eta) \Phi(\eta, \xi^\prime) - \Phi(\eta, \tau)^\top Q(\eta) \Phi(\eta, \xi)\|_{\mathbb{R}^{n \times n}}}{(\eta - \tau^\prime)^{1 - \alpha} (\eta - \xi^\prime)^{1 - \alpha}} \, \rd \eta \\*
            & \quad + \biggl\| \int_{\tau^\prime \vee \xi^\prime}^{T} \frac{\Phi(\eta, \tau)^\top Q(\eta) \Phi(\eta, \xi)}
            {(\eta - \tau^\prime)^{1 - \alpha} (\eta - \xi^\prime)^{1 - \alpha}} \, \rd \eta
            - \int_{\tau \vee \xi}^{T} \frac{\Phi(\eta, \tau)^\top Q(\eta) \Phi(\eta, \xi)}
            {(\eta - \tau)^{1 - \alpha} (\eta - \xi)^{1 - \alpha}} \, \rd \eta \biggr\|_{\mathbb{R}^{n \times n}} \\
            & \leq \frac{T^{2 \alpha - 1} \hat{\omega}(|\tau^\prime - \tau| + |\xi^\prime - \xi|)}{2 \alpha - 1}
            + \hat{\mu} \bar{\omega}_{1 - \alpha}^{(3)}(|\tau^\prime - \tau| + |\xi^\prime - \xi|),
        \end{align*}
        which implies continuity of the function $\hat{K}(\cdot, \cdot)$.
    \end{proof}

    The proof of Proposition \ref{proposition_M} requires some preliminary considerations.

    Let $t \in [0, T)$ be given.
    Let us denote by $\L2 ([t, T], \mathbb{R}^m)$ the Hilbert space of all (classes of equivalence of) measurable and norm square integrable functions $\f2 \colon [t, T] \to \mathbb{R}^m$ with the inner product
    \begin{equation*}
        \langle \f2(\cdot), \g2(\cdot) \rangle_{\L2([t, T], \mathbb{R}^m)}
        \doteq \int_{t}^{T} \langle \f2(\tau), \g2(\tau) \rangle_{\mathbb{R}^m} \, \rd \tau
        \quad \forall \f2(\cdot), \g2(\cdot) \in \L2([t, T], \mathbb{R}^m)
    \end{equation*}
    and the corresponding norm
    \begin{equation*}
        \|\f2(\cdot)\|_{\L2([t, T], \mathbb{R}^m)}
        \doteq \biggl( \int_{t}^{T} \|\f2(\tau)\|_{\mathbb{R}^m}^2 \, \rd \tau \biggr)^{1 / 2}
        \quad \forall \f2(\cdot) \in \L2([t, T], \mathbb{R}^m).
    \end{equation*}
    Let us consider a linear operator $\A2_{t} \colon \L2([t, T], \mathbb{R}^m) \to \L2([t, T], \mathbb{R}^m)$ defined by
    \begin{equation} \label{A_definition}
        (\A2_t \f2)(\tau)
        \doteq R(\tau) \f2(\tau)
        + \frac{1}{(T - \tau)^{1 - \alpha}} \int_{t}^{T} \frac{K(\tau, \eta) \f2(\eta)}{(T - \eta)^{1 - \alpha}} \, \rd \eta
    \end{equation}
    for a.e. $\tau \in [t, T]$ and every $\f2(\cdot) \in \L2([t, T], \mathbb{R}^m)$.
    By virtue of the properties of the function $K(\cdot, \cdot)$ established in Proposition \ref{proposition_K}, and since the function $R(\cdot)$ is continuous and the matrices $R(\tau)$ for all $\tau \in [0, T]$ are symmetric, we conclude that the operator $\A2_t$ is well-defined, bounded, and self-adjoint (see, e.g., \cite[Examples 4.2.4 and 4.4.6]{Debnath_Mikusinski_2005} and \cite[Ch. 7, \S 15, Example B and Ch. 9, \S 8, Theorem 1]{Zaanen_1953}).
    In addition, we note that, for any $\f2(\cdot) \in \L2([t, T], \mathbb{R}^m)$, the inequality
    \begin{equation*}
        \langle (\A2_t \f2)(\cdot), \f2(\cdot) \rangle_{\L2([t, T], \mathbb{R}^m)}
        \geq \theta \|\f2(\cdot)\|_{\L2([t, T], \mathbb{R}^m)}^2
    \end{equation*}
    is valid, where the number $\theta$ is taken from \eqref{theta}.
    Indeed, we have
    \begin{equation*}
        \begin{split}
            & \langle (\A2_t \f2)(\cdot), \f2(\cdot) \rangle_{\L2([t, T], \mathbb{R}^m)} \\
            & \geq \theta \int_{t}^{T} \|\f2(\tau)\|_{\mathbb{R}^m}^2 \, \rd \tau
            + \int_{t}^{T} \frac{1}{(T - \tau)^{1 - \alpha}} \biggl\langle \int_{t}^{T} \frac{K(\tau, \eta) \f2(\eta)}{(T - \eta)^{1 - \alpha}} \, \rd \eta,
            \f2(\tau) \biggr\rangle_{\mathbb{R}^m} \, \rd \tau,
        \end{split}
    \end{equation*}
    and, for the second term of the obtained estimate, in view of definition \eqref{K_definition} of the function $K(\cdot, \cdot)$ and recalling that the matrices $P$ and $Q(\tau)$ for all $\tau \in [0, T]$ are positive semi-definite, we derive
    \begin{equation*}
        \begin{split}
            & \int_{t}^{T} \frac{1}{(T - \tau)^{1 - \alpha}} \biggl\langle \int_{t}^{T} \frac{K(\tau, \eta) \f2(\eta)}{(T - \eta)^{1 - \alpha}} \, \rd \eta,
            \f2(\tau) \biggr\rangle_{\mathbb{R}^m} \, \rd \tau \\
            & = \biggl\langle \int_{t}^{T} \frac{\Phi(T, \tau) B(\tau) \f2(\tau)}{(T - \tau)^{1 - \alpha}} \, \rd \tau,
            P \int_{t}^{T} \frac{\Phi(T, \tau) B(\tau) \f2(\tau)}{(T - \tau)^{1 - \alpha}} \, \rd \tau \biggr\rangle_{\mathbb{R}^m} \\
            & \quad + \int_{t}^{T} \biggl\langle
            \int_{t}^{\eta} \frac{\Phi(\eta, \tau) B(\tau) \f2(\tau)}{(\eta - \tau)^{1 - \alpha}} \, \rd \tau,
            Q(\eta) \int_{t}^{\eta} \frac{\Phi(\eta, \tau) B(\tau) \f2(\tau)}{(\eta - \tau)^{1 - \alpha}} \, \rd \tau \biggr\rangle_{\mathbb{R}^m} \, \rd \eta \\
            & \geq 0.
        \end{split}
    \end{equation*}
    In particular, we conclude that the operator $\A2_{t}$ is invertible and the inverse operator $\A2_{t}^{- 1} \colon \L2([t, T], \mathbb{R}^m) \to \L2([t, T], \mathbb{R}^m)$ admits the estimate
    \begin{equation} \label{norm_of_inverse}
        \|(\A2_{t}^{- 1} \g2)(\cdot)\|_{\L2([t, T], \mathbb{R}^m)}
        \leq \frac{1}{\theta} \|\g2(\cdot)\|_{\L2([t, T], \mathbb{R}^m)}
        \quad \forall \g2(\cdot) \in \L2([t, T], \mathbb{R}^m)
    \end{equation}
    and is self-adjoint (see, e.g., \cite[Theorem 4.6.11]{Debnath_Mikusinski_2005}).

    The properties of the operator $\A2_{t}$ described above allow us to get
    \begin{proposition} \label{proposition_existence_of_a_continuous_solution}
        For any $t \in [0, T]$ and any $g(\cdot) \in \C([t, T], \mathbb{R}^m)$, there exists a unique solution $f(\cdot) \in \C([t, T], \mathbb{R}^m)$ of the Fredholm integral equation
        \begin{equation} \label{proposition_existence_of_a_continuous_solution_main}
            R(\tau) f(\tau) + \int_{t}^{T} \frac{K(\tau, \eta) f(\eta)}{(T - \eta)^{2 - 2 \alpha}} \, \rd \eta
            = g(\tau)
            \quad \forall \tau \in [t, T].
        \end{equation}
        In addition, there exist a number $\mu_{\A2} \geq 0$ and a modulus of continuity $\omega_{\A2}(\cdot)$ such that, for any $t \in [0, T]$ and any $g(\cdot) \in \C([t, T], \mathbb{R}^m)$, the solution $f(\cdot) \in \C([t, T], \mathbb{R}^m)$ of equation
        \eqref{proposition_existence_of_a_continuous_solution_main} satisfies the estimates
        \begin{align}
            \label{proposition_existence_of_a_continuous_solution_main_2}
            \|f(\tau)\|_{\mathbb{R}^m}
            & \leq \mu_{\A2} \|g(\cdot)\|_{\C([t, T], \mathbb{R}^m)}, \\
            \label{proposition_existence_of_a_continuous_solution_main_3}
            \|f(\tau^\prime) - f(\tau)\|_{\mathbb{R}^m}
            & \leq \frac{\omega_g(|\tau^\prime - \tau|)}{\theta} + \|g(\cdot)\|_{\C([t, T], \mathbb{R}^m)} \omega_{\A2}(|\tau^\prime - \tau|)
        \end{align}
        for all $\tau$, $\tau^\prime \in [t, T]$, where $\theta$ is the number from \eqref{theta} and $\omega_g(\cdot)$ is the modulus of continuity of the function $g(\cdot)$.
    \end{proposition}
    \begin{proof}
        Recalling that the functions $R(\cdot)$ and $K(\cdot, \cdot)$ are continuous, we denote their moduli of continuity by $\omega_R(\cdot)$ and  $\omega_K(\cdot)$, respectively, and put
        \begin{equation} \label{mu_K_mu_R}
            \mu_R
            \doteq \max_{\tau \in [0, T]} \|R(\tau)\|_{\mathbb{R}^{m \times m}},
            \quad \mu_K
            \doteq \max_{\tau, \xi \in [0, T]} \|K(\tau, \xi)\|_{\mathbb{R}^{m \times m}}.
        \end{equation}

        Let $t \in [0, T)$ and $g(\cdot) \in \C([t, T], \mathbb{R}^m)$ be fixed.
        Let us take $\g2(\cdot) \in \L2([t, T], \mathbb{R}^m)$ satisfying the equality $\g2(\tau) = g(\tau) / (T - \tau)^{1 - \alpha}$ for a.e. $\tau \in [t, T]$ and then define $\f2(\cdot) \doteq (\A2_{t}^{- 1} \g2)(\cdot) \in \L2([t, T], \mathbb{R}^m)$.
        By definition \eqref{A_definition} of the operator $\A2_{t}$, we have
        \begin{equation} \label{equality_f_ast}
            R(\tau) \f2(\tau) + \frac{1}{(T - \tau)^{1 - \alpha}} \int_{t}^{T} \frac{K(\tau, \eta) \f2(\eta)}{(T - \eta)^{1 - \alpha}} \, \rd \eta
            = \g2(\tau)
            \quad \text{for a.e. } \tau \in [t, T],
        \end{equation}
        and, using \eqref{norm_of_inverse} and denoting $\mu_g \doteq \|g(\cdot)\|_{\C([t, T], \mathbb{R}^m)}$ for brevity, we obtain
        \begin{equation} \label{estimate_norm_f_ast}
            \|\f2(\cdot)\|_{\L2([t, T], \mathbb{R}^m)}
            \leq \frac{\|\g2(\cdot)\|_{\L2([t, T], \mathbb{R}^m)}}{\theta}
            \leq \frac{T^{\alpha - 1 / 2} \mu_g}{\theta (2 \alpha - 1)^{1 / 2}}.
        \end{equation}
        From \eqref{equality_f_ast} and \eqref{estimate_norm_f_ast}, applying estimate \eqref{R^-1_estimate} and the Cauchy--Schwarz inequality, we derive
        \begin{equation} \label{estimate_f_ast}
            \begin{split}
                & \|\f2(\tau)\|_{\mathbb{R}^m} \\
                & \leq \|R(\tau)^{- 1}\|_{\mathbb{R}^m}
                \biggl( \|\g2(\tau)\|_{\mathbb{R}^m}
                + \frac{1}{(T - \tau)^{1 - \alpha}} \int_{t}^{T} \frac{\|K(\tau, \eta)\|_{\mathbb{R}^{m \times m}} \|\f2(\eta)\|_{\mathbb{R}^m}}
                {(T - \eta)^{1 - \alpha}} \, \rd \eta \biggr) \\
                & \leq \frac{1}{\theta} \biggl( \|\g2(\tau)\|_{\mathbb{R}^m}
                + \frac{\|\f2(\cdot)\|_{\L2([t, T], \mathbb{R}^m)}}{(T - \tau)^{1 - \alpha}}
                \biggl( \int_{t}^{T} \frac{\|K(\tau, \eta)\|_{\mathbb{R}^{m \times m}}^2}{(T - \eta)^{2 - 2 \alpha}} \, \rd \eta \biggr)^{1 / 2} \biggr) \\
                & \leq \frac{\mu_g}{(T - \tau)^{1 - \alpha}} \biggl( \frac{1}{\theta} + \frac{T^{2 \alpha - 1} \mu_K}{(2 \alpha - 1) \theta^2} \biggr)
                \quad \text{for a.e. } \tau \in [t, T].
            \end{split}
        \end{equation}
        Now, let us consider a function $\hat{f} \colon [t, T] \to \mathbb{R}^m$ such that $\hat{f}(\tau) = (T - \tau)^{1 - \alpha} \f2(\tau)$ for a.e. $\tau \in [t, T]$.
        Then, due to \eqref{equality_f_ast}, we have
        \begin{equation} \label{equality_f}
            R(\tau) \hat{f}(\tau) + \int_{t}^{T} \frac{K(\tau, \eta) \hat{f}(\eta)}{(T - \eta)^{2 - 2 \alpha}} \, \rd \eta
            = g(\tau)
            \quad \text{for a.e. } \tau \in [t, T],
        \end{equation}
        and, by \eqref{estimate_f_ast}, we get
        \begin{equation} \label{estimate_f}
            \|\hat{f}(\tau)\|_{\mathbb{R}^m}
            \leq \mu_{\A2} \mu_g
            \quad \text{for a.e. } \tau \in [t, T]
        \end{equation}
        with the number
        \begin{equation*}
            \mu_{\A2}
            \doteq \frac{1}{\theta} + \frac{T^{2 \alpha - 1} \mu_K}{(2 \alpha - 1) \theta^2}.
        \end{equation*}

        Let us denote by $\mathcal{T}$ the set of all points $\tau \in [t, T]$ for which both equality \eqref{equality_f} and estimate \eqref{estimate_f} are satisfied.
        Then, for any $\tau$, $\tau^\prime \in \mathcal{T}$, we derive
        \begin{equation*}
            \begin{split}
                & \|R(\tau^\prime) \hat{f}(\tau^\prime) - R(\tau) \hat{f}(\tau)\|_{\mathbb{R}^m} \\
                & \leq \|g(\tau^\prime) - g(\tau)\|_{\mathbb{R}^m}
                + \int_{t}^{T} \frac{\|K(\tau^\prime, \eta) - K(\tau, \eta)\|_{\mathbb{R}^{m \times m}} \|\hat{f}(\eta)\|_{\mathbb{R}^m}}
                {(T - \eta)^{2 - 2 \alpha}} \, \rd \eta \\
                & \leq \omega_g (|\tau^\prime - \tau|) + \frac{T^{2 \alpha - 1} \mu_{\A2} \mu_g \omega_K(|\tau^\prime - \tau|)}{2 \alpha - 1}
            \end{split}
        \end{equation*}
        and, consequently,
        \begin{equation} \label{estimate_f_ast_final}
            \begin{split}
                & \|\hat{f}(\tau^\prime) - \hat{f}(\tau)\|_{\mathbb{R}^m} \\
                & \leq \|R(\tau^\prime)^{- 1}\|_{\mathbb{R}^{m \times m}} \|R(\tau^\prime) \hat{f}(\tau^\prime) - R(\tau^\prime) \hat{f}(\tau)\|_{\mathbb{R}^m} \\
                & \leq \frac{1}{\theta} \bigl( \|R(\tau^\prime) \hat{f}(\tau^\prime) - R(\tau) \hat{f}(\tau)\|_{\mathbb{R}^m}
                + \|R(\tau) - R(\tau^\prime)\|_{\mathbb{R}^{m \times m}} \|\hat{f}(\tau)\|_{\mathbb{R}^m} \bigr) \\
                & \leq \frac{\omega_g (|\tau^\prime - \tau|)}{\theta} + \mu_g \omega_{\A2}(|\tau^\prime - \tau|)
            \end{split}
        \end{equation}
        with the modulus of continuity
        \begin{equation*}
            \omega_{\A2}(\delta)
            \doteq \frac{T^{2 \alpha - 1} \mu_{\A2} \omega_K(\delta)}{(2 \alpha - 1) \theta}
            + \frac{\mu_{\A2} \omega_R(\delta)}{\theta}
            \quad \forall \delta
            \geq 0.
        \end{equation*}
        Thus, the restriction of the function $\hat{f}$ to the set $\mathcal{T}$ is uniformly continuous.
        Hence, taking into account that the set $\mathcal{T}$ is dense in $[t, T]$, there exists (see, e.g., \cite[Theorem I.2.13]{Warga_1972}) a unique function $f(\cdot) \in \C([t, T], \mathbb{R}^m)$ such that $f(\tau) = \hat{f}(\tau)$ for all $\tau \in \mathcal{T}$.
        In particular, let us note that estimates \eqref{proposition_existence_of_a_continuous_solution_main_2} and \eqref{proposition_existence_of_a_continuous_solution_main_3} for this function $f(\cdot)$ follow directly from \eqref{estimate_f} and \eqref{estimate_f_ast_final}, respectively.

        Let us verify that the function $f(\cdot)$ satisfies equation \eqref{proposition_existence_of_a_continuous_solution_main}.
        Let $\tau \in [t, T]$ and let a sequence $\{\tau_i\}_{i \in \mathbb{N}} \subset \mathcal{T}$ be such that $\tau_i \to \tau$ as $i \to \infty$.
        For any $i \in \mathbb{N}$, we have
        \begin{equation} \label{equality_for_every_n}
            R(\tau_i) f(\tau_i) + \int_{t}^{T} \frac{K(\tau_i, \eta) f(\eta)}{(T - \eta)^{2 - 2 \alpha}} \, \rd \eta
            = R(\tau_i) \hat{f}(\tau_i) + \int_{t}^{T} \frac{K(\tau_i, \eta) \hat{f}(\eta)}{(T - \eta)^{2 - 2 \alpha}} \, \rd \eta
            = g(\tau_i).
        \end{equation}
        Hence, owing to continuity of the functions $R(\cdot)$, $f(\cdot)$, $g(\cdot)$, and
        \begin{equation*}
            [0, T] \ni \tau \mapsto \int_{t}^{T} \frac{K(\tau, \eta) f(\eta)}{(T - \eta)^{2 - 2 \alpha}} \, \rd \eta \in \mathbb{R}^m,
        \end{equation*}
        we derive the desired equality by passing to the limit as $i \to \infty$ in \eqref{equality_for_every_n}.

        Finally, let us suppose that there is another solution $f^\prime(\cdot) \in \C([t, T], \mathbb{R}^m)$ of \eqref{proposition_existence_of_a_continuous_solution_main}.
        Then, considering $\f2^\prime(\cdot) \in \L2([t, T], \mathbb{R}^m)$ such that $\f2^\prime(\tau) = f^\prime(\tau) / (T - \tau)^{1 - \alpha}$ for a.e. $\tau \in [t, T]$, we find that $\f2^\prime(\cdot)$ satisfies \eqref{equality_f_ast}, and, therefore, $\f2^\prime(\cdot) = (\A2_{t}^{- 1} \g2)(\cdot) = \f2(\cdot)$.
        Consequently, and since both functions $f^\prime(\cdot)$ and $f(\cdot)$ are continuous, the equality $f^\prime(\tau) = f(\tau)$ is valid for all $\tau \in [t, T]$.

        In the case $t = T$, a unique solution of \eqref{proposition_existence_of_a_continuous_solution_main} is given by $f(T) = R(T)^{- 1} g(T)$.
        In addition, we have $\|f(T)\|_{\mathbb{R}^m} \leq \|g(T)\|_{\mathbb{R}^m} / \theta \leq \mu_{\A2} \|g(T)\|_{\mathbb{R}^m}$, which gives \eqref{proposition_existence_of_a_continuous_solution_main_2}, while estimate \eqref{proposition_existence_of_a_continuous_solution_main_3} holds automatically.
        The proposition is proved.
    \end{proof}

    Note that Proposition \ref{proposition_existence_of_a_continuous_solution} immediately implies
    \begin{corollary} \label{corollary_existence_of_a_continuous_solution}
        For any $t \in [0, T]$ and any $G(\cdot) \in \C([t, T], \mathbb{R}^{m \times m})$, there exists a unique solution $F(\cdot) \in \C([t, T], \mathbb{R}^{m \times m})$ of the Fredholm integral equation
        \begin{equation} \label{corollary_existence_of_a_continuous_solution_main}
            R(\tau) F(\tau) + \int_{t}^{T} \frac{K(\tau, \eta) F(\eta)}{(T - \eta)^{2 - 2 \alpha}} \, \rd \eta
            = G(\tau)
            \quad \forall \tau \in [t, T].
        \end{equation}
        This solution satisfies the estimates
        \begin{align*}
            \|F(\tau)\|_{\mathbb{R}^{m \times m}}
            & \leq m \mu_{\A2} \|G(\cdot)\|_{\C([t, T], \mathbb{R}^{m \times m})}, \\
            \|F(\tau^\prime) - F(\tau)\|_{\mathbb{R}^{m \times m}}
            & \leq \frac{m \omega_G(|\tau^\prime - \tau|)}{\theta} + m \|G(\cdot)\|_{\C([t, T], \mathbb{R}^{m \times m})} \omega_{\A2}(|\tau^\prime - \tau|)
        \end{align*}
        for all $\tau$, $\tau^\prime \in [t, T]$, where $\theta$ is the number from \eqref{theta}, $\omega_G(\cdot)$ is the modulus of continuity of the function $G(\cdot)$, and the number $\mu_{\A2}$ and the modulus of continuity $\omega_{\A2}(\cdot)$ are taken from Proposition \ref{proposition_existence_of_a_continuous_solution}.
    \end{corollary}

    \begin{proof}[Proof of Proposition \ref{proposition_M}]
        Since the function $K(\cdot, \cdot)$ is continuous, it follows from Corollary \ref{corollary_existence_of_a_continuous_solution} that, for any $t \in [0, T]$ and any $\xi \in [t, T]$, there exists a unique continuous function $[t, T] \ni \tau \mapsto M(\tau, \xi \mid t) \in \mathbb{R}^{m \times m}$ satisfying \eqref{M_definition}.
        Hence, in order to prove the first part of the proposition, it suffices to verify that the function $M(\cdot, \cdot \mid \cdot)$ defined in this way is continuous.

        Let us note that Corollary \ref{corollary_existence_of_a_continuous_solution} also yields the estimates
        \begin{align}
            \|M(\tau, \xi \mid t)\|_{\mathbb{R}^{m \times m}}
            & \leq \mu_M,
            \nonumber \\
            \label{omega_M^1}
            \|M(\tau^\prime, \xi \mid t) - M(\tau, \xi \mid t)\|_{\mathbb{R}^{m \times m}}
            & \leq \omega_M^{(1)}(|\tau^\prime - \tau|)
        \end{align}
        for all $\tau$, $\tau^\prime$, $\xi \in [t, T]$ and all $t \in [0, T]$ with the number $\mu_M \geq 0$ and the modulus of continuity $\omega_M^{(1)}(\cdot)$ given by
        \begin{equation*}
            \mu_M
            \doteq \frac{m \mu_{\A2} \mu_K}{\theta},
            \quad \omega_M^{(1)}(\delta)
            \doteq \frac{m \omega_K(\delta)}{\theta^2} + \frac{m \mu_K \omega_{\A2}(\delta)}{\theta}
            \quad \forall \delta
            \geq 0,
        \end{equation*}
        where the numbers $\theta$ and $\mu_K$ are taken from \eqref{theta} (see also \eqref{R^-1_estimate}) and \eqref{mu_K_mu_R}, respectively, and $\omega_K(\cdot)$ is the modulus of continuity of the function $K(\cdot, \cdot)$.

        Further, let us show that there exists a modulus of continuity $\omega_M^{(2)}(\cdot)$ such that
        \begin{equation} \label{omega_M^2}
            \|M(\tau, \xi^\prime \mid t) - M(\tau, \xi \mid t)\|_{\mathbb{R}^{m \times m}}
            \leq \omega_M^{(2)}(|\xi^\prime - \xi|)
        \end{equation}
        for all $\tau$, $\xi$, $\xi^\prime \in [t, T]$ and all $t \in [0, T]$.
        Let us take $t \in [0, T]$ and $\xi$, $\xi^\prime \in [t, T]$ and consider the functions
        \begin{equation*}
            F(\tau)
            \doteq M(\tau, \xi^\prime \mid t) - M(\tau, \xi \mid t),
            \quad G(\tau)
            \doteq - K(\tau, \xi^\prime) R(\xi^\prime)^{- 1} + K(\tau, \xi) R(\xi)^{- 1}
        \end{equation*}
        for all $\tau \in [t, T]$.
        The functions $F(\cdot)$ and $G(\cdot)$ are continuous and satisfy equality \eqref{corollary_existence_of_a_continuous_solution_main} due to \eqref{M_definition}.
        Moreover, we have
        \begin{equation*}
            \|G(\tau)\|_{\mathbb{R}^{m \times m}}
            \leq \mu_K \omega_{R^{- 1}} (|\xi^\prime - \xi|) + \frac{\omega_K(|\xi^\prime - \xi|)}{\theta}
            \quad \forall \tau \in [t, T],
        \end{equation*}
        where $\omega_{R^{- 1}}(\cdot)$ is the modulus of continuity of the function $R(\cdot)^{- 1}$.
        Hence, applying Corollary \ref{corollary_existence_of_a_continuous_solution}, we find that
        \begin{equation*}
            \|F(\tau)\|_{\mathbb{R}^{m \times m}}
            \leq m \mu_{\A2} \biggl( \mu_K \omega_{R^{- 1}} (|\xi^\prime - \xi|) + \frac{\omega_K(|\xi^\prime - \xi|)}{\theta} \biggr)
            \quad \forall \tau \in [t, T],
        \end{equation*}
        which implies
        \eqref{omega_M^2}.

        Arguing in a similar way, let us verify that there exists a modulus of continuity $\omega_M^{(3)}(\cdot)$ such that, for any $t \in [0, T]$, any $t^\prime \in [t, T]$, and any $\tau, \xi \in [t^\prime, T]$,
        \begin{equation} \label{omega_M^3}
            \|M(\tau, \xi \mid t^\prime) - M(\tau, \xi \mid t)\|_{\mathbb{R}^{m \times m}}
            \leq \omega_M^{(3)}(t^\prime - t).
        \end{equation}
        Let us take $t \in [0, T]$, $t^\prime \in [t, T]$, and $\xi \in [t^\prime, T]$ and consider the functions
        \begin{equation*}
            F(\tau)
            \doteq M(\tau, \xi \mid t^\prime) - M(\tau, \xi \mid t),
            \quad G(\tau)
            \doteq \int_{t}^{t^\prime} \frac{K(\tau, \eta) M(\eta, \xi \mid t)}{(T - \eta)^{2 - 2 \alpha}} \, \rd \eta
        \end{equation*}
        for all $\tau \in [t^\prime, T]$.
        The functions $F(\cdot)$ and $G(\cdot)$ are continuous and, in addition,
        \begin{equation} \label{F_equation_tau}
            R(\tau) F(\tau) + \int_{t^\prime}^{T} \frac{K(\tau, \eta) F(\eta)}{(T - \eta)^{2 - 2 \alpha}} \, \rd \eta
            = G(\tau)
            \quad \forall \tau \in [t^\prime, T].
        \end{equation}
        Furthermore, we derive
        \begin{equation*}
            \|G(\tau)\|_{\mathbb{R}^{m \times m}}
            \leq \frac{\mu_K \mu_M (t^\prime - t)^{2 \alpha - 1}}{2 \alpha - 1}
            \quad \forall \tau \in [t^\prime, T].
        \end{equation*}
        Thus, by Corollary \ref{corollary_existence_of_a_continuous_solution}, we get \eqref{omega_M^3}.

        Finally, it remains to observe that continuity of the function $M(\cdot, \cdot \mid \cdot)$ follows directly from relations \eqref{omega_M^1}, \eqref{omega_M^2}, and \eqref{omega_M^3}.

        Let us prove statement (i).
        Let $t \in [0, T)$ and $\vartheta \in (t, T)$ be fixed.
        Since the functions $K(\cdot, \cdot)$ and $M(\cdot, \cdot \mid \cdot)$ are continuous, there exist a number $\tilde{\mu} \geq 0$ and a modulus of continuity $\tilde{\omega}(\cdot)$ such that, for any $\tau$, $\xi \in [t, T]$ and any $\eta \in [t, \vartheta]$,
        \begin{align*}
            \| K(\tau, \eta) M(\eta, t \mid t) R(t) M(t, \xi \mid t) \|_{\mathbb{R}^{m \times m}}
            & \leq \tilde{\mu}, \\
            \biggl\| \frac{K(\tau, \eta) M(\eta, \xi \mid t)}{(T - \eta)^{2 - 2 \alpha}}
            - \frac{K(\tau, t) M(t, \xi \mid t)}{(T - t)^{2 - 2 \alpha}} \biggr\|_{\mathbb{R}^{m \times m}}
            & \leq \tilde{\omega}(\eta - t).
        \end{align*}
        Let us take $t^\prime \in (t, \vartheta]$ and $\xi \in [t^\prime, T]$ and consider the functions
        \begin{equation*}
            F(\tau)
            \doteq M(\tau, \xi \mid t^\prime) - M(\tau, \xi \mid t)
            + \frac{(t^\prime - t) M(\tau, t \mid t) R(t) M(t, \xi \mid t)}{(T - t)^{2 - 2 \alpha}}
        \end{equation*}
        and
        \begin{equation*}
            \begin{split}
                G(\tau)
                & \doteq \int_{t}^{t^\prime} \frac{K(\tau, \eta) M(\eta, \xi \mid t)}{(T - \eta)^{2 - 2 \alpha}} \, \rd \eta
                - \frac{(t^\prime - t) K(\tau, t) M(t, \xi \mid t)}{(T - t)^{2 - 2 \alpha}} \\
                & \quad - \frac{t^\prime - t}{(T - t)^{2 - 2 \alpha}}
                \int_{t}^{t^\prime} \frac{K(\tau, \eta) M(\eta, t \mid t)}{(T - \eta)^{2 - 2 \alpha}} \, \rd \eta
                \, R(t) M(t, \xi \mid t)
            \end{split}
        \end{equation*}
        for all $\tau \in [t^\prime, T]$.
        The functions $F(\cdot)$ and $G(\cdot)$ are continuous and, in view of definition \eqref{M_definition} of the function $M(\cdot, \cdot \mid \cdot)$, satisfy equality \eqref{F_equation_tau}.
        Moreover, we get
        \begin{equation*}
            \begin{split}
                \|G(\tau)\|_{\mathbb{R}^{m \times m}}
                & \leq \int_{t}^{t^\prime} \tilde{\omega}(\eta - t) \, \rd \eta
                + \frac{\tilde{\mu} (t^\prime - t)}{(T - t)^{2 - 2 \alpha}} \int_{t}^{t^\prime} \frac{\rd \eta}{(T - \eta)^{2 - 2 \alpha}} \\
                & \leq \tilde{\omega}(t^\prime - t) (t^\prime - t) + \frac{ \tilde{\mu} (t^\prime - t)^2}{(T - \vartheta)^{4 - 4 \alpha}}
                \quad \forall \tau \in [t^\prime, T],
            \end{split}
        \end{equation*}
        and, therefore, using Corollary \ref{corollary_existence_of_a_continuous_solution}, we obtain the required property.

        Thus, it remains to prove statement (ii).
        Let us note that, for every $t \in [0, T)$, the inverse operator $\A2_t^{- 1}$ can be represented as follows:
        \begin{equation*}
            (\A2_t^{- 1} \g2)(\tau)
            = R(\tau)^{- 1} \g2(\tau) + \frac{1}{(T - \tau)^{1 - \alpha}} \int_{t}^{T} \frac{M(\tau, \xi \mid t) \g2(\xi)}{(T - \xi)^{1 - \alpha}} \, \rd \xi
        \end{equation*}
        for a.e. $\tau \in [t, T]$ and every $\g2(\cdot) \in \L2([t, T], \mathbb{R}^m)$.
        This fact can be verified by a direct substitution on the basis of definition \eqref{M_definition} of the function $M(\cdot, \cdot \mid \cdot)$.
        Hence, and since the operator $\A2_{t}^{- 1}$ is self-adjoint and the matrices $R(\tau)^{- 1}$ for all $\tau \in [0, T]$ are symmetric, we conclude that $M(\tau, \xi \mid t) = M(\xi, \tau \mid t)^\top$ for a.e. $\tau \in [t, T]$ and a.e. $\xi \in [t, T]$ (see, e.g., \cite[Ch. 9, \S 8, Theorem 1]{Zaanen_1953} and \cite[Example 4.4.6]{Debnath_Mikusinski_2005}).
        Then, due to continuity of the function $M(\cdot, \cdot \mid \cdot)$, we get \eqref{M_symmetric}.
    \end{proof}

\section{Proofs of Section \ref{subsection_quadratic_functionals}}
\label{appendix_2}

    The proofs presented in this appendix are of technical nature and consist essentially of verification of boundedness, continuity, and differentiability properties of the mappings under consideration.
    For the sake of brevity, we will use the following shorthand notation.
    If pairs $(t, w(\cdot))$, $(t^\prime, w^\prime(\cdot)) \in \mathcal{G}$ are fixed and a family of functions $v(\cdot \mid \cdot, \cdot)$ is considered, we put (see \eqref{dist})
    \begin{equation} \label{notation_dist_v}
        v(\cdot)
        \doteq v(\cdot \mid t, w(\cdot)),
        \ \ v^\prime(\cdot)
        \doteq v(\cdot \mid t^\prime, w^\prime(\cdot)),
        \quad \dist
        \doteq \dist \bigl( (t^\prime, w^\prime(\cdot)), (t, w(\cdot)) \bigr).
    \end{equation}
    The corresponding notation will also be used for families of functions $\partial_t^\alpha v(\cdot \mid \cdot, \cdot)$ and $\nabla^\alpha v(\cdot \mid \cdot, \cdot)$.
    Furthermore, if $x(\cdot) \in \mathcal{X}(t, w(\cdot))$ and $t^\prime \in [t, T]$ are given, we will apply the same notation as above assuming that (see \eqref{x_t})
    \begin{equation} \label{notation_w^prime}
        w^\prime(\cdot)
        \doteq x_{t^\prime}(\cdot).
    \end{equation}

    \begin{proof}[Proof of Proposition \ref{proposition_v_int_t_T}]
        Let us put
        \begin{equation} \label{mu_W}
            \mu_W
            \doteq \max_{(\tau, \xi) \in \Omega}
            \|W(\tau, \xi)\|_{\mathbb{R}^{k \times k}}
        \end{equation}
        and denote by $\omega_W(\cdot)$ the modulus of continuity of the function $W(\cdot, \cdot)$.

        Let us show that conditions (i$_\star$) and (ii$_\star$) are satisfied.
        Note that we deal with the families of functions $\hat{v}(\cdot \mid \cdot, \cdot)$ and $\partial_t^\alpha \hat{v}(\cdot \mid \cdot, \cdot)$ only and omit the proof for $\nabla^\alpha \hat{v}(\cdot \mid \cdot, \cdot)$, since it can be carried out in a similar way.
        Let $\mathcal{G}_\ast \subset \mathcal{G}$ be a compact set, and let the corresponding number $\mu_v$ and modulus of continuity $\omega_v^{(2)}(\cdot)$ be chosen by conditions (i$_\star$) and (ii$_\star$) for the original family of functions $v(\cdot \mid \cdot, \cdot)$.

        Let $(t, w(\cdot)) \in \mathcal{G}_\ast$ be fixed.
        For any $\tau \in [t, T]$, we have (see \eqref{notation_dist_v})
        \begin{align*}
            \|\hat{v}(\tau)\|_{\mathbb{R}^k}
            & \leq \frac{T^{1 - \gamma} \mu_W \mu_v}{1 - \gamma}, \\
            \|\partial_t^\alpha \hat{v}(\tau)\|_{\mathbb{R}^k}
            & \leq \mu_W \mu_v \int_{\tau}^{T} \frac{\rd \xi}{(\xi - \tau)^\gamma (\xi - t)^\gamma}
            \leq \frac{T^{1 - 2 \gamma} \mu_W \mu_v}{1 - 2 \gamma},
        \end{align*}
        which implies the first part of condition (i$_\star$).
        Further, for any $\tau$, $\tau^\prime \in [t, T]$ with $\tau^\prime \geq \tau$, relying on Proposition \ref{proposition_technical_continuity_0}, we derive
        \begin{equation*}
            \begin{split}
                \| \hat{v} (\tau^\prime) - \hat{v}(\tau) \|_{\mathbb{R}^k}
                & \leq \biggl\| \int_{\tau^\prime}^{T} \frac{(W(\xi, \tau^\prime) - W(\xi, \tau)) v(\xi)}{(\xi - \tau^\prime)^\gamma} \, \rd \xi \biggr\|_{\mathbb{R}^k} \\
                & \quad + \biggl\| \int_{\tau^\prime}^{T} \frac{W(\xi, \tau) v(\xi)}{(\xi - \tau^\prime)^\gamma} \, \rd \xi
                - \int_{\tau}^{T} \frac{W(\xi, \tau) v(\xi)}{(\xi - \tau)^\gamma} \, \rd \xi \biggr\|_{\mathbb{R}^k} \\
                & \leq \frac{T^{1 - \gamma} \mu_v \omega_W(\tau^\prime - \tau)}{1 - \gamma}
                + \mu_W \mu_v \bar{\omega}_\gamma^{(1)}(\tau^\prime - \tau).
            \end{split}
        \end{equation*}
        In addition, due to Proposition \ref{proposition_technical_continuity_2}, we get
        \begin{equation*}
            \begin{split}
                \| \partial_t^\alpha \hat{v}(\tau^\prime) - \partial_t^\alpha \hat{v}(\tau) \|_{\mathbb{R}^k}
                & \leq \biggl\| \int_{\tau^\prime}^{T} \frac{(W(\xi, \tau^\prime) - W(\xi, \tau)) \partial_t^\alpha v(\xi)}
                {(\xi - \tau^\prime)^\gamma (\xi - t)^\gamma} \, \rd \xi \biggr\|_{\mathbb{R}^k} \\
                & \quad + \biggl\| \int_{\tau^\prime}^{T} \frac{W(\xi, \tau) \partial_t^\alpha v(\xi)}
                {(\xi - \tau^\prime)^\gamma (\xi - t)^\gamma} \, \rd \xi
                - \int_{\tau}^{T} \frac{W(\xi, \tau) \partial_t^\alpha v(\xi)}
                {(\xi - \tau)^\gamma (\xi - t)^\gamma} \, \rd \xi \biggr\|_{\mathbb{R}^k} \\
                & \leq \frac{T^{1 - 2 \gamma} \mu_v \omega_W(\tau^\prime - \tau)}{1 - 2 \gamma}
                + \mu_W \mu_v \bar{\omega}_\gamma^{(3)}(\tau^\prime - \tau).
            \end{split}
        \end{equation*}
        Consequently, we obtain the second part of condition (i$_\star$).

        Now, let $(t, w(\cdot))$, $(t^\prime, w^\prime(\cdot)) \in \mathcal{G}_\ast$ with $t^\prime \geq t$ be fixed and let $\tau \in [t^\prime, T]$.
        Then, we have (see \eqref{notation_dist_v})
        \begin{equation*}
            \| \hat{v}^\prime(\tau) - \hat{v}(\tau) \|_{\mathbb{R}^k}
            = \biggl\| \int_{\tau}^{T} \frac{W(\xi, \tau) (v^\prime(\xi) - v(\xi))}
            {(\xi - \tau)^\gamma} \, \rd \xi \biggr\|_{\mathbb{R}^k}
            \leq \frac{T^{1 - \gamma} \mu_W \omega_v^{(2)}(\dist)}{1 - \gamma}.
        \end{equation*}
        Moreover, using Proposition \ref{proposition_technical_continuity_2} once again, we derive
        \begin{align*}
            \| \partial_t^\alpha \hat{v}^\prime(\tau) - \partial_t^\alpha \hat{v}(\tau) \|_{\mathbb{R}^k}
            & \leq \biggl\| \int_{\tau}^{T} \frac{W(\xi, \tau)
            (\partial_t^\alpha v^\prime(\xi) - \partial_t^\alpha v(\xi))}
            {(\xi - \tau)^\gamma (\xi - t^\prime)^\gamma} \, \rd \xi \biggr\|_{\mathbb{R}^k} \\*
            & \quad + \biggl\| \int_{\tau}^{T} \frac{W(\xi, \tau) \partial_t^\alpha v(\xi)}
            {(\xi - \tau)^\gamma (\xi - t^\prime)^\gamma} \, \rd \xi
            - \int_{\tau}^{T} \frac{W(\xi, \tau) \partial_t^\alpha v(\xi)}
            {(\xi - \tau)^\gamma (\xi - t)^\gamma} \, \rd \xi \biggr\|_{\mathbb{R}^k} \\
            & \leq \frac{T^{1 - 2 \gamma} \mu_W \omega_v^{(2)} (\dist)}{1 - 2 \gamma}
            + \mu_W \mu_v \bar{\omega}_\gamma^{(3)} (\dist).
        \end{align*}
        Thus, we conclude that condition (ii$_\star$) is satisfied.

        Finally, let us verify condition (iii$_\star$).
        Let us fix $(t, w(\cdot)) \in \mathcal{G}_0$, $x(\cdot) \in \mathcal{X}(t, w(\cdot))$, and $\vartheta \in (t, T)$ and choose the corresponding number $\varkappa_v$ and infinitesimal modulus of continuity $\omega_v^\ast(\cdot)$ by condition (iii$_\star$) for $v(\cdot \mid \cdot, \cdot)$.
        Then, we derive (see \eqref{notation_w^prime})
        \begin{equation*}
            \begin{split}
                & \biggl\| \hat{v}^\prime(\tau) - \hat{v}(\tau)
                - \frac{\partial_t^\alpha \hat{v}(\tau) (t^\prime - t)
                + \nabla^\alpha \hat{v}(\tau) z(t^\prime)}{(T - t)^\beta} \biggr\|_{\mathbb{R}^k} \\
                & = \biggl\| \int_{\tau}^{T} \frac{W(\xi, \tau)}{(\xi - \tau)^\gamma}
                \biggl( v^\prime(\xi) - v(\xi)
                - \frac{\partial_t^\alpha v(\xi)(t^\prime - t) + \nabla^\alpha v(\xi) z(t^\prime)}
                {(T - t)^\beta (\xi - t)^\gamma} \biggr) \, \rd \xi \biggr\|_{\mathbb{R}^k} \\
                & \leq \mu_W \omega_v^\ast(t^\prime - t) \int_{\tau}^{T} \frac{\rd \xi}{(\xi - \tau)^\gamma (\xi - t^\prime)^{\varkappa_v}} \\
                & \leq \frac{T^{1 - \gamma} \mu_W \omega_v^\ast(t^\prime - t)}{(1 - \gamma) (\tau - t^\prime)^{\varkappa_v}}
                \quad \forall \tau \in (t^\prime, T] \quad \forall t^\prime \in (t, \vartheta],
            \end{split}
        \end{equation*}
        which implies condition (iii$_\star$) for $\hat{v}(\cdot \mid \cdot, \cdot)$ and completes the proof.
    \end{proof}

    \begin{proof}[Proof of Proposition \ref{proposition_v_int_t_ast_t}]
        We follow the scheme of the proof of Proposition \ref{proposition_v_int_t_T}.

        Let us show that conditions (i$_\star$) and (ii$_\star$) are satisfied (we omit the proof for $\nabla^\alpha \hat{v}(\cdot \mid \cdot, \cdot)$).
        Let $\mathcal{G}_\ast \subset \mathcal{G}$ be a compact set, and let the corresponding number $\mu_v$ and moduli of continuity $\omega_v^{(1)}(\cdot)$ and $\omega_v^{(2)}(\cdot)$ be chosen by conditions (i$_\star$) and (ii$_\star$) for the family of functions $v(\cdot \mid \cdot, \cdot)$.

        Let $(t, w(\cdot)) \in \mathcal{G}_\ast$ be fixed.
        For any $\tau \in [t, T]$, using \eqref{technical_beta}, we obtain
        \begin{equation*}
            \|\hat{v}(\tau)\|_{\mathbb{R}^k}
            \leq \frac{T^{1 - \gamma} \mu_W \mu_v}{1 - \gamma},
            \quad \| \partial_t^\alpha \hat{v}(\tau)\|_{\mathbb{R}^k}
            \leq \mathrm{B}(1 - \gamma, 1 - \gamma) T^{1 - \gamma} \mu_W \mu_v
            + T^\beta \mu_W \mu_v,
        \end{equation*}
        where the number $\mu_W$ is taken from \eqref{mu_W} and $\mathrm{B}(\cdot, \cdot)$ is the beta-function.
        Further, for any $\tau$, $\tau^\prime \in [t, T]$ with $\tau^\prime \geq \tau$, in view of Proposition \ref{proposition_technical_continuity_0}, we get
        \begin{equation*}
            \begin{split}
                \| \hat{v}(\tau^\prime) - \hat{v}(\tau) \|_{\mathbb{R}^k}
                & \leq \biggl\| \int_{t}^{\tau^\prime} \frac{W(\tau^\prime, \xi) v(\xi)}{(\tau^\prime - \xi)^\gamma} \, \rd \xi
                - \int_{t}^{\tau} \frac{W(\tau^\prime, \xi) v(\xi)}{(\tau - \xi)^\gamma} \, \rd \xi \biggr\|_{\mathbb{R}^k} \\
                & \quad + \biggl\| \int_{t}^{\tau} \frac{(W(\tau^\prime, \xi) - W(\tau, \xi)) v(\xi)}{(\tau - \xi)^\gamma} \, \rd \xi \biggr\|_{\mathbb{R}^k} \\
                & \leq \mu_W \mu_v \bar{\omega}_\gamma^{(1)} (\tau^\prime - \tau) + \frac{T^{1 - \gamma} \mu_v \omega_W(\tau^\prime - \tau)}{1 - \gamma},
            \end{split}
        \end{equation*}
        where $\omega_W(\cdot)$ is the modulus of continuity of the function $W(\cdot, \cdot)$.
        Moreover, due to formula \eqref{technical_beta} and Proposition \ref{proposition_technical_continuity}, we derive
        \begin{equation*}
            \begin{split}
                & \| \partial_t^\alpha \hat{v}(\tau^\prime) - \partial_t^\alpha \hat{v}(\tau) \|_{\mathbb{R}^k} \\
                & \leq \biggl\|(\tau^\prime - t)^\gamma \int_{t}^{\tau^\prime}
                \frac{W(\tau^\prime, \xi) \partial_t^\alpha v(\xi)}
                {(\tau^\prime - \xi)^\gamma (\xi - t)^\gamma} \, \rd \xi
                - (\tau - t)^\gamma \int_{t}^{\tau} \frac{W(\tau^\prime, \xi) \partial_t^\alpha v(\xi)}
                {(\tau - \xi)^\gamma (\xi - t)^\gamma} \, \rd \xi \biggr\|_{\mathbb{R}^k} \\
                & \quad + (\tau - t)^\gamma \biggl\|  \int_{t}^{\tau}
                \frac{(W(\tau^\prime, \xi) - W(\tau, \xi)) \partial_t^\alpha v(\xi)}
                {(\tau - \xi)^\gamma (\xi - t)^\gamma} \, \rd \xi \biggr\|_{\mathbb{R}^k} \\
                & \quad + (T - t)^\beta \| (W(\tau^\prime, t) - W(\tau, t)) v(t) \|_{\mathbb{R}^k} \\
                & \leq \mu_W \mu_v \bar{\omega}_{\gamma, \gamma}^{(2)}(\tau^\prime - \tau)
                + \mathrm{B}(1 - \gamma, 1 - \gamma) T^{1 - \gamma} \mu_v \omega_W(\tau^\prime - \tau)
                + T^\beta \mu_v \omega_W(\tau^\prime - \tau),
            \end{split}
        \end{equation*}
        which completes the proof of condition (i$_\star$).

        Now, let $(t, w(\cdot))$, $(t^\prime, w^\prime(\cdot)) \in \mathcal{G}_\ast$ with $t^\prime \geq t$ be fixed and let $\tau \in [t^\prime, T]$.
        We have
        \begin{equation*}
            \begin{split}
                \|\hat{v}^\prime (\tau) - \hat{v}(\tau)\|_{\mathbb{R}^k}
                & \leq \biggl\| \int_{t^\prime}^{\tau} \frac{W(\tau, \xi) (v^\prime(\xi) - v(\xi))}
                {(\tau - \xi)^\gamma} \, \rd \xi \biggr\|_{\mathbb{R}^k}
                + \biggl\| \int_{t}^{t^\prime} \frac{W(\tau, \xi) v(\xi)}{(\tau - \xi)^\gamma} \, \rd \xi \biggr\|_{\mathbb{R}^k} \\
                & \leq \frac{T^{1 - \gamma} \mu_W \omega_v^{(2)}(\dist)}{1 - \gamma}
                + \frac{\mu_W \mu_v \dist^{1 - \gamma}}{1 - \gamma}.
            \end{split}
        \end{equation*}
        Furthermore, we obtain
        \begin{equation*}
            \begin{split}
                & \| \partial_t^\alpha \hat{v}^\prime(\tau) - \partial_t^\alpha \hat{v}(\tau)\|_{\mathbb{R}^k} \\
                & \leq (\tau - t^\prime)^\gamma \biggl\| \int_{t^\prime}^{\tau} \frac{W(\tau, \xi)
                (\partial_t^\alpha v^\prime(\xi) - \partial_t^\alpha v(\xi))}{(\tau - \xi)^\gamma (\xi - t^\prime)^\gamma} \, \rd \xi \biggr\|_{\mathbb{R}^k} \\
                & \quad + \biggl\| (\tau - t^\prime)^\gamma \int_{t^\prime}^{\tau} \frac{W(\tau, \xi) \partial_t^\alpha v(\xi)}
                {(\tau - \xi)^\gamma (\xi - t^\prime)^\gamma} \, \rd \xi
                - (\tau - t)^\gamma \int_{t}^{\tau} \frac{W(\tau, \xi) \partial_t^\alpha v(\xi)}
                {(\tau - \xi)^\gamma (\xi - t)^\gamma} \, \rd \xi \biggr\|_{\mathbb{R}^k} \\
                & \quad + \| (T - t^\prime)^\beta W(\tau, t^\prime) v^\prime(t^\prime)
                - (T - t)^\beta W(\tau, t) v(t) \|_{\mathbb{R}^k} \\
                & \doteq \mathscr{A}_1 + \mathscr{A}_2 + \mathscr{A}_3.
            \end{split}
        \end{equation*}
        According to formula \eqref{technical_beta} and Proposition \ref{proposition_technical_continuity}, we get
        \begin{equation*}
            \mathscr{A}_1
            \leq \mathrm{B}(1 - \gamma, 1 - \gamma) T^{1 - \gamma} \mu_W \omega_v^{(2)}(\dist),
            \quad \mathscr{A}_2
            \leq \mu_W \mu_v \bar{\omega}_{\gamma, \gamma}^{(2)}(\dist).
        \end{equation*}
        In addition, we derive
        \begin{equation*}
            \begin{split}
                \mathscr{A}_3
                & \leq \| ((T - t^\prime)^\beta - (T - t)^\beta) W(\tau, t^\prime) v^\prime(t^\prime) \|_{\mathbb{R}^k} \\
                & \quad + (T - t)^\beta \bigl( \|(W(\tau, t^\prime) - W(\tau, t)) v^\prime(t^\prime)\|_{\mathbb{R}^k} \\
                & \quad + \| W(\tau, t) (v^\prime(t^\prime) - v(t^\prime)) \|_{\mathbb{R}^k}
                + \| W(\tau, t) (v(t^\prime) - v(t))\|_{\mathbb{R}^k} \bigr) \\
                & \leq \mu_W \mu_v \dist^\beta
                + T^\beta \bigl( \mu_v \omega_W(\dist)
                + \mu_W \omega_v^{(2)}(\dist)
                + \mu_W \omega_v^{(1)}(\dist) \bigr),
            \end{split}
        \end{equation*}
        and, hence, condition (ii$_\star$) is proved.

        Further, let us fix $(t, w(\cdot)) \in \mathcal{G}_0$, $x(\cdot) \in \mathcal{X}(t, w(\cdot))$, and $\vartheta \in (t, T)$ and choose the corresponding number $\varkappa_v$ and infinitesimal modulus of continuity $\omega_v^\ast(\cdot)$ by condition (iii$_\star$) for $v(\cdot \mid \cdot, \cdot)$.
        Moreover, according to condition (i$_\star$) for $v(\cdot \mid \cdot, \cdot)$, let us take the number $\mu_v$ and modulus of continuity $\omega_v^{(1)}(\cdot)$ that correspond to the compact set $\mathcal{G}_\ast \doteq \{(t, w(\cdot))\}$.
        For any $t^\prime \in (t, \vartheta]$ and any $\tau \in (t^\prime, T]$, we have
        \begin{equation*}
            \begin{split}
                & \biggl\| \hat{v}^\prime(\tau) - \hat{v}(\tau)
                - \frac{\partial_t^\alpha \hat{v}(\tau) (t^\prime - t)
                + \nabla^\alpha \hat{v}(\tau) z(t^\prime)}{(T - t)^\beta (\tau - t)^\gamma} \biggr\|_{\mathbb{R}^k} \\
                & \leq \biggl\| \int_{t^\prime}^{\tau} \frac{W(\tau, \xi)}{(\tau - \xi)^\gamma}
                \biggl( v^\prime(\xi) - v(\xi)
                - \frac{\partial_t^\alpha v(\xi) (t^\prime - t) + \nabla^\alpha v(\xi) z(t^\prime)}{(T - t)^\beta (\xi - t)^\gamma} \biggr) \, \rd \xi \biggr\|_{\mathbb{R}^k} \\
                & \quad + \frac{1}{(T - t)^\beta} \biggl\| \int_{t}^{t^\prime} \frac{W(\tau, \xi)
                \bigl( \partial_t^\alpha v(\xi)(t^\prime - t) + \nabla^\alpha v(\xi) z(t^\prime) \bigr)}
                {(\tau - \xi)^\gamma (\xi - t)^\gamma} \, \rd \xi \biggr\|_{\mathbb{R}^k} \\
                & \quad + \biggl\| \int_{t}^{t^\prime} \frac{W(\tau, \xi) v(\xi)}{(\tau - \xi)^\gamma} \, \rd \xi
                - \frac{(t^\prime - t) W(\tau, t) v(t)}{(\tau - t)^\gamma} \biggr\|_{\mathbb{R}^k} \\
                & \doteq \mathscr{B}_1 + \mathscr{B}_2 + \mathscr{B}_3.
            \end{split}
        \end{equation*}
        Applying \eqref{technical_beta}, we obtain
        \begin{equation*}
            \mathscr{B}_1
            \leq \mu_W \omega_v^\ast(t^\prime - t) \int_{t^\prime}^{\tau} \frac{\rd \xi}{(\tau - \xi)^\gamma (\xi - t^\prime)^{\varkappa_v}}
            \leq \frac{\mathrm{B}(1 - \gamma, 1 - \varkappa_v) T^{1 - \varkappa_v} \mu_W \omega_v^\ast(t^\prime - t)}{(\tau - t^\prime)^\gamma}.
        \end{equation*}
        Using estimate \eqref{z_estimate}, where the number $\mu_x^\ast$ is given by \eqref{mu_x^ast}, we get
        \begin{equation*}
            \mathscr{B}_2
            \leq \frac{\mu_W \mu_v (1 + \mu_x^\ast) (t^\prime - t)}{(T - \vartheta)^\beta}
            \int_{t}^{t^\prime} \frac{\rd \xi}{(\tau - \xi)^\gamma (\xi - t)^\gamma}
            \leq \frac{\mu_W \mu_v (1 + \mu_x^\ast) (t^\prime - t)^{2 - \gamma}}{(1 - \gamma) (T - \vartheta)^\beta (\tau - t^\prime)^\gamma}.
        \end{equation*}
        Finally, if $\gamma > 0$, we derive
        \begin{equation*}
            \begin{split}
                \mathscr{B}_3
                & \leq \biggl\| \int_{t}^{t^\prime} \frac{(W(\tau, \xi) - W(\tau, t)) v(\xi)}
                {(\tau - \xi)^\gamma} \, \rd \xi \biggr\|_{\mathbb{R}^k}
                + \biggl\| W(\tau, t) \int_{t}^{t^\prime} \frac{v(\xi) - v(t)}
                {(\tau - \xi)^\gamma} \, \rd \xi \biggr\|_{\mathbb{R}^k} \\
                & \quad + \biggl\| W(\tau, t) v(t) \int_{t}^{t^\prime}
                \biggl( \frac{1}{(\tau - \xi)^\gamma} - \frac{1}{(\tau - t)^\gamma} \biggr) \, \rd \xi \biggr\|_{\mathbb{R}^k} \\
                & \leq \frac{\mu_v \omega_W(t^\prime - t) (t^\prime - t)}{(\tau - t^\prime)^\gamma}
                + \frac{\mu_W \omega_v^{(1)}(t^\prime - t) (t^\prime - t)}{(\tau - t^\prime)^\gamma}
                + \frac{\mu_W \mu_v (t^\prime - t)^{\gamma + 1}}{(\tau - t^\prime)^{2 \gamma}}.
            \end{split}
        \end{equation*}
        Otherwise, the same estimate for $\mathscr{B}_3$ holds but without the last term.
        Thus, we conclude that condition (iii$_\star$) is satisfied for $\hat{v}(\cdot \mid \cdot, \cdot)$.
        The proposition is proved.
    \end{proof}

    \begin{proof}[Proof of Proposition
    \ref{proposition_v_int_M}]
        The proof follows the same lines as those of Propositions \ref{proposition_v_int_t_T} and \ref{proposition_v_int_t_ast_t}.

        Let us denote by $\omega_N(\cdot)$ and $\omega_{\partial_t N}(\cdot)$ the moduli of continuity of the functions $N(\cdot, \cdot \mid \cdot)$ and $\partial_t N(\cdot, \cdot \mid \cdot)$ and put
        \begin{equation*}
            \mu_N
            \doteq \max_{(\tau, \xi, t) \in \Theta} \|N(\tau, \xi \mid t)\|_{\mathbb{R}^{k \times k}},
            \quad \mu_{\partial_t N}
            \doteq \max_{(\tau, \xi, t) \in \Theta} \|\partial_t N(\tau, \xi \mid t)\|_{\mathbb{R}^{k \times k}}.
        \end{equation*}

        Let us verify conditions (i$_\star$) and (ii$_\star$) (the proof for $\nabla^\alpha \hat{v}(\cdot \mid \cdot, \cdot)$ is omitted).
        Let $\mathcal{G}_\ast \subset \mathcal{G}$ be a compact set, and let the corresponding number $\mu_v$ and moduli of continuity $\omega_v^{(1)}(\cdot)$ and $\omega_v^{(2)}(\cdot)$ be chosen by conditions (i$_\star$) and (ii$_\star$) for $v(\cdot \mid \cdot, \cdot)$.

        Let $(t, w(\cdot)) \in \mathcal{G}_\ast$ be fixed.
        For any $\tau \in [t, T]$, we have
        \begin{align*}
            \| \hat{v}(\tau) \|_{\mathbb{R}^k}
            & \leq \frac{T^{1 - \sigma} \mu_N \mu_v}{1 - \sigma}, \\
            \|\partial_t^\alpha \hat{v}(\tau)\|_{\mathbb{R}^k}
            & \leq \frac{T^{1 + \hat{\beta} - 2 \sigma} \mu_{\partial_t N} \mu_v}{1 - \sigma}
            + \frac{T^{1 + \hat{\beta} - \beta - \sigma} \mu_N \mu_v}{1 - \sigma}
            + T^{\hat{\beta} - \sigma} \mu_N \mu_v.
        \end{align*}
        For any $\tau$, $\tau^\prime \in [t, T]$ with $\tau^\prime \geq \tau$, we derive
        \begin{equation*}
            \begin{split}
                \| \hat{v}(\tau^\prime) - \hat{v}(\tau) \|_{\mathbb{R}^k}
                & = \biggl\| \int_{t}^{T} \frac{(N(\tau^\prime, \xi \mid t) - N(\tau, \xi \mid t)) v(\xi)}{(T - \xi)^\sigma} \, \rd \xi \biggr\|_{\mathbb{R}^k} \\
                & \leq \frac{T^{1 - \sigma} \mu_v \omega_N(\tau^\prime - \tau)}{1 - \sigma}
            \end{split}
        \end{equation*}
        and
        \begin{equation*}
            \begin{split}
                & \| \partial_t^\alpha \hat{v}(\tau^\prime) - \partial_t^\alpha \hat{v}(\tau) \|_{\mathbb{R}^k} \\
                & \leq (T - t)^{\hat{\beta} - \sigma}
                \biggl\| \int_{t}^{T} \frac{(\partial_t N(\tau^\prime, \xi \mid t) - \partial_t N(\tau, \xi \mid t)) v(\xi)}{(T - \xi)^\sigma} \, \rd \xi \biggr\|_{\mathbb{R}^k} \\
                & \quad + (T - t)^{\hat{\beta} - \beta}
                \biggl\| \int_{t}^{T} \frac{(N(\tau^\prime, \xi \mid t) - N(\tau, \xi \mid t)) \partial_t^\alpha v(\xi)}{(T - \xi)^\sigma} \, \rd \xi \biggr\|_{\mathbb{R}^k} \\
                & \quad + (T - t)^{\hat{\beta} - \sigma}
                \| (N(\tau^\prime, t \mid t) - N(\tau, t \mid t)) v(t) \|_{\mathbb{R}^k} \\
                & \leq \frac{T^{1 + \hat{\beta} - 2 \sigma} \mu_v \omega_{\partial_t N}(\tau^\prime - \tau)}{1 - \sigma}
                + \frac{T^{1 + \hat{\beta} - \beta - \sigma} \mu_v \omega_N(\tau^\prime - \tau)}{1 - \sigma}
                + T^{\hat{\beta} - \sigma} \mu_v \omega_N(\tau^\prime - \tau).
            \end{split}
        \end{equation*}
        Hence, condition (i$_\star$) is proved.

        Now, let $(t, w(\cdot))$, $(t^\prime, w^\prime(\cdot)) \in \mathcal{G}_\ast$ with $t^\prime \geq t$ be fixed and let $\tau \in [t^\prime, T]$.
        We get
        \begin{equation} \label{estimate_hat_v_basic}
            \begin{split}
                \| \hat{v}^\prime(\tau) - \hat{v}(\tau) \|_{\mathbb{R}^k}
                & \leq \biggl\| \int_{t^\prime}^{T} \frac{(N(\tau, \xi \mid t^\prime) - N(\tau, \xi \mid t)) v^\prime(\xi)}
                {(T - \xi)^\sigma} \, \rd \xi \biggr\|_{\mathbb{R}^k} \\
                & \quad + \biggl\| \int_{t^\prime}^{T} \frac{N(\tau, \xi \mid t) (v^\prime(\xi) - v(\xi))}
                {(T - \xi)^\sigma} \, \rd \xi \biggr\|_{\mathbb{R}^k} \\
                & \quad + \biggl\| \int_{t}^{t^\prime} \frac{N(\tau, \xi \mid t) v(\xi)}{(T - \xi)^\sigma} \, \rd \xi \biggr\|_{\mathbb{R}^k} \\
                & \leq \frac{T^{1 - \sigma} \mu_v \omega_N(\dist)}{1 - \sigma}
                + \frac{T^{1 - \sigma} \mu_N \omega_v^{(2)}(\dist)}{1 - \sigma}
                + \frac{\mu_N \mu_v \dist^{1 - \sigma}}{1 - \sigma}.
            \end{split}
        \end{equation}
        In addition, we obtain
        \begin{equation*}
            \begin{split}
                & \| \partial_t^\alpha \hat{v}^\prime(\tau) - \partial_t^\alpha \hat{v}(\tau) \|_{\mathbb{R}^k} \\
                & \leq ((T - t)^{\hat{\beta} - \sigma} - (T - t^\prime)^{\hat{\beta} - \sigma})
                \biggl\| \int_{t^\prime}^{T} \frac{\partial_t N(\tau, \xi \mid t^\prime) v^\prime(\xi)}
                {(T - \xi)^\sigma} \, \rd \xi \biggr\|_{\mathbb{R}^k} \\
                & \quad + (T - t)^{\hat{\beta} - \sigma} \biggl\| \int_{t^\prime}^{T} \frac{\partial_t N(\tau, \xi \mid t^\prime) v^\prime(\xi)}
                {(T - \xi)^\sigma} \, \rd \xi
                - \int_{t}^{T} \frac{\partial_t N(\tau, \xi \mid t) v(\xi)}
                {(T - \xi)^\sigma} \, \rd \xi \biggr\|_{\mathbb{R}^k} \\
                & \quad + ((T - t)^{\hat{\beta} - \beta} - (T - t^\prime)^{\hat{\beta} - \beta})
                \biggl\| \int_{t^\prime}^{T} \frac{N(\tau, \xi \mid t^\prime) \partial_t^\alpha v^\prime(\xi)}
                {(T - \xi)^\sigma} \, \rd \xi \biggr\|_{\mathbb{R}^k} \\
                & \quad + (T - t)^{\hat{\beta} - \beta}
                \biggl\| \int_{t^\prime}^{T} \frac{N(\tau, \xi \mid t^\prime) \partial_t^\alpha v^\prime(\xi)}
                {(T - \xi)^\sigma} \, \rd \xi
                - \int_{t}^{T} \frac{N(\tau, \xi \mid t) \partial_t^\alpha v(\xi)}
                {(T - \xi)^\sigma} \, \rd \xi \biggr\|_{\mathbb{R}^k} \\
                & \quad + ((T - t)^{\hat{\beta} - \sigma} - (T - t^\prime)^{\hat{\beta} - \sigma})
                \|N(\tau, t^\prime \mid t^\prime) v^\prime(t^\prime)\|_{\mathbb{R}^k} \\
                & \quad + (T - t)^{\hat{\beta} - \sigma}
                \bigl\| N(\tau, t^\prime \mid t^\prime) v^\prime(t^\prime) - N(\tau, t \mid t) v(t) \bigr\|_{\mathbb{R}^k} \\
                & \doteq \mathscr{A}_1 + \mathscr{A}_2 + \mathscr{A}_3 + \mathscr{A}_4 + \mathscr{A}_5 + \mathscr{A}_6.
            \end{split}
        \end{equation*}
        If $\hat{\beta} = \sigma$, we have $\mathscr{A}_1 = \mathscr{A}_5 = 0$, and, otherwise,
        \begin{equation*}
            \mathscr{A}_1
            \leq \frac{T^{1 - \sigma} \mu_{\partial_t N} \mu_v \dist^{\hat{\beta} - \sigma}}{1 - \sigma},
            \quad \mathscr{A}_5
            \leq \mu_N \mu_v \dist^{\hat{\beta} - \sigma}.
        \end{equation*}
        If $\hat{\beta} = \beta$, we have $\mathscr{A}_3 = 0$, and, otherwise,
        \begin{equation*}
            \mathscr{A}_3
            \leq \frac{T^{1 - \sigma} \mu_N \mu_v \dist^{\hat{\beta} - \beta}}{1 - \sigma}.
        \end{equation*}
        Moreover, similarly to \eqref{estimate_hat_v_basic}, we derive
        \begin{align*}
            \mathscr{A}_2
            & \leq T^{\hat{\beta} - \sigma} \biggl( \frac{T^{1 - \sigma} \mu_v \omega_{\partial_t N}(\dist)}{1 - \sigma}
            + \frac{T^{1 - \sigma} \mu_{\partial_t N} \omega_v^{(2)}(\dist)}{1 - \sigma}
            + \frac{\mu_{\partial_t N} \mu_v \dist^{1 - \sigma}}{1 - \sigma} \biggr), \\
            \mathscr{A}_4
            & \leq T^{\hat{\beta} - \beta} \biggl( \frac{T^{1 - \sigma} \mu_v \omega_N(\dist)}{1 - \sigma}
            + \frac{T^{1 - \sigma} \mu_N \omega_v^{(2)}(\dist)}{1 - \sigma}
            + \frac{\mu_N \mu_v \dist^{1 - \sigma}}{1 - \sigma} \biggr).
        \end{align*}
        Finally, we get
        \begin{equation*}
            \begin{split}
                \mathscr{A}_6
                & \leq T^{\hat{\beta} - \sigma} \bigl( \| (N(\tau, t^\prime \mid t^\prime) - N(\tau, t \mid t))  v^\prime(t^\prime) \|_{\mathbb{R}^k}
                + \| N(\tau, t \mid t) (v^\prime(t^\prime) - v(t^\prime)) \|_{\mathbb{R}^k} \\
                & \quad + \| N(\tau, t \mid t) (v(t^\prime) - v(t)) \|_{\mathbb{R}^k} \bigr) \\
                & \leq T^{\hat{\beta} - \sigma} \bigl( \mu_v \omega_N(2^{1 / 2} \dist)
                + \mu_N \omega_v^{(2)}(\dist)
                + \mu_N \omega_v^{(1)}(\dist) \bigr),
            \end{split}
        \end{equation*}
        which completes the proof of condition (ii$_\star$).

        Let $(t, w(\cdot)) \in \mathcal{G}_0$, $x(\cdot) \in \mathcal{X}(t, w(\cdot))$, and $\vartheta \in (t, T)$ be fixed.
        Let us choose the corresponding number $\varkappa_v$ and infinitesimal modulus of continuity $\omega_v^\ast(\cdot)$ by condition (iii$_\star$) for $v(\cdot \mid \cdot, \cdot)$.
        Furthermore, according to conditions (i$_\star$) and (ii$_\star$) for $v(\cdot \mid \cdot, \cdot)$, let us take the number $\mu_v$ and moduli of continuity $\omega_v^{(1)}(\cdot)$ and $\omega_v^{(2)}(\cdot)$ that correspond to the compact set (see \eqref{x_t})
        \begin{equation} \label{G_ast_by_x}
            \mathcal{G}_\ast
            \doteq \bigl\{ (t^\prime, x_{t^\prime}(\cdot)) \in \mathcal{G} \colon
            t^\prime \in [t, \vartheta] \bigr\}.
        \end{equation}
        For any $t^\prime \in (t, \vartheta]$ and any $\tau \in (t^\prime, T]$, we have
        \begin{equation*}
            \begin{split}
                & \biggl\| \hat{v}^\prime(\tau) - \hat{v}(\tau)
                - \frac{\partial_t^\alpha \hat{v}(\tau) (t^\prime - t)
                + \nabla^\alpha \hat{v}(\tau) z(t^\prime)}{(T - t)^{\hat{\beta}}} \biggr\|_{\mathbb{R}^k} \\
                & \leq \biggl\| \int_{t^\prime}^{T} \frac{(N(\tau, \xi \mid t^\prime) - N(\tau, \xi \mid t)) v^\prime(\xi)}
                {(T - \xi)^\sigma} \, \rd \xi
                - \frac{t^\prime - t}{(T - t)^{\sigma}} \int_{t}^{T} \frac{\partial_t N(\tau, \xi \mid t) v(\xi)}
                {(T - \xi)^\sigma} \, \rd \xi \biggr\|_{\mathbb{R}^k} \\
                & \quad + \biggl\| \int_{t^\prime}^{T} \frac{N(\tau, \xi \mid t)}{(T - \xi)^\sigma}
                \biggl( v^\prime(\xi) - v(\xi)
                - \frac{\partial_t^\alpha v(\xi) (t^\prime - t)
                + \nabla^\alpha v(\xi) z(t^\prime)}{(T - t)^\beta} \biggr) \, \rd \xi \biggr\|_{\mathbb{R}^k} \\
                & \quad + \frac{1}{(T - t)^\beta} \biggl\| \int_{t}^{t^\prime} \frac{N(\tau, \xi \mid t)
                \bigl( \partial_t^\alpha v(\xi) (t^\prime - t) + \nabla^\alpha v(\xi) z(t^\prime) \bigr)}{(T - \xi)^\sigma} \, \rd \xi \biggr\|_{\mathbb{R}^k} \\
                & \quad + \biggl\| \int_{t}^{t^\prime} \frac{N(\tau, \xi \mid t) v(\xi)}{(T - \xi)^\sigma} \, \rd \xi
                - \frac{(t^\prime - t) N(\tau, t \mid t) v(t)}{(T - t)^\sigma} \biggr\|_{\mathbb{R}^k} \\
                & \doteq \mathscr{B}_1 + \mathscr{B}_2 + \mathscr{B}_3 + \mathscr{B}_4.
            \end{split}
        \end{equation*}
        We obtain
        \begin{equation*}
            \begin{split}
                \mathscr{B}_1
                & \leq \biggl\| \int_{t^\prime}^{T} \biggl( N(\tau, \xi \mid t^\prime) - N(\tau, \xi \mid t)
                - \frac{(t^\prime - t) \partial_t N(\tau, \xi \mid t)}{(T - t)^\sigma} \biggr)
                \frac{v^\prime(\xi)}{(T - \xi)^\sigma} \, \rd \xi \biggr\|_{\mathbb{R}^k} \\
                & \quad + \frac{t^\prime - t}{(T - t)^\sigma}
                \biggl\| \int_{t^\prime}^{T} \frac{\partial_t N(\tau, \xi \mid t) (v^\prime(\xi) - v(\xi))}
                {(T - \xi)^\sigma} \, \rd \xi \biggr\|_{\mathbb{R}^k} \\
                & \quad + \frac{t^\prime - t}{(T - t)^\sigma}
                \biggl\| \int_{t}^{t^\prime} \frac{\partial_t N(\tau, \xi \mid t) v(\xi)}{(T - \xi)^\sigma} \, \rd \xi \biggr\|_{\mathbb{R}^k} \\
                & \leq \frac{T^{1 - \sigma} \mu_v \omega_N^\ast(t^\prime - t)}{1 - \sigma}
                + \frac{T^{1 - \sigma} \mu_{\partial_t N} \omega_v^{(2)}(\dist) (t^\prime - t)}{(1 - \sigma) (T - \vartheta)^\sigma}
                + \frac{\mu_{\partial_t N} \mu_v (t^\prime - t)^2}{(T - \vartheta)^{2 \sigma}},
            \end{split}
        \end{equation*}
        where, in accordance with the introduced notation (see \eqref{notation_dist_v} and \eqref{notation_w^prime}),
        \begin{equation} \label{dist_along_x}
            \dist
            \doteq \dist\bigl( (t^\prime, x_{t^\prime}(\cdot)), (t, w(\cdot)) \bigr)
        \end{equation}
        and $\omega_N^\ast(\cdot)$ is chosen by condition \eqref{omega_N^ast}.
        In view of \eqref{z_estimate} and \eqref{technical_beta}, we get
        \begin{equation*}
            \mathscr{B}_2
            \leq \frac{\mathrm{B}(1 - \sigma, 1 - \varkappa_v) T^{1 - \varkappa_v} \mu_N \omega_v^\ast(t^\prime - t)}{(T - \vartheta)^\sigma},
            \quad \mathscr{B}_3
            \leq \frac{\mu_N \mu_v (1 + \mu_x^\ast) (t^\prime - t)^2}{(T - \vartheta)^{\beta + \sigma}},
        \end{equation*}
        where the number $\mu_x^\ast$ is given by \eqref{mu_x^ast}.
        Finally, if $\sigma > 0$, we derive
        \begin{align*}
            \mathscr{B}_4
            & \leq \biggl\| \int_{t}^{t^\prime} \frac{(N(\tau, \xi \mid t) - N(\tau, t \mid t)) v(\xi)}
            {(T - \xi)^\sigma} \, \rd \xi \biggr\|_{\mathbb{R}^k} \\*
            & \quad + \biggl\| N(\tau, t \mid t)
            \int_{t}^{t^\prime} \frac{v(\xi) - v(t)}{(T - \xi)^\sigma} \, \rd \xi \biggr\|_{\mathbb{R}^k} \\*
            & \quad + \biggl\| N(\tau, t \mid t) v(t) \int_{t}^{t^\prime}
            \biggl( \frac{1}{(T - \xi)^\sigma} - \frac{1}{(T - t)^\sigma} \biggr) \, \rd \xi \biggr\|_{\mathbb{R}^k} \\
            & \leq \frac{\mu_v \omega_N(t^\prime - t) (t^\prime - t)}{(T - \vartheta)^\sigma}
            + \frac{\mu_N \omega_v^{(1)}(t^\prime - t) (t^\prime - t)}{(T - \vartheta)^\sigma}
            + \frac{\mu_N \mu_v (t^\prime - t)^{1 + \sigma}}{(T - \vartheta)^{2 \sigma}}.
        \end{align*}
        Otherwise, the same estimate for $\mathscr{B}_4$ holds but without the last term.
        Thus, noting that the value $\dist$ from \eqref{dist_along_x} tends to $0$ as $t^\prime \to t^+$, we conclude that condition (iii$_\star$) is satisfied.
        The proposition is proved.
    \end{proof}

    \begin{proof}[Proof of Lemma \ref{lemma_terminal_functional}]
        In order to prove that the functional $\varphi$ is continuous, it suffices to take a compact set $\mathcal{G}_\ast \subset \mathcal{G}$ and show that $\varphi$ is continuous on $\mathcal{G}_\ast$.
        By conditions (i$_\star$) and (ii$_\star$) for the families of functions $v_1(\cdot \mid \cdot, \cdot)$ and $v_2(\cdot \mid \cdot, \cdot)$, let us choose the numbers $\mu_{v_1}$ and $\mu_{v_2}$ and moduli of continuity $\omega_{v_1}^{(2)}(\cdot)$ and $\omega_{v_2}^{(2)}(\cdot)$ that correspond to $\mathcal{G}_\ast$.
        Then, for any $(t, w(\cdot))$, $(t^\prime, w^\prime(\cdot)) \in \mathcal{G}_\ast$, we have
        (see \eqref{notation_dist_v})
        \begin{equation*}
            \begin{split}
                & | \varphi(t^\prime, w^\prime(\cdot)) - \varphi(t, w(\cdot)) | \\
                & \leq \| v_1^\prime(T) - v_1(T) \|_{\mathbb{R}^k}
                \| v_2^\prime(T)\|_{\mathbb{R}^k}
                + \|v_1(T)\|_{\mathbb{R}^k}
                \|v_2^\prime(T) - v_2(T)\|_{\mathbb{R}^k} \\
                & \leq \mu_{v_2} \omega_{v_1}^{(2)}(\dist)
                + \mu_{v_1} \omega_{v_2}^{(2)}(\dist),
            \end{split}
        \end{equation*}
        which implies that $\varphi$ is continuous on $\mathcal{G}_\ast$.
        In a similar way, it can be verified that, for every compact set $\mathcal{G}_\ast \subset \mathcal{G}$, the mappings
        \begin{align*}
            \mathcal{G}_\ast \ni (t, w(\cdot)) & \mapsto
            \langle \partial_t^\alpha v_i(T \mid t, w(\cdot)), v_{3 - i}(T \mid t, w(\cdot)) \rangle_{\mathbb{R}^k}
            \in \mathbb{R}, \\
            \mathcal{G}_\ast \ni (t, w(\cdot)) & \mapsto
            \nabla^\alpha v_i(T \mid t, w(\cdot))^\top v_{3 - i}(T \mid t, w(\cdot)) \in \mathbb{R}^n
        \end{align*}
        for $i \in \{1, 2\}$ are continuous, which yields continuity of the mappings $\partial_t^\alpha \varphi \colon \mathcal{G}_0 \to \mathbb{R}$ and $\nabla^\alpha \varphi \colon \mathcal{G}_0 \to \mathbb{R}^n$ from \eqref{varphi_ci_derivative_t} and \eqref{varphi_ci_gradient}.

        It remains to show that $\varphi$ is $ci$-differentiable of order $\alpha$ at a point $(t, w(\cdot)) \in \mathcal{G}_0$ with $\partial_t^\alpha \varphi(t, w(\cdot))$ and $\nabla^\alpha \varphi(t, w(\cdot))$ given by \eqref{varphi_ci_derivative_t} and \eqref{varphi_ci_gradient}.
        Let $x(\cdot) \in \mathcal{X}(t, w(\cdot))$ and $\vartheta \in (t, T)$ be fixed.
        In accordance with condition (iii$_\star$) for $v_1(\cdot \mid \cdot, \cdot)$ and $v_2(\cdot \mid \cdot, \cdot)$, let us choose the corresponding numbers $\varkappa_{v_1}$ and $\varkappa_{v_2}$ and infinitesimal moduli of continuity $\omega_{v_1}^\ast(\cdot)$ and $\omega_{v_2}^\ast(\cdot)$.
        In addition, by conditions (i$_\star$) and (ii$_\star$), let us take the numbers $\mu_{v_1}$ and $\mu_{v_2}$ and modulus of continuity $\omega_{v_2}^{(2)}(\cdot)$ that correspond to the compact set $\mathcal{G}_\ast$ defined by \eqref{G_ast_by_x}.
        Then, for any $t^\prime \in (t, \vartheta]$, we derive (see \eqref{notation_w^prime})
        \begin{equation*}
            \begin{split}
                & |\varphi(t^\prime, x_{t^\prime}(\cdot)) - \varphi(t, w(\cdot))
                - \partial_t^\alpha \varphi(t, w(\cdot)) (t^\prime - t)
                - \langle \nabla^\alpha \varphi(t, w(\cdot)), z(t^\prime) \rangle_{\mathbb{R}^n}| \\
                & \leq \biggl| \langle v_1^\prime(T) - v_1(T), v_2^\prime(T) \rangle_{\mathbb{R}^k}
                - \frac{\langle \partial_t^\alpha v_1(T) (t^\prime - t)
                + \nabla^\alpha v_1(T) z(t^\prime), v_2(T) \rangle_{\mathbb{R}^k}}
                {(T - t)^{\beta_1 + \gamma_1}} \biggr| \\
                & \quad + \biggl| \langle v_1(T), v_2^\prime(T) - v_2(T) \rangle_{\mathbb{R}^k}
                - \frac{\langle \partial_t^\alpha v_2(T) (t^\prime - t)
                + \nabla^\alpha v_2(T) z(t^\prime), v_1(T) \rangle_{\mathbb{R}^k}}
                {(T - t)^{\beta_2 + \gamma_2}} \biggr| \\
                & \doteq \mathscr{A}_1 + \mathscr{A}_2,
            \end{split}
        \end{equation*}
        where $z(t^\prime)$ is given by \eqref{z}.
        Recalling estimate \eqref{z_estimate}, we get
        \begin{equation*}
            \begin{split}
                \mathscr{A}_1
                & \leq \biggl| \biggl\langle v_1^\prime(T) - v_1(T)
                - \frac{\partial_t^\alpha v_1(T) (t^\prime - t)
                + \nabla^\alpha v_1(T) z(t^\prime)}{(T - t)^{\beta_1 + \gamma_1}},
                v_2^\prime(T) \biggr\rangle_{\mathbb{R}^k} \biggr| \\
                & \quad + \biggl| \frac{\langle \partial_t^\alpha v_1(T) (t^\prime - t) + \nabla^\alpha v_1(T) z(t^\prime),
                v_2^\prime(T) - v_2(T) \rangle_{\mathbb{R}^k}}{(T - t)^{\beta_1 + \gamma_1}} \biggr| \\
                & \leq \frac{\mu_{v_2} \omega_{v_1}^\ast(t^\prime - t)}{(T - \vartheta)^{\varkappa_{v_1}}}
                + \frac{\mu_{v_1} (1 + \mu_x^\ast) \omega_{v_2}^{(2)}(\dist) (t^\prime - t)}{(T - \vartheta)^{\beta_1 + \gamma_1}},
            \end{split}
        \end{equation*}
        where $\mu_x^\ast$ is the number from \eqref{mu_x^ast}, and, similarly,
        \begin{equation*}
            \mathscr{A}_2
            \leq \frac{\mu_{v_1} \omega_{v_2}^\ast(t^\prime - t)}{(T - \vartheta)^{\varkappa_{v_2}}}.
        \end{equation*}
        As a result, since the value $\dist$ tends to $0$ as $t^\prime \to t^+$ (see
        \eqref{dist_along_x}), we obtain that $\varphi$ is $ci$-differentiable of order $\alpha$ at $(t, w(\cdot))$ and equalities \eqref{varphi_ci_derivative_t} and \eqref{varphi_ci_gradient} hold.
    \end{proof}

    \begin{proof}[Proof of Lemma \ref{lemma_integral_functional}]
        The proof follows the same lines as that of Lemma \ref{lemma_terminal_functional}.

        For $i \in \{1, 2\}$, according to conditions (i$_\star$) and (ii$_\star$), let us choose the number $\mu_{v_i}$ and moduli of continuity $\omega_{v_i}^{(1)}(\cdot)$ and $\omega_{v_i}^{(2)}(\cdot)$ that correspond to a compact set $\mathcal{G}_\ast \subset \mathcal{G}$.
        Then, for any $(t, w(\cdot))$, $(t^\prime, w^\prime(\cdot)) \in \mathcal{G}_\ast$ with $t^\prime \geq t$, we have
        \begin{equation*}
            \begin{split}
                & | \varphi(t^\prime, w^\prime(\cdot)) - \varphi(t, w(\cdot)) | \\
                & \leq \biggl| \int_{t^\prime}^{T} \frac{\langle v_1^\prime(\tau) - v_1(\tau), v_2^\prime(\tau) \rangle_{\mathbb{R}^k}}{(T - \tau)^\sigma} \, \rd \tau \biggr|
                + \biggl| \int_{t^\prime}^{T} \frac{\langle v_1(\tau), v_2^\prime(\tau) - v_2(\tau) \rangle_{\mathbb{R}^k}}{(T - \tau)^\sigma} \, \rd \tau \biggr| \\
                & \quad + \biggl| \int_{t}^{t^\prime} \frac{\langle v_1(\tau),
                v_2(\tau) \rangle_{\mathbb{R}^k}}{(T - \tau)^\sigma} \, \rd \tau \biggr| \\
                & \leq \frac{T^{1 - \sigma} \bigl( \mu_{v_2} \omega_{v_1}^{(2)}(\dist) + \mu_{v_1} \omega_{v_2}^{(2)}(\dist) \bigr)}{1 - \sigma}
                + \frac{\mu_{v_1} \mu_{v_2} \dist^{1 - \sigma}}{1 - \sigma},
            \end{split}
        \end{equation*}
        which implies continuity of the functional $\varphi$ on the set $\mathcal{G}_\ast$, and, therefore, on the whole space $\mathcal{G}$.
        Further, let us note that, for $i \in \{1, 2\}$, the functional
        \begin{equation*}
            \psi_i(t, w(\cdot))
            \doteq (T - t)^{\sigma \vee \gamma_i} \int_{t}^{T}
            \frac{\langle \partial_t^\alpha v_i(\tau \mid t, w(\cdot)), v_{3 - i}(\tau \mid t, w(\cdot)) \rangle_{\mathbb{R}^k}}
            {(T - \tau)^\sigma (\tau - t)^{\gamma_i}} \, \rd \tau
        \end{equation*}
        for all $(t, w(\cdot)) \in \mathcal{G}_\ast$ is continuous.
        Indeed, for any $(t, w(\cdot))$, $(t^\prime, w^\prime(\cdot)) \in \mathcal{G}_\ast$ with $t^\prime \geq t$, using formula \eqref{technical_beta} and Proposition \ref{proposition_technical_continuity}, we derive
        \begin{equation*}
            \begin{split}
                & | \psi_i(t^\prime, w^\prime(\cdot)) - \psi_i(t, w(\cdot)) | \\
                & \leq (T - t^\prime)^{\sigma \vee \gamma_i} \biggl| \int_{t^\prime}^{T}
                \frac{\langle \partial_t^\alpha v_i^\prime(\tau) - \partial_t^\alpha v_i(\tau), v_{3 - i}^\prime(\tau) \rangle_{\mathbb{R}^k}}
                {(T - \tau)^\sigma (\tau - t^\prime)^{\gamma_i}} \, \rd \tau \biggr| \\
                & \quad + (T - t^\prime)^{\sigma \vee \gamma_i} \biggl| \int_{t^\prime}^{T}
                \frac{\langle \partial_t^\alpha v_i(\tau), v_{3 - i}^\prime(\tau) - v_{3 - i}(\tau) \rangle_{\mathbb{R}^k}}
                {(T - \tau)^\sigma (\tau - t^\prime)^{\gamma_i}} \, \rd \tau \biggr| \\
                & \quad + \biggl| (T - t^\prime)^{\sigma \vee \gamma_i} \int_{t^\prime}^{T}
                \frac{\langle \partial_t^\alpha v_i(\tau), v_{3 - i}(\tau) \rangle_{\mathbb{R}^k}}
                {(T - \tau)^\sigma (\tau - t^\prime)^{\gamma_i}} \, \rd \tau \\
                & \quad - (T - t)^{\sigma \vee \gamma_i} \int_{t}^{T}
                \frac{\langle \partial_t^\alpha v_i(\tau), v_{3 - i}(\tau) \rangle_{\mathbb{R}^k}}
                {(T - \tau)^\sigma (\tau - t)^{\gamma_i}} \, \rd \tau \biggr| \\
                & \leq \mathrm{B}(1 - \sigma, 1 - \gamma_i) T^{1 + \sigma \vee \gamma_i - \sigma - \gamma_i}
                \bigl( \mu_{v_{3 - i}} \omega_{v_i}^{(2)}(\dist) + \mu_{v_i} \omega_{v_{3 - i}}^{(2)}(\dist) \bigr) \\
                & \quad + \mu_{v_i} \mu_{v_{3 - i}} \bar{\omega}_{\sigma, \gamma_i}^{(2)}(\dist).
            \end{split}
        \end{equation*}
        In addition, for any $(t, w(\cdot))$, $(t^\prime, w^\prime(\cdot)) \in \mathcal{G}_\ast$ with $t^\prime \geq t$, we get
        \begin{equation*}
            \begin{split}
                & | \langle v_1^\prime(t^\prime), v_2^\prime(t^\prime) \rangle_{\mathbb{R}^k}
                - \langle v_1(t), v_2(t) \rangle_{\mathbb{R}^k} | \\
                & \leq \|v_2^\prime(t^\prime)\|_{\mathbb{R}^k} \bigl( \|v_1^\prime (t^\prime) - v_1(t^\prime)\|_{\mathbb{R}^k}
                + \|v_1(t^\prime) - v_1(t)\|_{\mathbb{R}^k} \bigr) \\
                & \quad + \|v_1(t)\|_{\mathbb{R}^k} \bigl( \|v_2^\prime(t^\prime) - v_2(t^\prime)\|_{\mathbb{R}^k}
                + \|v_2(t^\prime) - v_2(t) \|_{\mathbb{R}^k} \bigr) \\
                & \leq \mu_{v_2} \bigl( \omega_{v_1}^{(2)}(\dist) + \omega_{v_1}^{(1)}(\dist) \bigr)
                + \mu_{v_1} \bigl( \omega_{v_2}^{(2)}(\dist) + \omega_{v_2}^{(1)}(\dist) \bigr).
            \end{split}
        \end{equation*}
        Thus, we obtain that $\partial_t^\alpha \varphi \colon \mathcal{G}_0 \to \mathbb{R}$ from \eqref{varphi_integral_ci_derivative_t} is continuous.
        Similarly to the proof of continuity of the functional $\psi_i$ above, it can be verified that the mappings
        \begin{equation*}
            \mathcal{G}_\ast \ni (t, w(\cdot))
            \mapsto (T - t)^{\sigma \vee \gamma_i} \int_{t}^{T} \frac{\nabla^\alpha v_i(\tau \mid t, w(\cdot))^\top v_{3 - i}(\tau \mid t, w(\cdot))}
            {(T - \tau)^\sigma (\tau - t)^{\gamma_i}} \, \rd \tau \in \mathbb{R}^n
        \end{equation*}
        for $i \in \{1, 2\}$ are continuous, which gives continuity of $\nabla^\alpha \varphi \colon \mathcal{G}_0 \to \mathbb{R}^n$ from
        \eqref{varphi_integral_ci_gradient}.

        Now, let us show that $\varphi$ is $ci$-differentiable of order $\alpha$ at a point $(t, w(\cdot)) \in \mathcal{G}_0$ with $\partial_t^\alpha \varphi(t, w(\cdot))$ and $\nabla^\alpha \varphi(t, w(\cdot))$ from \eqref{varphi_integral_ci_derivative_t} and \eqref{varphi_integral_ci_gradient}.
        Let $x(\cdot) \in \mathcal{X}(t, w(\cdot))$ and $\vartheta \in (t, T)$ be fixed.
        For $i \in \{1, 2\}$, let us choose the corresponding number $\varkappa_{v_i}$ and infinitesimal modulus of continuity $\omega_{v_i}^\ast(\cdot)$ according to condition (iii$_\ast$).
        In addition, for $i \in \{1, 2\}$, using conditions (i$_\star$) and (ii$_\star$), let us take the number $\mu_{v_i}$ and moduli of continuity $\omega_{v_i}^{(1)}(\cdot)$ and $\omega_{v_i}^{(2)}(\cdot)$ that correspond to the compact set $\mathcal{G}_\ast$ defined by \eqref{G_ast_by_x}.
        For any $t^\prime \in (t, \vartheta]$, we derive
        \begin{equation*}
            \begin{split}
                & | \varphi(t^\prime, x_{t^\prime}(\cdot)) - \varphi(t, w(\cdot))
                - \partial_t^\alpha \varphi(t, w(\cdot)) (t^\prime - t)
                - \langle \nabla^\alpha \varphi(t, w(\cdot)), z(t^\prime) \rangle_{\mathbb{R}^n} | \\
                & \leq \biggl| \int_{t^\prime}^{T} \frac{\langle v_1^\prime(\tau) - v_1(\tau),
                v_2^\prime(\tau) \rangle_{\mathbb{R}^k}}{(T - \tau)^\sigma} \, \rd \tau \\
                & \quad - \frac{1}{(T - t)^{\beta_1}} \int_{t}^{T}
                \frac{\langle \partial_t^\alpha v_1(\tau) (t^\prime - t) + \nabla^\alpha v_1(\tau) z(t^\prime), v_2(\tau) \rangle_{\mathbb{R}^k}}
                {(T - \tau)^\sigma (\tau - t)^{\gamma_1}} \, \rd \tau \biggr| \\
                & \quad + \biggl| \int_{t^\prime}^{T} \frac{\langle v_1(\tau),
                v_2^\prime(\tau) - v_2(\tau) \rangle_{\mathbb{R}^k}}{(T - \tau)^\sigma} \, \rd \tau \\
                & \quad - \frac{1}{(T - t)^{\beta_2}} \int_{t}^{T}
                \frac{\langle \partial_t^\alpha v_2(\tau) (t^\prime - t) + \nabla^\alpha v_2(\tau) z(t^\prime), v_1(\tau) \rangle_{\mathbb{R}^k}}
                {(T - \tau)^\sigma (\tau - t)^{\gamma_2}} \, \rd \tau \biggr| \\
                & \quad + \biggl| \int_{t}^{t^\prime} \frac{\langle v_1(\tau),
                v_2(\tau) \rangle_{\mathbb{R}^k}}{(T - \tau)^\sigma} \, \rd \tau
                - \frac{(t^\prime - t) \langle v_1(t), v_2(t) \rangle_{\mathbb{R}^k}}{(T - t)^\sigma} \biggr| \\
                & \doteq \mathscr{A}_1 + \mathscr{A}_2 + \mathscr{A}_3,
            \end{split}
        \end{equation*}
        where $z(t^\prime)$ is given by \eqref{z}.
        In view of \eqref{z_estimate} and \eqref{technical_beta}, we get
        \begin{equation*}
            \begin{split}
                \mathscr{A}_1
                & \leq \biggl| \int_{t^\prime}^{T} \biggl\langle v_1^\prime(\tau) - v_1(\tau)
                - \frac{\partial_t^\alpha v_1(\tau) (t^\prime - t) + \nabla^\alpha v_1(\tau) z(t^\prime)}{(T - t)^{\beta_1} (\tau - t)^{\gamma_1}},
                \frac{v_2^\prime(\tau)}{(T - \tau)^\sigma} \biggr\rangle_{\mathbb{R}^k} \, \rd \tau \biggr| \\
                & \quad + \frac{1}{(T - t)^{\beta_1}} \biggl| \int_{t^\prime}^{T}
                \frac{\langle \partial_t^\alpha v_1(\tau) (t^\prime - t) + \nabla^\alpha v_1(\tau) z(t^\prime),
                v_2^\prime(\tau) - v_2(\tau) \rangle_{\mathbb{R}^k}}
                {(T - \tau)^\sigma (\tau - t)^{\gamma_1}} \, \rd \tau \biggr| \\
                & \quad + \frac{1}{(T - t)^{\beta_1}} \biggl| \int_{t}^{t^\prime}
                \frac{\langle \partial_t^\alpha v_1(\tau) (t^\prime - t) + \nabla^\alpha v_1(\tau) z(t^\prime),
                v_2(\tau) \rangle_{\mathbb{R}^k}}
                {(T - \tau)^\sigma (\tau - t)^{\gamma_1}} \, \rd \tau \biggr| \\
                & \leq \frac{\mathrm{B}(1 - \sigma, 1 - \varkappa_{v_1}) T^{1 - \varkappa_{v_1}} \mu_{v_2} \omega_{v_1}^\ast(t^\prime - t)}
                {(T - \vartheta)^\sigma} \\
                & \quad + \frac{\mathrm{B}(1 - \sigma, 1 - \gamma_1) T^{1 - \gamma_1} \mu_{v_1} (1 + \mu_x^\ast) \omega_{v_2}^{(2)}(\dist) (t^\prime - t)}{(T - \vartheta)^{\beta_1 + \sigma}} \\
                & \quad + \frac{\mu_{v_1} (1 + \mu_x^\ast) \mu_{v_2} (t^\prime - t)^{2 - \gamma_1}}{(1 - \gamma_1)(T - \vartheta)^{\beta_1 + \sigma}},
            \end{split}
        \end{equation*}
        where $\mu_x^\ast$ is the number from \eqref{mu_x^ast}, and, similarly,
        \begin{equation*}
            \mathscr{A}_2
            \leq \frac{\mathrm{B}(1 - \sigma, 1 - \varkappa_{v_2}) T^{1 - \varkappa_{v_2}} \mu_{v_1} \omega_{v_2}^\ast(t^\prime - t)}
            {(T - \vartheta)^\sigma}
            + \frac{\mu_{v_2} (1 + \mu_x^\ast) \mu_{v_1} (t^\prime - t)^{2 - \gamma_2}}{(1 - \gamma_2)(T - \vartheta)^{\beta_2 + \sigma}}.
        \end{equation*}
        Finally, if $\sigma > 0$, we have
        \begin{equation*}
            \begin{split}
                \mathscr{A}_3
                & \leq \biggl|\int_{t}^{t^\prime} \frac{\langle v_1(\tau) - v_1(t),
                v_2(\tau) \rangle_{\mathbb{R}^k}}{(T - \tau)^\sigma} \, \rd \tau \biggr|
                + \biggl| \int_{t}^{t^\prime} \frac{\langle v_1(t),
                v_2(\tau) - v_2(t) \rangle_{\mathbb{R}^k}}{(T - \tau)^\sigma} \, \rd \tau \biggr| \\
                & \quad + \biggl| \langle v_1(t), v_2(t) \rangle_{\mathbb{R}^k}
                \int_{t}^{t^\prime} \biggl( \frac{1}{(T - \tau)^\sigma} - \frac{1}{(T - t)^\sigma} \biggr) \, \rd \tau \biggr| \\
                & \leq \frac{\bigl( \mu_{v_2} \omega_{v_1}^{(1)}(t^\prime - t) + \mu_{v_1} \omega_{v_2}^{(1)}(t^\prime - t) \bigr)
                (t^\prime - t)}{(T - \vartheta)^\sigma}
                + \frac{\mu_{v_1} \mu_{v_2} (t^\prime - t)^{1 + \sigma}}{(T - \vartheta)^{2 \sigma}}.
            \end{split}
        \end{equation*}
        Otherwise, the same estimate for $\mathscr{A}_3$ holds but without the last term.
        Consequently, $\varphi$ is $ci$-differentiable of order $\alpha$ at $(t, w(\cdot))$ and equalities \eqref{varphi_integral_ci_derivative_t} and \eqref{varphi_integral_ci_gradient} hold.
    \end{proof}

\section{Auxiliary statements}
\label{appendix_3}

    We start by noting that, for any $\beta < 1$, any $\gamma < 1$, and any $t$, $t^\prime \in [0, T]$ with $t^\prime > t$, the equality
    \begin{equation} \label{technical_beta}
        \int_{t}^{t^\prime} \frac{\rd \tau}{(t^\prime - \tau)^\beta (\tau - t)^\gamma}
        = (t^\prime - t)^{1 - \beta - \gamma} \mathrm{B}(1 - \beta, 1 - \gamma)
    \end{equation}
    is valid (see, e.g., \cite[Theorem D.6]{Diethelm_2010}), where $\mathrm{B}(\cdot, \cdot)$ is the beta-function.

    \begin{proposition} \label{proposition_technical_continuity_0}
        For any $\beta \in [0, 1)$, there exists a modulus of continuity $\bar{\omega}_\beta^{(1)} (\cdot)$ such that, for any $t$, $t^\prime \in [0, T]$ with $t^\prime \geq t$, any $k \in \mathbb{N}$, and any $v(\cdot) \in \C([t, t^\prime], \mathbb{R}^k)$, the functions
        \begin{equation*}
            \hat{v}(\tau)
            \doteq \int_{t}^{\tau} \frac{v(\xi)}{(\tau - \xi)^\beta} \, \rd \xi,
            \quad \hat{v}^\prime(\tau)
            \doteq \int_{\tau}^{t^\prime} \frac{v(\xi)}{(\xi - \tau)^\beta} \, \rd \xi
            \quad \forall \tau \in [t, t^\prime]
        \end{equation*}
        satisfy the estimate
        \begin{equation*}
            \|\hat{v}(\tau^\prime) - \hat{v}(\tau)\|_{\mathbb{R}^k}
            \vee \|\hat{v}^\prime(\tau^\prime) - \hat{v}^\prime(\tau)\|_{\mathbb{R}^k}
            \leq \|v(\cdot)\|_{\C([t, t^\prime], \mathbb{R}^k)} \bar{\omega}_\beta^{(1)}(|\tau^\prime - \tau|)
            \ \ \forall \tau, \tau^\prime \in [t, t^\prime].
        \end{equation*}
    \end{proposition}

    The proof of this proposition can be carried out by the scheme from, e.g., \cite[Theorem 3.6 and Remark 3.3]{Samko_Kilbas_Marichev_1993} and \cite[Theorem 2.6]{Diethelm_2010} (see also \cite[Proposition 2.1]{Gomoyunov_2019_FCAA_2}).

    \begin{proposition} \label{proposition_technical_continuity}
        For any $\beta$, $\gamma \in [0, 1)$, there exists a modulus of continuity $\bar{\omega}_{\beta, \gamma}^{(2)} (\cdot)$ such that, for any $t$, $t^\prime \in [0, T]$ with $t^\prime \geq t$, any $k \in \mathbb{N}$, and any $v(\cdot) \in \C([t, t^\prime], \mathbb{R}^k)$, the functions
        \begin{align*}
            \hat{v}(\tau)
            & \doteq (\tau - t)^{\beta \vee \gamma} \int_{t}^{\tau} \frac{v(\xi)}{(\tau - \xi)^\beta (\xi - t)^\gamma} \, \rd \xi, \\
            \hat{v}^\prime(\tau)
            & \doteq (t^\prime - \tau)^{\beta \vee \gamma} \int_{\tau}^{t^\prime} \frac{v(\xi)}{(t^\prime - \xi)^\beta (\xi - \tau)^\gamma} \, \rd \xi
            \quad \forall \tau \in [t, t^\prime]
        \end{align*}
        satisfy the estimate
        \begin{equation*}
            \|\hat{v}(\tau^\prime) - \hat{v}(\tau)\|_{\mathbb{R}^k}
            \vee \|\hat{v}^\prime(\tau^\prime) - \hat{v}^\prime(\tau)\|_{\mathbb{R}^k}
            \leq \|v(\cdot)\|_{\C([t, t^\prime], \mathbb{R}^k)} \bar{\omega}_{\beta, \gamma}^{(2)}(|\tau^\prime - \tau|)
            \ \ \forall \tau, \tau^\prime \in [t, t^\prime].
        \end{equation*}
    \end{proposition}
    \begin{proof}
        Let $t$, $t^\prime \in [0, T]$ with $t^\prime \geq t$, $k \in \mathbb{N}$, $v(\cdot) \in \C([t, t^\prime], \mathbb{R}^k)$, and $\tau$, $\tau^\prime \in [t, t^\prime]$ with $\tau^\prime \geq \tau$ be fixed.
        Then, denoting $\varkappa \doteq \beta \vee \gamma$ for brevity, we obtain
        \begin{align*}
            & \|\hat{v}(\tau^\prime) - \hat{v}(\tau)\|_{\mathbb{R}^k} \\*
            & = \biggl\| (\tau^\prime - t)^\varkappa \int_{t}^{\tau^\prime}
            \frac{(\tau^\prime - \xi)^{\varkappa - \beta} (\xi - t)^{\varkappa - \gamma}  v(\xi)}
            {(\tau^\prime - \xi)^\varkappa (\xi - t)^\varkappa} \, \rd \xi \\*
            & \quad - (\tau - t)^\varkappa \int_{t}^{\tau}
            \frac{(\tau - \xi)^{\varkappa - \beta} (\xi - t)^{\varkappa - \gamma} v(\xi)}
            {(\tau - \xi)^\varkappa (\xi - t)^\varkappa} \, \rd \xi \biggr\|_{\mathbb{R}^k} \\
            & \leq \biggl\| (\tau^\prime - t)^\varkappa \int_{t}^{\tau^\prime}
            \frac{(\tau^\prime - \xi)^{\varkappa - \beta} (\xi - t)^{\varkappa - \gamma}  v(\xi)}
            {(\tau^\prime - \xi)^\varkappa (\xi - t)^\varkappa} \, \rd \xi \\*
            & \quad - (\tau - t)^\varkappa \int_{t}^{\tau}
            \frac{(\tau^\prime - \xi)^{\varkappa - \beta} (\xi - t)^{\varkappa - \gamma} v(\xi)}
            {(\tau - \xi)^\varkappa (\xi - t)^\varkappa} \, \rd \xi \biggr\|_{\mathbb{R}^k} \\*
            & \quad + (\tau - t)^\varkappa \biggl\| \int_{t}^{\tau}
            \frac{((\tau^\prime - \xi)^{\varkappa - \beta} - (\tau - \xi)^{\varkappa - \beta})
            (\xi - t)^{\varkappa - \gamma} v(\xi)}
            {(\tau - \xi)^\varkappa (\xi - t)^\varkappa} \, \rd \xi \biggr\|_{\mathbb{R}^k} \\*
            & \doteq \mathscr{A}_1 + \mathscr{A}_2.
        \end{align*}
        Taking \cite[Corollary 2.1]{Gomoyunov_2019_FCAA_2} into account, we get
        \begin{equation*}
            \begin{split}
                \mathscr{A}_1
                & \leq \lambda_\varkappa
                \max_{\xi \in [t, \tau^\prime]} \|(\tau^\prime - \xi)^{\varkappa - \beta} (\xi - t)^{\varkappa - \gamma}  v(\xi)\|_{\mathbb{R}^k}
                (\tau^\prime - \tau)^{1 - \varkappa} \\
                & \leq T^{2 \varkappa - \beta - \gamma} \lambda_\varkappa \|v(\cdot)\|_{\C([t, t^\prime], \mathbb{R}^k)} (\tau^\prime - \tau)^{1 - \varkappa}
            \end{split}
        \end{equation*}
        with a number $\lambda_\varkappa \geq 0$ determined by the number $\varkappa$ only.
        Further, if $\varkappa = \beta$, we have $\mathscr{A}_2 = 0$.
        Otherwise, using \eqref{technical_beta}, we derive
        \begin{equation*}
            \begin{split}
                \mathscr{A}_2
                & \leq (\tau - t)^{2 \varkappa - \gamma} (\tau^\prime - \tau)^{\varkappa - \beta} \|v(\cdot)\|_{\C([t, t^\prime], \mathbb{R}^k)}
                \int_{t}^{\tau} \frac{\rd \xi}{(\tau - \xi)^\varkappa (\xi - t)^\varkappa} \\
                & \leq \mathrm{B}(1 - \varkappa, 1 - \varkappa) T^{1 - \gamma} \|v(\cdot)\|_{\C([t, t^\prime], \mathbb{R}^k)}
                (\tau^\prime - \tau)^{\varkappa - \beta}.
            \end{split}
        \end{equation*}
        The obtained estimates imply the required result for the function $\hat{v}(\cdot)$.
        The result for the function $\hat{v}^\prime(\cdot)$ can be verified in a similar way or by performing a change of variables (see, e.g., \cite[equation (2.19)]{Samko_Kilbas_Marichev_1993}).
        The proposition is proved.
    \end{proof}

    \begin{proposition} \label{proposition_technical_continuity_2}
        For any $\beta \in [0, 1 / 2)$, there exists a modulus of continuity $\bar{\omega}_\beta^{(3)}(\cdot)$ such that, for any $k \in \mathbb{N}$ and any $v(\cdot) \in \C([0, T], \mathbb{R}^k)$, the function
        \begin{equation*}
            \hat{v} (\tau, \xi)
            \doteq \int_{\tau \vee \xi}^{T} \frac{v(\eta)}{(\eta - \tau)^\beta (\eta - \xi)^\beta} \, \rd \eta
            \quad \forall \tau, \xi \in [0, T]
        \end{equation*}
        satisfies the estimate
        \begin{equation*}
            \|\hat{v}(\tau^\prime, \xi^\prime) - \hat{v}(\tau, \xi)\|_{\mathbb{R}^k}
            \leq \|v(\cdot)\|_{\C([0, T], \mathbb{R}^k)} \bar{\omega}_\beta^{(3)}(|\tau^\prime - \tau| + |\xi^\prime - \xi|)
            \quad \forall \tau, \tau^\prime, \xi, \xi^\prime \in [0, T].
        \end{equation*}
    \end{proposition}
    \begin{proof}
        Let us introduce the auxiliary function
        \begin{equation*}
            p(\tau, \xi)
            \doteq \int_{\tau}^{T} \frac{\rd \eta}{(\eta - \tau)^\beta (\eta - \xi)^\beta}
            \quad \forall (\tau, \xi) \in \Omega,
        \end{equation*}
        where the set $\Omega$ is taken from \eqref{Omega}, and show that this function is continuous.
        Let a point $(\tau_0, \xi_0) \in \Omega$ and a sequence $\{(\tau_i, \xi_i)\}_{i \in \mathbb{N}} \subset \Omega$ be such that $(\tau_i, \xi_i) \to (\tau_0, \xi_0)$ as $i \to \infty$.
        We first consider the case when $\tau_0 < T$, and, then, we can assume that $\tau_i < T$ for all $i \in \mathbb{N}$.
        By performing a change of variables, we get
        \begin{equation*}
            p(\tau_i, \xi_i)
            = (T - \tau_i)^{1 - 2 \beta} \int_{0}^{1} \frac{\rd \zeta}{\zeta^\beta (\frac{\tau_i - \xi_i}{T - \tau_i} + \zeta)^\beta}
            \quad \forall i \in \mathbb{N} \cup \{0\}.
        \end{equation*}
        It is clear that $(T - \tau_i)^{1 - 2 \beta} \to (T - \tau_0)^{1 - 2 \beta}$ as $i \to \infty$.
        In addition, it can be verified by applying the Lebesgue dominated convergence theorem that
        \begin{equation*}
            \int_{0}^{1} \frac{\rd \zeta}{\zeta^\beta \bigl( \frac{\tau_i - \xi_i}{T - \tau_i} + \zeta \bigr)^\beta}
            \to \int_{0}^{1} \frac{\rd \zeta}{\zeta^\beta \bigl( \frac{\tau_0 - \xi_0}{T - \tau_0} + \zeta \bigr)^\beta}
            \quad \text{as } i \to \infty.
        \end{equation*}
        Hence, we conclude that $p(\tau_i, \xi_i) \to p(\tau_0, \xi_0)$ as $i \to \infty$.
        Now, let us suppose that $\tau_0 = T$.
        In this case, we have $p(\tau_0, \xi_0) = 0$ by definition.
        For every $i \in \mathbb{N}$, recalling that $\xi_i \leq \tau_i$, we derive
        \begin{equation*}
            0
            \leq p(\tau_i, \xi_i)
            \leq \int_{\tau_i}^{T} \frac{\rd \eta}{(\eta - \tau_i)^{2 \beta}}
            = \frac{(T - \tau_i)^{1 - 2 \beta}}{1 - 2 \beta},
        \end{equation*}
        and, thus, $p(\tau_i, \xi_i) \to 0$ as $i \to \infty$, which completes the proof of continuity of the function $p(\cdot, \cdot)$.
        Let $\omega_p(\cdot)$ be the modulus of continuity of this function.

        Let $k \in \mathbb{N}$, $v(\cdot) \in \C([0, T], \mathbb{R}^k)$, and $\tau$, $\tau^\prime$, $\xi$, $\xi^\prime \in [0, T]$ be fixed.
        Denote
        \begin{equation*}
            \mu_v
            \doteq \|v(\cdot)\|_{\C([0, T], \mathbb{R}^k)},
            \quad \vartheta_\vee
            \doteq \tau \vee \xi,
            \quad \vartheta_\wedge
            \doteq \tau \wedge \xi,
            \quad \vartheta_\vee^\prime
            \doteq \tau^\prime \vee \xi^\prime,
            \quad \vartheta_\wedge^\prime
            \doteq \tau^\prime \wedge \xi^\prime
        \end{equation*}
        for brevity and suppose that $\vartheta_\vee^\prime \geq \vartheta_\vee$ for definiteness.
        We have
        \begin{equation*}
            \begin{split}
                & \|\hat{v}(\tau^\prime, \xi^\prime) - \hat{v}(\tau, \xi)\|_{\mathbb{R}^k} \\
                & = \biggl\| \int_{\vartheta_\vee^\prime}^{T} \frac{v(\eta)}{(\eta - \vartheta_\vee^\prime)^\beta (\eta - \vartheta_\wedge^\prime)^\beta} \, \rd \eta
                - \int_{\vartheta_\vee}^{T} \frac{v(\eta)}{(\eta - \vartheta_\vee)^\beta (\eta - \vartheta_\wedge)^\beta} \, \rd \eta \biggr\|_{\mathbb{R}^k} \\
                & \leq \biggl\| \int_{\vartheta_\vee^\prime}^{T} \frac{v(\eta)}{(\eta - \vartheta_\vee^\prime)^\beta (\eta - \vartheta_\wedge^\prime)^\beta} \, \rd \eta
                - \int_{\vartheta_\vee^\prime}^{T} \frac{v(\eta)}{(\eta - \vartheta_\vee^\prime)^\beta (\eta - \vartheta_\wedge)^\beta} \, \rd \eta \biggr\|_{\mathbb{R}^k} \\
                & \quad + \biggl\| \int_{\vartheta_\vee^\prime}^{T} \frac{v(\eta)}{(\eta - \vartheta_\vee^\prime)^\beta (\eta - \vartheta_\wedge)^\beta} \, \rd \eta
                - \int_{\vartheta_\vee^\prime}^{T} \frac{v(\eta)}{(\eta - \vartheta_\vee)^\beta (\eta - \vartheta_\wedge)^\beta} \, \rd \eta \biggr\|_{\mathbb{R}^k} \\
                & \quad + \biggl\| \int_{\vartheta_\vee}^{\vartheta_\vee^\prime}
                \frac{v(\eta)}{(\eta - \vartheta_\vee)^\beta (\eta - \vartheta_\wedge)^\beta} \, \rd \eta \biggr\|_{\mathbb{R}^k} \\
                & \doteq \mathscr{A}_1 + \mathscr{A}_2 + \mathscr{A}_3.
            \end{split}
        \end{equation*}
        We derive
        \begin{equation*}
            \begin{split}
                \mathscr{A}_1
                & \leq \mu_v
                \int_{\vartheta_\vee^\prime}^{T} \frac{1}{(\eta - \vartheta_\vee^\prime)^\beta} \biggl| \frac{1}{(\eta - \vartheta_\wedge^\prime)^\beta}
                - \frac{1}{(\eta - \vartheta_\wedge)^\beta} \biggr| \, \rd \eta \\
                & = \mu_v
                \biggl| \int_{\vartheta_\vee^\prime}^{T} \frac{\rd \eta}{(\eta - \vartheta_\vee^\prime)^\beta (\eta - \vartheta_\wedge^\prime)^\beta}
                - \int_{\vartheta_\vee^\prime}^{T} \frac{\rd \eta}{(\eta - \vartheta_\vee^\prime)^\beta (\eta - \vartheta_\wedge)^\beta} \biggr| \\
                & \leq \mu_v \omega_p(|\vartheta_\wedge^\prime - \vartheta_\wedge|)
            \end{split}
        \end{equation*}
        and
        \begin{equation*}
            \begin{split}
                \mathscr{A}_2
                & \leq \mu_v \int_{\vartheta_\vee^\prime}^{T} \frac{1}{(\eta - \vartheta_\wedge)^\beta}
                \biggl| \frac{1}{(\eta - \vartheta_\vee^\prime)^\beta} - \frac{1}{(\eta - \vartheta_\vee)^\beta} \biggr| \, \rd \eta \\
                & = \mu_v \biggl( \int_{\vartheta_\vee^\prime}^{T} \frac{\rd \eta}{(\eta - \vartheta_\vee^\prime)^\beta (\eta - \vartheta_\wedge)^\beta}
                - \int_{\vartheta_\vee^\prime}^{T} \frac{\rd \eta}{(\eta - \vartheta_\vee)^\beta (\eta - \vartheta_\wedge)^\beta} \biggr) \\
                & \leq \mu_v \biggl( \omega_p(\vartheta_\vee^\prime - \vartheta_\vee)
                + \int_{\vartheta_\vee}^{\vartheta_\vee^\prime} \frac{\rd \eta}{(\eta - \vartheta_\vee)^{2 \beta}} \biggr) \\
                & = \mu_v \biggl( \omega_p(\vartheta_\vee^\prime - \vartheta_\vee)
                + \frac{(\vartheta_\vee^\prime - \vartheta_\vee)^{1 - 2 \beta}}{1 - 2 \beta} \biggr).
            \end{split}
        \end{equation*}
        Similarly, we get
        \begin{equation*}
            \mathscr{A}_3
            \leq \frac{\mu_v (\vartheta_\vee^\prime - \vartheta_\vee)^{1 - 2 \beta}}{1 - 2 \beta}.
        \end{equation*}
        Thus, noting that $(\vartheta_\vee^\prime - \vartheta_\vee) \vee  | \vartheta_\wedge^\prime - \vartheta_\wedge | \leq |\tau^\prime - \tau | + | \xi^\prime - \xi|$, the result follows.
        The proposition is proved.
     \end{proof}

     \begin{remark} \label{remark_appendix}
        Propositions \ref{proposition_technical_continuity_0}, \ref{proposition_technical_continuity}, and \ref{proposition_technical_continuity_2} remain valid with the same moduli of continuity $\bar{\omega}_\beta^{(1)}(\cdot)$, $\bar{\omega}_{\beta, \gamma}^{(2)}(\cdot)$, and $\bar{\omega}_\beta^{(3)}(\cdot)$ if, instead of functions with values in $\mathbb{R}^k$, we take functions with values in $\mathbb{R}^{k \times \hat{k}}$ for some number $\hat{k} \in \mathbb{N}$.
    \end{remark}

\medskip

Received xxxx 20xx; revised xxxx 20xx; early access xxxx 20xx.

\medskip

\end{document}